\documentclass[11pt]{amsart}
\usepackage{amssymb,amsmath}
\usepackage{eucal}
\usepackage{epstopdf}
\usepackage{mathrsfs}
\usepackage{enumerate}
\usepackage{fullpage} 
\usepackage{setspace}
\usepackage{xcolor}
\usepackage{dsfont}
\usepackage{url}

\usepackage{hyperref}

\usepackage{Fourier}

\newtheorem{theorem}{Theorem}[section]

\newtheorem{lemma}[theorem]{Lemma}
\newtheorem{proposition}[theorem]{Proposition}

\newtheorem*{conjecture*}{Conjecture}
\newtheorem*{theorem*}{Theorem}
\newtheorem*{lemma*}{Lemma}
\numberwithin{equation}{section}

\DeclareMathOperator{\I}{I}
\DeclareMathOperator{\Cov}{Cov}
\DeclareMathOperator{\Var}{Var}

\theoremstyle{remark}
\newtheorem{remark}{Remark}

\usepackage{mathtools}

\DeclarePairedDelimiter\floor{\lfloor}{\rfloor}

\renewcommand{\Re}{{\mathfrak{Re}}}

\newcommand{\mi}{\mathrm{i}}
\newcommand{\oh}{\mbox{$\frac{1}{2}$}}

\begin{document}

\begin{center}


\title{On rates of convergence in central limit theorems \\ of Selberg and Bourgade}

\author{Po-Han Hsu}

\address[Po-Han Hsu]{Department of Applied Mathematics\\
National Sun Yat-Sen University\\
Kaohsiung City, Taiwan}
\email{phsu2@math.nsysu.edu.tw}

\author{Peng-Jie Wong}

\address[Peng-Jie Wong]{Department of Applied Mathematics\\
National Sun Yat-Sen University\\
Kaohsiung City, Taiwan}
\email{pjwong@math.nsysu.edu.tw}

\subjclass[2000]{Primary 11M06; Secondary 11M41, 60F05}

 
\keywords{Dirichlet $L$-function, central limit theorem, rate of convergence}

\dedicatory{Dedicated to the memory of our teacher, Professor Chao-Liang Shen, with gratitude and admiration}

\thanks{This research was partially supported by the NSTC grant 114-2628-M-110-006-MY4 of PJW and the NSTC grant 113-2115-M-110-001-MY2 and 114-2115-M-110-004-MY3 of PHH}


\begin{abstract}
Based on the recent works of {Radziwi\l\l}-Soundararajan and Roberts, we establish a rate of convergence in Bourgade's central limit theorem for shifted Dirichlet $L$-functions. Our results also indicate that the dependence structure in the components of a random vector could have a dramatic impact on the rate of convergence in such a  multivariate central limit theorem.
\end{abstract}

\maketitle
\end{center}

\section{Introduction}

For a Dirichlet character $\chi$, let $L(s,\chi)$ stand for the Dirichlet $L$-function of $\chi$.  Let $N\in\Bbb{N}$. Given Dirichlet characters $\chi_j$ (mod $q_j$) and real-valued functions $\alpha_j=\alpha_j(T)$ of $T$ such that $|\alpha_j|\leq 0.5 T$, we consider 
\begin{align}\label{X_{T}}
{\bf X}_{T} = {\bf X}_{T}(U) = (\mathcal{X}_{\alpha_{1}, \chi_{1}, T}(U), \ldots, \mathcal{X}_{\alpha_{N}, \chi_{N}, T}(U)) ,
\end{align}
where $U = U_{T}$ is a uniform distribution on $[T, 2T]$,  and
 \begin{align}\label{def-XjT-chij}
\mathcal{X}_{\alpha_j, \chi_j, T}(U) = (\log|L(\oh +\mi(U+\alpha_{j}),\chi_j)|) / \sqrt{\oh\log\log T}.
\end{align}
The celebrated Selberg's central limit theorem \cite{Se46, Se92} asserts that each $\mathcal{X}_{\alpha_j, \chi_j, T}(U)$ converges to a standard normal random variable in distribution (see also \cite{HW,RS}). More recently, in \cite{HW-Forum}, the present authors showed that ${\bf X}_{T}(U)$ converges, in distribution, to an $N$-variate normal distribution under some appropriate conditions on $\alpha_j$. Such a convergence presents a generalisation of Bourgade's central limit theorem on shifted Riemann zeta functions (see \cite{Bourgade}).

The main objective of this article is to estimate the rate of convergence for such a central limit theorem. To state our main theorem, we recall some background knowledge. For any function $f:\mathbb{R}^{N}\rightarrow\mathbb{R}$, we denote the $p$-norm of $f$ by $\|f\|_{p}$, and we  call $f$ a Lipschitz (continuous) function if there exists $L>0$ such that $|f(X) - f(Y)|\leq L \|X - Y\|_{2}$ for every $X,Y\in \mathbb{R}^{N}$. We define $\|f\|_{Lip}$ to be the smallest possible $L$ in the inequality. We shall consider
\begin{align}\label{the space LLM}
\mathcal{L}_{L,M}:=\{f:\mathbb{R}^{N}\rightarrow\mathbb{R}\mid f \ \text{is Lipschitz with} \ \|f\|_{Lip}\leq L \mbox{ and } \|f\|_{\infty}\leq M\}
\end{align}
and note that 
$
\cup_{L,M\in \Bbb{R}^+} \mathcal{L}_{L,M}
$
forms the space of all bounded Lipschitz functions on $ \mathbb{R}^{N}$. In addition, for $\mathbb{R}^{N}$-valued random vectors ${\bf X}$ and ${\bf Y}$, the (generalised) Dudley distance between ${\bf X}$ and ${\bf Y}$ is defined as
\begin{align}\label{Dudley distance}
d_{\mathcal{D},L,M}({\bf X}, {\bf Y}):=\sup_{f\in\mathcal{L}_{L,M}}|\mathbb{E}[f({\bf X})] - \mathbb{E}[f({\bf Y})]|.
\end{align}

Throughout our discussion, we require real-valued functions  $\alpha_{j}=\alpha_{j}(T)$ satisfy $|\alpha_j|\leq 0.5 T$ and $ |\alpha_{i} - \alpha_{j}| = d_{i,j}(T) \le \Delta(T)$ with non-negative $\Delta(T)\ll(\log\log T)^{\varepsilon}$ for a fixed 
$0<\varepsilon<\frac{2}{3}$. We also require $\Delta(T)$ satisfies
\begin{align*}
O_{q_i,q_j}(|\alpha_{i} - \alpha_{j}| +1) + O \Big(\frac{(\log\log\log T)^{2}}{\log\log T}\Big)
= O_1{(\Delta(T))},
\end{align*}
where the dependence on $q_i,q_j$ is the same as in \eqref{distance-cond}, and the subscript $1$ in $O$ indicates that the implied constant is $1$ (\`a la Harper \cite{Harper09}). Similar to \cite{HW-Forum}, we use the following $c_{i,j}$ to characterise the distance between shifts $\alpha_i$ and $\alpha_j$. Letting $\delta(T)$ be a non-negative function such that $\delta(T)\ll(\log\log T)^{\varepsilon}$, we always take $c_{i, i} =1$, and we define $c_{i,j}$ 
for $i\neq j$ as follows \\
 \noindent (1) We set $c_{i, j}=1$ if
	\begin{equation}\label{cond-1}
	\log\bigg( \frac{1}{|\alpha_{i} - \alpha_{j}|}\bigg) \geq\log\log T + e_{i,j}(T) \quad\text{with}\quad |e_{i,j}(T)|\le \delta(T);
	\end{equation}
\noindent (2) we define $c_{i,j}=c$ if there is a constant $c\in [0,1]$ (independent of $T$) such that
	\begin{equation}\label{cond-01}
	\log\bigg( \frac{1}{|\alpha_{i} - \alpha_{j}|}\bigg) = c\log\log T + e_{i,j}(T)  \quad\text{with}\quad |e_{i,j}(T)|\le \delta(T).
	\end{equation}
	We will always assume that each $(\alpha_i,\alpha_j)$ verifies either \eqref{cond-1} and \eqref{cond-01}. 
In addition, 
we set 
\begin{equation}\label{def-deltaij}
\delta_{i,j}
=\begin{cases} 
1& \mbox{if $\chi_i \bar{\chi}_j$ is principal;}\\ 
0 & \mbox{otherwise.}
\end{cases}
\end{equation}
With the above notation in hand, we introduce the symmetric matrix
\begin{align}\label{the covariance matrix}
\mathfrak{K} = (k_{i, j})_{1\leq i, j\leq N} 
\quad
\text{with}
\quad
k_{i, j}=
\begin{cases}
1\quad&\mbox{if}\quad i = j;\\
\delta_{i, j}c_{i, j}\quad&\mbox{if}\quad i\neq j.
\end{cases}
\end{align}
Note that when $\mathfrak{K}$ is positive-definite (or, equivalently,  $\sum_{1\leq i, j\leq N}a_{i}a_{j}\delta_{i, j}c_{i, j}\geq 0$ for any $a_{i}, a_{j} \in\Bbb{R}$), there exists an $N$-variate normal distribution 
\begin{align}\label{normal X}
\tilde{{\bf X}} = (\tilde{\mathcal{X}}_{1}, \ldots, \tilde{\mathcal{X}}_{N}), 
\end{align}
with mean ${\bf 0}_{N}$ and covariance matrix  $\mathfrak{K}$.

In the above notation, we are in a position to state our main theorem:
\begin{theorem}\label{main-thm}
Let ${\bf X}_{T} = {\bf X}_{T}(U)$ and $\tilde{{\bf X}}$ be defined as in \eqref{X_{T}} and \eqref{normal X}, respectively. Then for any fixed $0<\varepsilon_{1}<\varepsilon_{2}$, with $\varepsilon_{1} + \varepsilon_{2}<1$, and any  sufficiently large $T$,
\begin{align*}
d_{\mathcal{D}}({\bf X}_{T}, \tilde{{\bf X}})
\ll_{{\mathfrak{K}}, N }\frac{L}{(\log\log\log T)^{\varepsilon_{1}}}  +  M(\log\log\log T)^{N(\varepsilon_{1} + \varepsilon_{2})}\exp((-1/2)(\log\log\log T)^{\varepsilon_{1} + \varepsilon_{2}}).
\end{align*}
 Consequently, as $T\rightarrow\infty$, the vector ${\bf X}_{T}$ converges to $\tilde{{\bf X}}$ in distribution.
\end{theorem} 

\begin{remark}
The constants $L$ and $M$ in \eqref{the space LLM} may be taken as functions of $T$ as long as the convergence is not destroyed. In addition, taking  $L=o((\log\log\log T)^{\varepsilon_{1}})$ and  $M =o( \frac{\exp((1/2)(\log\log\log T)^{\varepsilon_{1} + \varepsilon_{2}}))}{(\log\log\log T)^{N(\varepsilon_{1} + \varepsilon_{2})}})$, by \eqref{the space LLM}, we derive that $\cup_{L,M\in \Bbb{R}^+} \mathcal{L}_{L,M} = \cup_{T\in\Bbb{R}^+}\mathcal{L}_{L(T), M(T)} = C_{b}(\mathbb{R}^{n})$, the space of all bounded continuous functions defined on $\mathbb{R}^{n}$, which recovers the usual test function space of convergence in distribution. So, Theorem \ref{main-thm} is valid for any given bounded continuous function when $T$ is sufficiently large.
\end{remark}

As may be noticed, simply substituting $\chi = \chi_{0}$, the trivial character, and $N = 1$ in Theorem \ref{main-thm} would \emph{not} lead the rate of of convergence of order $\ll (\log\log\log T)^{2}/\sqrt{\log\log T})$ for Selberg's central limit theorem established by Tsang \cite[Theorem 6.1]{Tsang thesis}. Nevertheless, there is no ``shifting'' in Tsang's result, so both $\Delta (T)$ and $\delta(T)$ are \emph{identically zero}. Observing that the estimates for these two quantities are central to Proposition \ref{Prop 6 in AR24}, we were prompted to choose a set of new parameters to prove the following result parallel to Tsang's work. It shall be worthwhile remarking that compared with Theorem \ref{main-thm}, we further establish a stronger rate of convergence of the present authors' previous work \cite{HW} on the independence among two or three distinct Dirichlet $L$-functions.

\begin{theorem}\label{indep-rate}
Let $N = 1, 2, 3$. Suppose that $\Delta(T) = 0$ and $\delta(T) = 0$ for all $T$. Then for any fixed $\varepsilon_{3},\varepsilon_{4}>0$ and any sufficiently large $T$,
\begin{align*}
d_{\mathcal{D}}({\bf X}_{T}, \tilde{{\bf X}})
&\ll 
\begin{cases}
\frac{LN(\log\log\log T)^{2}}{\sqrt{\log\log T}} + \frac{M}{(\log\log T)^{1 - \varepsilon - \varepsilon_{3}}}\quad&\mbox{if} \quad N = 1, 2; \\
\frac{LN(\log\log\log T)^{2}}{\sqrt{\log\log T}} +  \frac{M}{(\log\log T)^{1/2-\varepsilon_{4}}}\quad&\mbox{if} \quad N = 3.
\end{cases}
\end{align*}
\end{theorem}

\begin{remark} (i) If $L = o(\sqrt{\log\log T}/ (\log\log\log T)^{2})$ and $M = o((\log\log T)^{1 - \varepsilon - \varepsilon_{3}})$ when $N = 1, 2$ or $M = o((\log\log T)^{1/2 -\varepsilon_{4}})$    when $N = 3$, then we still have that the corresponding Dudley distance converges to $0$ with a rate of convergence of $o(1)$.

\noindent (ii) It is undoubtedly desirable to extend the second bound for $N\ge 4$. However, unfortunately, our method would lead to a ``divergent'' upper bound in terms of $T$ (see \eqref{final-same-bd}). 
\end{remark}

Let $(\zeta_{K_{j}})_{j = 1}^{N}$ be a sequence of Dedekind zeta function of quadratic fields $K_{j}$. For each $j$, we let $\chi_{K_j}$ stand for the (quadratic) primitive Dirichlet character associated to $K_{j}$, and we recall the factorisation 
$
\zeta_{K}(s)=\zeta(s)L(s,\chi_{K}),
$
where $\zeta(s)$ is the Riemann zeta function, and $L(s,\chi_{K})$ is the Dirichlet $L$-function attached to $\chi_{K}$. In addition, if $K_i$ and $K_j$ are distinct, it is well-known that $\chi_{K_i}\neq \chi_{K_j}$ and thus $\zeta_{K_i}(s)\neq \zeta_{K_j}(s)$.
As in \cite{HW-Forum}, we have the following $N$-dimensional random vector:
\begin{align}\label{Dedekind-zeta-vector}
{\bf Y}_{T}:= (\log|\zeta_{K_{1}}(\oh + \mi (U + \alpha_{1}))|, \ldots, \log|\zeta_{K_{N}}(\oh + \mi (U + \alpha_{N}))|)/ \sqrt{\log\log T}.
\end{align}
Our next theorem is an upper bound for rates of convergence of the central limit theorem of Dedekind zeta functions established in \cite[Theorem 1.3]{HW-Forum}.

\begin{theorem}\label{CLT for Dedekind zeta}
Let $(\zeta_{K_{j}})_{j=1}^{N}$ be a sequence of Dedekind zeta functions of quadratic fields $K_{j}$, and let $\alpha= (\alpha_{1}, \alpha_{2}, \ldots, \alpha_{N})$ satisfy the same condition as in Theorem \ref{main-thm}. Let  $c_{i, j}$ be defined as in \eqref{cond-1} and \eqref{cond-01}, and set
$$
\Delta_{i,j}
=\begin{cases} 
1& \mbox{if $\chi_{K_i} \bar{\chi}_{K_j}$ is principal;}\\ 
0 & \mbox{otherwise.}
\end{cases}
$$
Set $\chi_j = \chi_0$ for $1\le j\le N$. Also, for $1\leq j\leq N$, define $\chi_{N+j} =\chi_{K_j}$ and $\alpha_{N+j} = \alpha_{j}$. Suppose that for any $b_{i}, b_{j}\in\mathbb{R}$,  $\sum_{1\leq i, j\leq 2N}b_{i}b_{j}\delta_{i, j}c_{i, j}\geq 0$, where $\delta_{i, j}$ is defined as in \eqref{def-deltaij}. Then there exists an $N$-variate normal random variable $\tilde{{\bf Y}} = (\tilde{\mathcal{Y}}_{1}, \ldots, \tilde{\mathcal{Y}}_{N})$ with mean ${\bf 0}_{N}$ and covariance matrix
\begin{align*}
\mathfrak{Q} = (\mathfrak{Q}_{i, j})_{1\leq i, j \leq N}
\quad
\text{with}
\quad
\mathfrak{Q}_{i, j} =
\begin{cases}
1\quad&\mbox{if}\quad $i = j$;\\
\frac{1}{2}c_{i, j}(\Delta_{i, j} +1)\quad&\mbox{if}\quad i\neq j,
\end{cases}
\end{align*}
such that for any fixed $0<\varepsilon_{1}<\varepsilon_{2}$, with $\varepsilon_{1} + \varepsilon_{2}<1$, and any  sufficiently large $T$, 
$$
d_{\mathcal{D}}({\bf Y}_{T}, \tilde{{\bf Y}})\ll_{{\mathfrak{Q}}, N }\frac{L}{(\log\log\log T)^{\varepsilon_{1}}}  +  M(\log\log\log T)^{N(\varepsilon_{1} + \varepsilon_{2})}\exp((-1/2)(\log\log\log T)^{\varepsilon_{1} + \varepsilon_{2}}),
$$
 Consequently, as $T\rightarrow\infty$, the vector ${\bf Y}_{T}$ converges to $\tilde{{\bf Y}}$ in distribution.
\end{theorem}

As revealed in the proof of \cite[Theorem 1.3]{HW-Forum}, by the identity $
\log|\zeta_{K}(s)| = \log |\zeta(s)| + \log |L(s, \chi_{K})| 
$, one can regard \eqref{Dedekind-zeta-vector} as \eqref{X_{T}} with $N$ being replace by $2N$, which allows one to deduce Theorem \ref{CLT for Dedekind zeta} from Theorem \ref{main-thm} straightforwardly. So, we shall omit the proof for the conciseness.

\subsection{Further discussions - number-theoretic aspect}

Selberg's proof of his central limit theorem \cite{Se46, Se92} relies on the method of moments (which is also used in \cite{Bourgade, HW, HW-Forum, RS}). A main drawback of the method of moments is that one would lose track of the rate of convergence. Nonetheless, in \cite[Theorem 6.1]{Tsang thesis}, Tsang extended Selberg's approach and showed that the rate of convergence is $O((\log\log\log T)^{2}/\sqrt{\log\log T})$ in Kolmogorov distance. Recently, Roberts \cite{AR24, ARPhD} proved such a rate of convergence in Dudley distance as well as a special case of Theorem \ref{main-thm} (with $N = 2$ and $L=M=1$) for any fixed Dirichlet $L$-function (including the Riemann zeta function). In \cite{AR25}, assuming the Riemann hypothesis, Roberts improved his result in \cite{AR24} to $O(\sqrt{\log\log\log\log T}/\sqrt{\log\log T})$. Moreover, we note that the work of Selberg \cite{Se46, Se92} also requires sophisticated estimates for certain sums over the non-trivial zeros of the Riemann zeta function (see also \cite{Bourgade, Tsang thesis}). In contrast, the method of {Radziwi\l\l}-Soundararajan uses only the elementary zero-density estimate. This method has been used by the present authors \cite{HW, HW-Forum} and by Roberts \cite{AR24}.

Furthermore, we discuss the main difference between the arguments of \cite{HW, HW-Forum} and \cite{AR24}. In the present authors' previous work \cite{HW-Forum}, we made use of the distance of the shifts to construct a positive semi-definite matrix (\cite[Eq. (1.11)]{HW-Forum}). Then by the existence theorem of multivariate normals (see, e.g., \cite{HH} and \cite[Sec. 29]{Billingsley}), we constructed a multivariate normal random variable, with the required mean and covariance matrix, and then employed the Cram\'er-Wold device and the method of moments to derive the desired convergence. This is implicitly used in \cite{HW} and seems complicated to track the rate of convergence as discussed above. In contrast, Roberts \cite[Lemma 9]{AR24} directly computed the distance between the characteristic functions of suitable approximations for the targeted random variables. Nonetheless, one still requires the information on moments of the involved random variables as in \cite[Lemma 7]{AR24}. 

As may already have been noticed, unfortunately, our main result, Theorem \ref{main-thm} (with $N=2$), does not recover Roberts' work in \cite{AR24}. It is worthwhile remarking that our Propositions \ref{Prop 1 in AR24}-\ref{Prop 5 in AR24} extends \cite[Propositions 1-5]{AR24} to a high-dimensional consideration with the same strength, while Proposition \ref{Prop 7} fixes \cite[Proposition 7]{AR24} with a negligible error.  The major obstacle comes from proving Proposition \ref{Prop 6 in AR24}, where we could only obtain an upper bound of the order $(\log\log\log T)^{-\varepsilon_{1}}$, for some $0<\varepsilon_{1}<1/2$, while its counterpart given in  \cite[Proposition 6]{AR24} is $\ll(\log\log T)^{-1/2}$. The major difficulties causing our weaker bounds are the following. First of all, in  \cite[Lemma 5]{AR24}, the error term $ O_{k}(V(u, u^{\prime})^{k - 2})$ appears (with an implicit dependence on $k$), but from the last line in \cite[p. 3361]{AR24} to \cite[Eq. (19)]{AR24}, this lower order term was dropped. When $k$ is fixed, this is reasonable; however, as $k$ would be chosen to be $\floor{\log\log T}$ later in \cite[Lemma 7]{AR24}, one could not discard this error directly (cf. \eqref{strategy explanation} below). 
In addition, in the second last displayed equation from \cite[p. 3367]{AR24}, it was claimed that the terms in
the brackets are bounded by  $O_n(V(\xi, \xi^{\prime})^{n/2-2}/\tilde{\mathfrak{s}}^{n})$. Unfortunately, without precise information on the implied constant, it is unclear whether the resulting sum
$$
\sum_{n\leq \mathfrak{N}} \frac{\mi^n}{n!}O_n(V(\xi, \xi^{\prime})^{n/2-2}/\tilde{\mathfrak{s}}^{n})
$$
could be bounded as claimed when $\mathfrak{N} \ge C \log \log T$. Also, we remark that both $\xi$ and $\xi^{\prime}$ are \emph{not} a fixed constant, but bounded by $C\sqrt{\log\log T}$. Consequently, the last claimed bound in \cite[p. 3367]{AR24} requires a further verification. Indeed, by \cite[Proposition 5]{AR24}, one has 
$
\tilde{\mathfrak{s}}^{2} = \frac{1}{2}\sum_{p\leq Y}\frac{1}{p}\sim\log\log T
$
for sufficiently large $T$. So, choosing $ u=\xi = u'=  \xi^{\prime}= C\sqrt{\log\log T}$ would lead to
\begin{align}\label{bullshit-Roberts-V}
\frac{V(\xi, \xi^{\prime})^{\mathfrak{N}/2-2}}{\tilde{\mathfrak{s}}^{\mathfrak{N}}}\sim\frac{((\log\log T)^{2})^{\mathfrak{N}/2 - 2}}{(\log\log T)^{\mathfrak{N}/2}}
=(\log\log T)^{\mathfrak{N}/2 - 4}.
\end{align}
However, as $\mathfrak{N}\geq C\log\log T$, it is quite unclear why \eqref{bullshit-Roberts-V} $\ll(\log\log T)^{-2}$.

The above-discussed issues, therefore, lead us to choose a set of different parameters in Section \ref{pf-prop-6}. In the setting of proving Theorem \ref{main-thm}, the parameters $\mathfrak{r}$ and $\|{\bf u}\|_{1}$ are chosen as in \eqref{final choice of mathfrak r} and \eqref{final choice of F}, respectively, which results in $\mathfrak{N}$ being as in \eqref{simplified mathfrak N}. (In fact, the choices of these parameters in \cite[Sec. 2.8]{AR24} were $\mathfrak{r} = \log\log T$ and $\|{\bf u}\|_{1} = \sqrt{\log\log T}$, which fails to converge in the present machinery, as explained in \eqref{explain-the-shit-of-AR24-end}). Nevertheless, taking $\mathfrak{r}$ and $\|{\bf u}\|_{1}$ in the fashion of \eqref{Tsang's r} and \eqref{the growth of u}, respectively, which is close to those in \cite[Proposition 6]{AR24}, we are able to recover Tsang's result and extend it to two distinct Dirichlet $L$-functions in Theorem \ref{indep-rate}. In light of this, it may be possible to fix the issues mentioned above to establish a stronger rate of convergence for (at least two) shifted Dirichlet $ L$-functions. We reserve this for further investigation.

\subsection{Further discussions - probabilistic aspect}\label{discussion-on-prob}
Regarding the uniform distributions $(U_{T})_{T\geq 0}$ as random variables defined in a generic probability space $(\mathfrak{S}, \mathfrak{F}, \mathfrak{P})$, one can consider $(X_{T})_{T\geq 0}$, defined  in \eqref{X_{T}}, as a \emph{stochastic process}. Therefore, the works in \cite{HW, HW-Forum}, and the present article can be put in a unified probabilistic scheme, namely, seeking the stationary distribution/limiting distribution of a stochastic process. As shown in Proposition \ref{Prop 6 in AR24}, the process $(X_{T})_{T\geq 0}$ is almost a \emph{Gaussian process} when $T$ is large enough. More precisely, it is an $N$-dimensional Gaussian process whose components are correlated with correlation given by $\log( \frac{1}{|\alpha_{i} - \alpha_{j}|})$. This is an important feature of the \emph{Gaussian multiplicative chaos}. Moreover, by the work of Sakman-Webb \cite{SW20} and Vihko \cite{Vihko}, it is not surprising that, from the probabilistic point of view, the components in $(X_{T})_{T\geq 0}$ are \emph{log-correlated}. Nonetheless, we would like to emphasise that such a log-correlation in the present article (and our previous work \cite{HW-Forum}) were inspired by number theory (see, e.g., \cite{Ch, NSW}). Additionally, we note here that, in fact, the uniform distributions $(U_{T})_{T>0}$ are  \emph{independent}, which implies that the random vectors $({\bf X}_{T})_{T>0}$ are independent. Such an independent assumption is natural from number theory (see, e.g., \cite{LLR19, RS}) although it may not be used explicitly in the present work. Thus, we may regard Theorem \ref{main-thm} as a concrete example of the multivariate central limit theorem of independent random vectors.

%

The study of the rate of convergence in central limit theorems (CLT) is a classical topic in probability theory. Let $(X_{j})_{j = 1}^{n}$ be independent identically distributed (iid) random variables and write $S_{n} = \sum_{j = 1}^{n}X_{j}$. The Berry-Esseen bound  \cite{Esseen45} of (univariate) CLT for $S_{n}/\sqrt{n}$ asserts that the rate of convergence is about $O(n^{-1/2})$. If the random variables exhibit specific dependence structures, in the univariate case, the Berry-Esseen bound remains  $O(n^{-1/2})$. The dependence structure appears as a constant factor on the right side of the estimate (see, e.g., classical works of Hoeffding-Robbins \cite{HR48}, Diananda \cite{Diananda55}, and Orey \cite{Orey58} as well as recent works of Chen-Shao \cite{CS04}, Janson-Pratelli-Rigo \cite{JPR24}). For multivariate CLT with independent random variables, the works of Sazonov \cite{Sazonov68},  G\"otze \cite{Gotze91},  Rai\v{c} \cite{Raic19}, and  Fang-Koike \cite{FK24} reveal that the rate of converges remains $O(n^{-1/2})$, and the effect from dimension is $N^{\alpha}$ for some $\alpha$. 
When dimensionality and dependence both manifest, the rate of convergence of such a CLT becomes more sophisticated. By the recent article of Chang-Chen-Wu \cite[Corollary 1]{CCW24}, the rate of convergence becomes $O(n^{-1/6})$  instead of $O(n^{-1/2})$, and the dimension appears as a factor $(\log N)^{7/6}$ on the right side. The number of dependence random variables $m$ appears as a factor $(m\vee 1)^{2/3}$.

According to its probabilistic counterpart (e.g., Sazonov \cite{Sazonov68},  G\"otze \cite{Gotze91},  Rai\v{c} \cite{Raic19}, and  Fang-Koike \cite{FK24}), one would expect that the rate of convergence in Theorem \ref{main-thm} shall be the same order as in the univariate case (Tsang's result or Theorem \ref{indep-rate}). Nevertheless, it is much worse. Secondly, we found here that the dependence among components can have a dramatic impact on the rate of convergence in CLT (cf. Lemmata \ref{lemma 9' in AR24} and \ref{better-lemma 9}). This seems beyond the prediction of probability theory. Thirdly, the effect from dimension in Theorems \ref{main-thm} and \ref{indep-rate} is also different from its probabilistic counterpart (e.g., Sazonov \cite{Sazonov68},  G\"otze \cite{Gotze91},  Rai\v{c} \cite{Raic19}, and  Fang-Koike \cite{FK24}). Particularly, in the proof of Theorem \ref{indep-rate}, we fail to track the corresponding rate of convergence if $N\geq 4$. 

The above-mentioned discrepancies between the  probabilistic results and their number-theoretic counterparts suggest that there may be potential improvement. In addition, the random matrix model suggests that the error in the Selberg's central limit theorem is $\ll (\log\log T)^{-3/2}$, but the method seems limited to yield a weaker error of order $ (\log\log T)^{-1/2} $. In fact, even removing $(\log\log\log T)^{2}$ seems quite challenging.
 
\subsection{Structure of the article}
The present article is organised as following. In Section \ref{Pre}, we collect preliminaries from both number theory and probability. The proof of Theorem \ref{main-thm} is presented in Section \ref{Pf-of-main-thm}. As the proof of Theorem \ref{main-thm} consists of several auxiliary propositions, the proofs of those propositions are in Sections \ref{pf-prop-2} - \ref{pf-prop-7}. As the proof of Theorem \ref{indep-rate} requires manipulations with a new set of parameters, we devote the whole Section \ref{Tsang} to its proof.

\section{Notation and preliminaries}\label{Pre}
Throughout this article, we use the Landau-Vinogradov notation. We denote $f=o(g)$ (resp., $f\sim g$) if $f(x)/g(x)$ tends to $0$ (resp., 1) as $x\rightarrow\infty$. In addition, we write both $f=O(g)$ and $f\ll g$ to mean that there is a constant $C$ such that $|f(x)| \le C g(x)$ for  $x$ sufficiently large. We let $\Omega(n)$ be the number of prime divisors of $n$ counted with multiplicity. The von Mangoldt function $\Lambda(n)$ is defined by $\Lambda(n)=\log p$
if $n$ is a power of a prime $p$, and $\Lambda(n)=0$ otherwise. The M\"obius function $\mu(n)$ is defined by $\mu(n)= (-1)^{\Omega(n)}$ if $n$ is square-free; $\mu(n)=0$  otherwise. Throughout this article, the Stirling formula employed is the following form
\begin{align}\label{Stirling formula}
n! = \sqrt{2\pi n}(n/e)^{n}(1 + O(1/n))
\end{align}
unless otherwise mentioned.

 All the $\varepsilon_{j}>0$ are independent of $T$ except for $j = 5, 6$. They can be arbitrarily small unless otherwise mentioned. All the $\kappa_{j}$'s are fixed positive constants as well. We will use $\mathcal{E}_{j}$ to denote some complicated terms. The definition of $\mathcal{E}_{j}$ may vary from proof to proof, but it is fixed within a single proof. Given ${\bf a} = (a_{1}, a_{2}, \ldots, a_{N})\in\mathbb{R}^{N}$, we follow the usual $p$-norm to define
\begin{align}\label{lp-norm}
\|{\bf a}\|_{p}:=\bigg(\sum_{j = 1}^{N}|a_{j}|^{p}\bigg)^{1/p}
\end{align}
for any $0<p<\infty$. Recall that for $0<p<q\leq\infty$, one has
$
\|{\bf a}\|_{q}\leq\|{\bf a}\|_{p}.
$ 
Note that all the boldface alphabets (e.g., ${\bf u}, {\bf v}$) will be used to denote $\mathbb{R}^{N}$-valued vectors.

For the sake of simplicity, in the rest of article, for given $L,M\ge 0$, we shall denote $\mathcal{L}=\mathcal{L}_{L,M}$ and write $d_{\mathcal{D}}= d_{\mathcal{D},L,M}$. It then follows directly from \eqref{the space LLM} and \eqref{Dudley distance} that 
\begin{align*}
d_{\mathcal{D}}({\bf X}, {\bf Y})
\leq\sup_{f\in\mathcal{L}}\mathbb{E}[|f({\bf X}) - f({\bf Y})|]
\leq \sup_{f\in\mathcal{L}} \mathbb{E}[\|f\|_{Lip}\|{\bf X} - {\bf Y}\|_{2}]
\leq L \mathbb{E}[\|{\bf X} - {\bf Y}\|_{2}]
\leq L \mathbb{E}  [\|{\bf X} - {\bf Y}\|_{1}].
\end{align*}
Thus, for any random vectors ${\bf X}=(X_1,\ldots,X_N)$ and ${\bf Y}=(Y_1,\ldots,Y_N)$, one has
\begin{align}\label{Dudley ineq}
d_{\mathcal{D}}({\bf X}, {\bf Y})\leq 
L\sum_{j = 1}^{N}\mathbb{E} [|X_{j} - Y_{j}|],
\end{align}
which will be used frequently throughout our argument.

 As mentioned in Section \ref{discussion-on-prob}, we let $(U_{T})_{T>0}$ be a family of independent uniform distributions defined on a probability space $(\mathfrak{S}, \mathfrak{F}, \mathfrak{P})$ whose density function being
\begin{align*}
f_{U_{T}}(t) = 
\begin{cases}
\frac{1}{T}\quad&\mbox{if}\quad T\leq t\leq 2T;\\
0\quad&\mbox{otherwise}.
\end{cases}
\end{align*}
We shall write $U = U_{T}$ when there is no ambiguity. We will use $\mathbb{E}$ to refer the expectation with respect to $(\mathfrak{S}, \mathfrak{F}, \mathfrak{P})$. Let $\mathfrak{L}_{T}$ denote the Lebesgue measure on $[T, 2T]$. For a Borel set $\mathcal{B} \subseteq[T, 2T]$, we denote $\mathfrak{P}_{T}(\mathcal{B})=\frac{1}{T}\mathfrak{L}_{T}(\mathcal{B})$. Correspondingly, we denote 
$\mathbb{E}_{T}[f]=\frac{1}{T}\int^{2T}_{T}f(t)dt$. Here, we particularly reserve the bracket ``$[\cdot ]$" for expectation only. Also, to indicate the input of a function, we will only use ``$(\cdot)$". For instance, we shall write $\mathbb{E}_{T}[\exp(t)]$. 

Following \cite[pp. 3350-3351]{AR24}, for sufficiently large constants $K$ and $K'$, with $2<K'<K$, we set
\begin{align*}
W=K(\log\log\log T)^2,\quad X=T^{(K'\log\log\log T)^{-1}},\quad Y=T^{(K'\log\log T)^{-1}},\quad \sigma_0=\frac{1}{2}+\frac{W}{\log T}.
\end{align*}
In addition, for each $j$, we set $s_{0, j} = \sigma_{0} + \mi (U + \alpha_{j})$.

\section{Key propositions and the proof of the main theorem}\label{Pf-of-main-thm}

\subsection{List of all key propositions}


In light of \cite[Eq. (1.5)]{HW-Forum}, 
to approximate ${\bf X}_{T}(U)$ defined in \eqref{X_{T}}, we consider
\begin{align}\label{X^{0}_{T}}
{\bf X}^{0}_{T} = {\bf X}^{0}_{T} (U)= (\mathcal{X}^{0}_{\alpha_{1}, \chi_{1}, T}(U), \ldots, \mathcal{X}^{0}_{\alpha_{N}, \chi_{N}, T}(U)),
\end{align}
where
\begin{align}\label{def-XjT-chij-0}
\mathcal{X}^{0}_{\alpha_j, \chi_j, T}(U) = (\log|L(\sigma_{0} +\mi(U+\alpha_{j}),\chi_j)|) / \sqrt{\oh\log\log T}.
\end{align}

\begin{proposition}\label{Prop 1 in AR24}
The Dudley distance between ${\bf X}_{T}$ and ${\bf X}^{0}_{T}$ (defined as in \eqref{X_{T}} and \eqref{X^{0}_{T}}, respectively) satisfy the following estimate:
\begin{align*}
d_{\mathcal{D}}({\bf X}_{T}, {\bf X}^{0}_{T})\ll 
\frac{LN(\log\log\log T)^{2}}{\sqrt{\log\log T}}.
\end{align*}
\end{proposition}

\begin{proof}
Using inequality \eqref{Dudley ineq} with ${\bf X}_{T} = {\bf X}$ and ${\bf X}^{0}_{T} = {\bf Y}$, one arrives at
\begin{align*}
d_{\mathcal{D}} ({\bf X}_{T}, {\bf X}^{0}_{T})
\leq 
L\sum_{j = 1}^{N}\mathbb{E}[|\mathcal{X}_{\alpha_{j}, \chi_{j}, T} - \mathcal{X}^{0}_{\alpha_{j}, \chi_{j}, T}|].
\end{align*}
Recall the notation from \eqref{def-XjT-chij} and \eqref{def-XjT-chij-0} and the convention of $\mathbb{E}_{T}$. We see that
\begin{align*}
\mathbb{E}_{T}[|\mathcal{X}_{\alpha_{j}, \chi_{j}, T} - \mathcal{X}^{0}_{\alpha_{j}, \chi_{j}, T}|]
\ll \frac{W}{\sqrt{\log\log T}}
\end{align*}
by \cite[Proposition 3.1]{HW-Forum}. Hence, we complete the proof.
\end{proof}

Let $A\ge 0$ be a fixed constant that will be chosen later. We introduce an auxiliary series as follows. We set $a(n) = 1$ if $n$ does not have any prime divisor greater $X$, and $n$ has at most $100\log\log T$ primes below $Y$ and at most $A\log\log\log T$ primes between $Y$ and $X$. Otherwise, we set $a(n) = 0$. (Note that consequently,  $a(n) = 0$ unless $n\leq Y^{100\log\log T}X^{A\log\log\log T} \le T^{\varepsilon}$ whenever $\frac{1}{K'} (100+A) \le \varepsilon$.)
 Based on \cite[Eq (3.3)]{HW-Forum}, we define
\begin{align}\label{def-M}
M (s; \chi) = \sum_{n}\frac{\mu(n)a(n)\chi(n)}{n^{s}}, 
\end{align}
and we set ${\bf M}_{T} = (\mathcal{M}_{1}, \mathcal{M}_{2}, \ldots, \mathcal{M}_{N}) $ with
\begin{align}\label{def of mathcal M}
\mathcal{M}_{j}:= (\log |M^{-1} (\sigma_{0} + \mi(U + \alpha_{j}); \chi_{j})| )/ \sqrt{\oh\log\log T}.
\end{align}
 We have the following estimate:
\begin{proposition}\label{Prop 2 in AR24}
With the notation above, we have
\begin{align*}
d_{\mathcal{D}}({\bf M}_{T}, {\bf X}^{0}_{T})
\ll\frac{LN}{\sqrt{\log\log T}}+MN(\log\log T)^{-K/K'}.
\end{align*}
\end{proposition}

Recall the following auxiliary series from \cite[Eq. (3.1)]{HW-Forum}:
\begin{align}\label{def of mathcal P}
\mathcal{P}(s; \chi):=\sum_{2\leq n\leq X}\frac{\Lambda(n)\chi(n)}{n^{s}\log n}.
\end{align}
For simplicity, we write 
\begin{align}\label{def of mathcal Q}
\mathcal{Q}_{j} = \mathcal{Q}_{j,T} = \Re(\mathcal{P}(\sigma_{0} + \mi (U+\alpha_{j}); \chi_{j}))/\sqrt{\oh \log\log T}
\end{align}
and ${\bf Q}_{T} = (\mathcal{Q}_{1}, \mathcal{Q}_{2}, \ldots, \mathcal{Q}_{N})$. We have the following result of approximating ${\bf M}_{T}$ by ${\bf Q}_{T}$.

\begin{proposition}\label{Prop 3 in AR24}
Let ${\bf M}_{T}$ be the same as in Proposition \ref{Prop 2 in AR24} and ${\bf Q}_{T}$ be as above. We then have
\begin{align*}
d_{\mathcal{D}}({\bf M}_{T}, {\bf Q}_{T})\ll N(L + M)(\log\log T)^{-80}.
\end{align*}
\end{proposition}

In the proof of Proposition \ref{Prop 3 in AR24}, we shall introduce the  auxiliary series
\begin{align*}
P(s) = P(s; \chi):=\Re\sum_{p\leq X}\frac{\chi(p)}{p^{s}}
\end{align*}
 (see \eqref{def of usual P}below). Plugging $s_{0, j} = \sigma_{0} + \mi (U+\alpha_{j})$ into $P(s)$, we make it a random variable. Define
\begin{align}\label{bf R_{T}}
{\bf R}_{T}=  (P(\sigma_{0} + \mi (U+\alpha_{j}));\chi_k)_{j=1}^{N})/\sqrt{\oh \log\log T}.
\end{align}
The following proposition is to show that we can use  ${\bf R}_{T}$ to approximate ${\bf Q}_{T}$.

\begin{proposition}\label{Prop 4 in AR24}
For ${\bf Q}_{T}$ and ${\bf R}_{T}$ being as above, we have the following estimate:
\begin{align*}
d_\mathcal{D}  ( {\bf Q}_{T}, {\bf R}_{T})
\ll\frac{LN}{\sqrt{\log\log T}}.
\end{align*}
\end{proposition}

As an analogue of $\tilde{\mathfrak{s}}$ in \cite{AR24}, we consider the normalisation factor
\begin{align}\label{MTq}
\mathfrak{M}_{T, \chi_j} = P_{1}(2\sigma_{0};\chi_j \bar{\chi}_j):=\sum_{13<p\leq Y}\frac{\chi_j\bar{\chi}_j(p)}{p^{2\sigma_{0} }},
\end{align}
and we define
\begin{align}\label{bf R^{1}_{T}}
{\bf R}^{1}_{T}=  (P_{1}(\sigma_{0} + \mi (U+\alpha_{j});\chi_j)_{j=1}^{N}) / \sqrt{\oh\mathfrak{M}_{T, \chi_j}}.
\end{align}
We will show that the Dudley distance changes only slightly with a gentle change in the normalisation factor. 
\begin{proposition}\label{Prop 5 in AR24}
Let ${\bf R}_{T}$ and ${\bf R}^{1}_{T}$ be defined as in \eqref{bf R_{T}} and \eqref{bf R^{1}_{T}}, respectively. We have
\begin{align*}
d_\mathcal{D}  ({\bf R}_{T}, {\bf R}^{1}_{T})
\ll L\sqrt{N} \frac{ \sqrt{ 1+ \log\log\log T}}{\sqrt{\log\log T}}.
\end{align*}
\end{proposition}

The reason for changing the normalisation factor in ${\bf R}_{T}$ will be apparent in the following proposition, in which we will construct a multivariate normal distribution having the \emph{same} ``covariance matrix'' as ${\bf R}^{1}_{T}$. If we were not changing the renormalisation factor in ${\bf R}_{T}$, then the diagonal of its ``covariance matrix'' would \emph{not} be its variance of its component, but approximate to it (see, e.g., \cite[Eq. (3.17)]{HW-Forum}). This would then violate the very definition of a covariance matrix.

Suppose that $\tilde{\bf Z}_{N, T} =(\tilde{\mathcal{Z}}_{1, T}, \ldots, \tilde{\mathcal{Z}}_{N, T})$ is an $N$-variate normal distribution with mean ${\bf 0}_{N}$ and covariance matrix the same as ${\bf R}^{1}_{T}$ for each $T> T_{0}$ for some $T_{0}$. What we require here is, in fact, a \emph{Gaussian process} rather than simply a multivariate Gaussian random variable. Thus, we will employ the full power of the existence theorem of a Gaussian process (cf. \cite[Theorem 1.1(iii)]{HW-Forum}).

\begin{proposition}\label{Prop 6 in AR24}
Let $\tilde{\bf Z}_{N, T} =(\tilde{\mathcal{Z}}_{1, T}, \ldots, \tilde{\mathcal{Z}}_{N, T})$ be an $N$-variate normal distribution with mean ${\bf 0}_{N}$ and covariance matrix the same as ${\bf R}^{1}_{T}$. Then for sufficiently large $T$, we have
\begin{align*}
d_{\mathcal{D}} ({\bf R}^{1}_{T}, \tilde{\bf Z}_{N, T} )
&\ll_{N}\frac{L}{(\log\log\log T)^{\varepsilon_{1}}} 
+ M(\log\log\log T)^{N(\varepsilon_{1} + \varepsilon_{2})}\exp((-1/2)(\log\log\log T)^{\varepsilon_{1} + \varepsilon_{2}}),
\end{align*}
where $0<\varepsilon_{1}<\varepsilon_{2}$ satisfying $\varepsilon_{1} + \varepsilon_{2}<1$.
\end{proposition}

By the propositions, we have shown that the vector ${\bf X}_{T}(U)$, formed by the logarithm of the real part of shifted Dirichlet $ L$-functions, is approximately normal with a particular covariance matrix. The last piece of the whole proof is the following: A slight change in the covariance matrices of two given multivariate normal distributions is negligible. The proof relies on techniques from matrix analysis.

\begin{proposition}\label{Prop 7}
Let  $\tilde{\bf X}$ be defined as in \eqref{normal X} and $\tilde{\bf Z}_{N, T}$ be the same as in Proposition \ref{Prop 6 in AR24}. Then for any $\varepsilon_{3}>0$, we have
\begin{align*}
d_{\mathcal{D}} (\tilde{\bf Z}_{N, T}, \tilde{\bf X} )
\ll_\mathfrak{K}M (\log\log T)^{-1+\varepsilon +\varepsilon_{3}}.
\end{align*}
\end{proposition}

\subsection{Proof of Theorem \ref{main-thm}}
Let ${\bf X}_{T}$ be as in \eqref{X_{T}} and $\tilde{\bf X}$ as in \eqref{normal X}. Also, suppose that 
${\bf X}^{0}_{T}, {\bf M}_{T},{\bf Q}_{T}, {\bf R}_{T}, {\bf R}^{1}_{T}$, and $\tilde{{\bf Z}}_{N, T}$ be defined as in Propositions \ref{Prop 1 in AR24}-\ref{Prop 7}. Then, by the triangle inequality, we have
\begin{align*}
d_{\mathcal{D}}({\bf X}_{T}, \tilde{{\bf X}})
&\ll d_{\mathcal{D}}({\bf X}_{T}, {\bf X}^{0}_{T}) 
+ d_{\mathcal{D}}({\bf X}^{0}_{T}, {\bf M}_{T})
+ d_{\mathcal{D}}({\bf M}_{T}, {\bf Q}_{T})
+ d_{\mathcal{D}}({\bf Q}_{T}, {\bf R}_{T})
+ d_{\mathcal{D}}({\bf R}_{T}, {\bf R}^{1}_{T})
+ d_{\mathcal{D}}({\bf R}^{1}_{T}, \tilde{{\bf Z}}_{T})
+ d_{\mathcal{D}}(\tilde{\bf Z}_{T}, \tilde{\bf X})
\\& \ll
\frac{LN(\log\log\log T)^{2}}{\sqrt{\log\log T}}
+\Big(\frac{LN}{\sqrt{\log\log T}}+MN(\log\log T)^{-K/K'}\Big)
 +N(L + M)(\log\log T)^{-80}
\\&+\frac{LN}{\sqrt{\log\log T}}
+L\sqrt{N} \frac{ \sqrt{ 1+ \log\log\log T}}{\sqrt{\log\log T}}
+\frac{L}{(\log\log\log T)^{\varepsilon_{1}}} 
\\&+ M(\log\log\log T)^{N(\varepsilon_{1} + \varepsilon_{2})}\exp((-1/2)(\log\log\log T)^{\varepsilon_{1} + \varepsilon_{2}})
+M (\log\log T)^{-1+\varepsilon +\varepsilon_{3}}
\\&\ll \frac{L}{(\log\log\log T)^{\varepsilon_{1}}}  +  M(\log\log\log T)^{N(\varepsilon_{1} + \varepsilon_{2})}\exp((-1/2)(\log\log\log T)^{\varepsilon_{1} + \varepsilon_{2}}),
\end{align*}
where $\varepsilon_{3}>0$ and $0<\varepsilon_{1}<\varepsilon_{2}$ satisfying $\varepsilon_{1} + \varepsilon_{2}<1$.

\begin{remark}
The constants $L$ and $M$ in \eqref{the space LLM} may be taken as a function in $T$ as long as the convergence structure is not destroyed. In the present setting, the constant $L$ may be taken as $o((\log\log\log T)^{\varepsilon_{1}})$ when $\varepsilon_{1}$ is chosen (e.g., $L = o(\log\log\log\log T)$); the constant $M$ may be chosen to be $ C(\log\log\log T)^{2}$ for any fixed $C$. Choosing $L$ and $M$ in the above mentioned manner, by \eqref{the space LLM}, we see that $\cup_{L,M\in \Bbb{R}^+} \mathcal{L}_{L,M} = \cup_{T>0}\mathcal{L}_{L(T), M(T)} = C_{b}(\mathbb{R}^{n})$, all bounded continuous functions defined on $\mathbb{R}^{n}$. Hence, we recover the test function space of convergence in distribution.
\end{remark}

\section{Proof of Proposition \ref{Prop 2 in AR24}}\label{pf-prop-2}

To prove Proposition \ref{Prop 2 in AR24}, we require the following Lemma.

\begin{lemma}\label{lemma 2 in AR24}
Let $\chi$ be a primitive Dirichlet character modulo $q$. Then we have
\begin{align*}
\frac{1}{T}\int^{2T}_{T}|1 - L(\sigma_{0} + \mi t , \chi) M(\sigma_{0} + \mi t)|^{2}dt
\ll(\log\log T)^{-K/K'}.
\end{align*}
\end{lemma}

\begin{proof}
The proof of the lemma is in the same spirit as \cite[Proposition 3.4]{HW}. Nonetheless, in order to track the contribution from different parameters, we provide essential calculations here. To begin, we use the approximate functional equation for Dirichlet $L$-functions \cite[Corollary]{Rane}, 
\begin{align*}
L(\sigma_{0} + \mi t, \chi) = \sum_{n\leq T}\frac{\chi(n)}{n^{\sigma_{0} + \mi t}} + O(T^{-\frac{1}{2}}),
\end{align*}
to obtain
\begin{align*}
\int^{2T}_{T}L(\sigma_{0} + \mi t, \chi) M(\sigma_{0} + \mi t)dt
&=\sum_{1\leq n\leq T}\frac{\chi(n)}{n^{\sigma_{0}}}\sum_{1\leq m<T^{1/K'}}\frac{\mu(m)a(m)\chi(m)}{m^{\sigma_{0}}}\int^{2T}_{T}(mn)^{-\mi t}dt + O(T^{\frac{1}{2} + \frac{1}{K'}})
\\& = T + O(T^{\frac{1}{2} + \frac{1}{K'}}).
\end{align*}
 Here, the range of $m$ is different from \cite{HW}, which leads the current error $O(T^{\frac{1}{2} + \frac{1}{K'}})$, which was $O(T^{\frac{1}{2} + \epsilon})$. Therefore, we have
\begin{align*}
&\int^{2T}_{T}|1 - L(\sigma_{0} + \mi t, \chi)M(\sigma_{0} + \mi t)|^{2}dt \nonumber
\\& = \sum_{h, k}\frac{\mu(h)\mu(k)a(h)a(k)}{(hk)^{\sigma_{0}}}\chi(k)\bar{\chi}(h)\int^{2T}_{T}\Big(\frac{h}{k}\Big)^{\mi t}|L(\sigma_{0} + \mi t, \chi)|^{2}dt - T + O(T^{\frac{1}{2} + \frac{1}{K'}}). 
\end{align*}

Recall that \cite[Lemma 7.2]{HW} states  that 
for  any non-negative integers $h, k\leq T$ and any $\frac{1}{2}<\sigma\leq 1$, we have
\begin{align*}
\int^{2T}_{T}\Big(\frac{h}{k}\Big)^{\mi t}|L(\sigma+\mi t,\chi)|^2dt
&=\chi(h)\bar{\chi}(k)\int^{2T}_{T}\Big(L(2\sigma, \chi_{0})\Big(\frac{(h,k)^2}{hk}\Big)^{\sigma}+\Big(\frac{qt}{2\pi}\Big)^{1-2\sigma}L(2-2\sigma, \chi_{0})\Big(\frac{(h,k)^{2}}{hk}\Big)^{1-\sigma}\Big)dt
\\&+O(T^{1-\sigma+\epsilon}\min\{h,k\}),
\end{align*}
where $\chi_0$ is the principal Dirichlet character modulo $q$. As $|a(n)|\le 1$ and  $a(n)=0$ unless $n\le T^\varepsilon$, we see that the error term is at most
\begin{align*}
\ll  \sum_{h, k}\frac{a(h)a(k)}{(hk)^{\sigma_{0}}} (T^{1 - \sigma_{0} +\epsilon}\min\{h, k\})
\ll  T^{\frac{1}{2} + 2\varepsilon}  \sum_{h, k}\frac{1}{(hk)^{\sigma_{0}}} 
\ll  T^{\frac{1}{2} + 2\varepsilon} \Big(\sum_{h\le T^\varepsilon}  \frac{1}{h} \Big)^{\frac{1}{2}}  \Big(\sum_{k\le T^\varepsilon}  \frac{1}{k} \Big)^{\frac{1}{2}} 
\ll T^{\frac{1}{2} + 3\varepsilon},
\end{align*}
where we used the Cauchy-Schwarz inequality in the third estimate.

Following the procedure and notation in  \cite[Eq. (7.9) \& (7.10)]{HW}, we therefore obtain
\begin{align}
 \begin{split}
&\int^{2T}_{T}|1 - L(\sigma_{0} + \mi t, \chi)M(\sigma_{0} + \mi t)|^{2}dt  
\\& = \sum_{h, k}\frac{\mu(h)\mu(k)a(h)a(k)}{(hk)^{\sigma_{0}}}\chi(k)\bar{\chi}(h)\int^{2T}_{T}\left(\frac{h}{k}\right)^{\mi t}|L(\sigma_{0} + \mi t, \chi)|^{2}dt - T + O(T^{\frac{1}{2} + \frac{1}{K'}})  
\\& = \sum_{h, k}\frac{\mu(h)\mu(k)a(h)a(k)}{(hk)^{\sigma_{0}}}\int^{2T}_{T}L(2\sigma_{0}, \chi_{0})\left(\frac{(h, k)^{2}}{hk}\right)^{\sigma_{0}}dt  \label{7.10 in HW}
\\& +  \sum_{h, k}\frac{\mu(h)\mu(k)a(h)a(k)}{(hk)^{\sigma_{0}}}\int^{2T}_{T}\left(\frac{q t}{2\pi}\right)^{1 - 2\sigma_{0}} L(2 - 2\sigma_{0}, \chi_{0})\left(\frac{(h, k)^{2}}{hk}\right)^{1 - \sigma_{0}}dt  
 -T + O(T^{\frac{1}{2} + 3\varepsilon}) +  O(T^{\frac{1}{2} + \frac{1}{K'}}) .  
 \end{split}
\end{align}
Recall that
\begin{align}\label{Dirichlet L with chi_{0}}
L(s, \chi_{0}) = \zeta(s)\prod_{p\mid q}\left(1 - \frac{1}{p^{s}}\right).
\end{align}
Applying \eqref{Dirichlet L with chi_{0}} to $L(2\sigma_{0}, \chi_{0})$ in the first main term in  \eqref{7.10 in HW} and processing as in \cite[pp. 3355-3356]{AR24}, we conclude that the first main term in  \eqref{7.10 in HW} is
$
\ll T(\log\log T)^{-K/K'}.
$

Then we consider the second main term in \eqref{7.10 in HW}. We shall remind the reader that it was $o(T)$ in \cite[p. 703]{HW} (see also \cite[p. 14]{RS}), which is not sufficient. Following the same reasoning as in \cite[p. 703]{HW}, we see that the second main term in \eqref{7.10 in HW} is
\begin{align}\label{second main term in 7.10 in HW}
\sim L(2 - 2\sigma_{0}, \chi_{0})\prod_{p\leq X}\Big(1 - \frac{2}{p} + \frac{1}{p^{2\sigma_{0}}}\Big) \int^{2T}_{T}\Big(\frac{qt}{2\pi}\Big)^{1 - 2\sigma_{0}}dt.
\end{align}
Since $\sigma_{0} = \frac{1}{2} + \frac{W}{\log T}$ and $W = K(\log\log\log T)^{2}$, we infer that the integral term in \eqref{second main term in 7.10 in HW} is
\begin{align*}
\ll_{q} T^{2 - 2\sigma_{0}} = T^{1 - \frac{2W}{\log T}} = T e^{-2W} = Te^{-2K(\log\log\log T)^{2}}.
\end{align*}
In addition, a manipulation of the exponent gives
\begin{align*}
2K(\log\log\log T)^{2} = 2K(\log\log\log T)(\log\log\log T) = \log\Big((\log\log T)^{2K(\log\log\log T)}\Big),
\end{align*}
which implies that
\begin{align*}
e^{-2K(\log\log\log T)^{2}} = \exp\Big( \log\Big((\log\log T)^{-2K(\log\log\log T)}\Big)\Big)
=\frac{1}{(\log\log T)^{2K(\log\log\log T)}}.
\end{align*}
Therefore, we conclude that the integral term in \eqref{second main term in 7.10 in HW} is
\begin{align}\label{second main term in 7.10 in HW_1}
\ll_{q} T^{2 - 2\sigma_{0}} 
\ll_{q} T \frac{1}{(\log\log T)^{2K(\log\log\log T)}}.
\end{align}
Applying \eqref{Dirichlet L with chi_{0}} to $L(2 - 2\sigma_{0}, \chi_{0})$, we obtain
\begin{align*}
L(2 - 2\sigma_{0}, \chi_{0}) = \zeta(2 - 2\sigma_{0})\prod_{p\mid q}\Big(1 - \frac{1}{p^{s}}\Big)
\ll |\zeta(2 - 2\sigma_{0})| = \Big|\frac{1}{(2 - 2\sigma_{0}) - 1}\Big| + O(1)
\end{align*}
since $\zeta(s) = \frac{1}{s - 1} + O(1)$. Plugging $\sigma_{0} = \frac{1}{2} + \frac{W}{\log T}$ and keep simplifying, we arrive at
\begin{align}\label{second main term in 7.10 in HW_2}
L(2 - 2\sigma_{0}, \chi_{0}) \ll \frac{\log T}{W} + O(1)\ll \frac{\log T}{W} = \frac{\log T}{K(\log\log\log T)^{2}}
\end{align}
since $W = K(\log\log\log T)^{2}$. For the second factor in \eqref{second main term in 7.10 in HW},  we have
\begin{align*}
\prod_{p\leq X}\Big(1 - \frac{2}{p} + \frac{1}{p^{2\sigma_{0}}}\Big)
\leq \prod_{p\leq X}\Big(1 - \frac{2}{p} + \frac{1}{p}\Big)
=\prod_{p\leq X}\Big(1 - \frac{1}{p}\Big)
\end{align*}
since $\frac{1}{p^{2\sigma_{0}}}<\frac{1}{p}$. Recall that Mertens' third theorem states
\begin{align*}
\lim_{X\rightarrow\infty}\prod_{p\leq X}\Big(1 - \frac{1}{p}\Big)^{-1} = e^{\kappa_{0}}\log X + O(1),
\end{align*}
where $\kappa_{0}$ is the Euler's constant (see, e.g., \cite[Theorem 2.7(e)]{M-V}). We therefore conclude that
\begin{align}\label{second main term in 7.10 in HW_3}
\prod_{p\leq X}\Big(1 - \frac{2}{p} + \frac{1}{p^{2\sigma_{0}}}\Big)
\ll \frac{1}{\log X} = \frac{K'(\log\log\log T)}{\log T}
\end{align}
since $X = T^{1/K'(\log\log\log T)}$.
Plugging \eqref{second main term in 7.10 in HW_1}, \eqref{second main term in 7.10 in HW_2}, and \eqref{second main term in 7.10 in HW_3} into \eqref{second main term in 7.10 in HW}, we have that  \eqref{second main term in 7.10 in HW} is
\begin{align*}
\ll_{q} T \frac{1}{(\log\log T)^{2K(\log\log\log T)}}\frac{\log T}{K(\log\log\log T)^{2}} \frac{K'(\log\log\log T)}{\log T}
\ll \frac{T}{(\log\log T)^{K/K'}},
\end{align*}
where the implied constant in the last big-O term depends on $q, K$, and $K'$.
\end{proof}

Recall that $s_{0, j} = \sigma_{0} + \mi(U+\alpha_{j})$. Therefore, we shall write $L(s_{0, j}; \chi_{j})$ and $M(s_{0, j}; \chi_{j})$ for simplicity in the following proof. Now, with Lemma \ref{lemma 2 in AR24} in hand, we are in a position to prove Proposition \ref{Prop 2 in AR24}.
\begin{proof}[Proof of Proposition \ref{Prop 2 in AR24}]
For each $j$, we consider the event 
$
\mathcal{E}_{j}:= \{|L(s_{0, j}; \chi_{j})M(s_{0, j}; \chi_{j}) - 1|\leq \oh\}.
$ 
Recalling the definition  \eqref{Dudley distance} of the Dudley distance, we have
\begin{align}\label{equation 4 in AR24}
\begin{split}
d_{\mathcal{D}}({\bf M}_{T}, {\bf X}^{0}_{T})&\leq\sup_{f\in\mathcal{L}}\mathbb{E}[|f({\bf M}_{T}) - f({\bf X}^{0}_{T})|]
\\
& \leq \sup_{f\in\mathcal{L}}\mathbb{E}[|f({\bf M}_{T}) - f({\bf X}^{0}_{T})|\cdot\mathds{1}(\cap_{j = 1}^{N}\mathcal{E}_{j})]
 +\sup_{f\in\mathcal{L}}\sum_{j = 1}^{N} \mathbb{E}[|f({\bf M}_{T}) - f({\bf X}^{0}_{T})|\cdot\mathds{1}(\mathcal{E}_{j}^{c})],
\end{split}
\end{align}
where $\mathcal{E}_{j}^{c}$ is the complement of $\mathcal{E}_{j}$. As $f$ is a Lipschitz function with $||f||_{Lip}\le L$, we have 
$|f({\bf M}_{T}) - f({\bf X}^{0}_{T})|\leq L \|{\bf M}_{T} - {\bf X}^{0}_{T}\|_{1}$. Thus, the first term on the right of \eqref{equation 4 in AR24} is bounded by
\begin{align*}
L\mathbb{E}[\|{\bf M}_{T} - {\bf X}^{0}_{T}\|_{1}\cdot\mathds{1}(\cap_{j = 1}^{N}\mathcal{E}_{j})]
&\leq L \sum_{j = 1}^{N}\mathbb{E}[|\mathcal{M}_{j} - \mathcal{X}^{0}_{\alpha_{j}, \chi_{j}, T}|\cdot\mathds{1}(\mathcal{E}_{j})]
\\
& = \frac{L}{\sqrt{\oh\log\log T}}\sum_{j = 1}^{N}\mathbb{E}[|\log|L(s_{0, j}; \chi_{j})M(s_{0, j}; \chi_{j})|| \cdot\mathds{1}(\mathcal{E}_{j})]
\ll\frac{LN}{\sqrt{\log\log T}}
\end{align*}
since $|\log|L(s_{0, j}; \chi_{j})M(s_{0, j}; \chi_{j})||\leq\log 2$ on the event $\mathcal{E}_{j}$.

On the other hand, by the boundedness of $f$ (i.e., $\|f\|_{\infty}\leq M$), the second term on the right of \eqref{equation 4 in AR24} is
\begin{align*}
&  \leq2M\sum_{j = 1}^{N}\mathfrak{P}_{T}(\mathcal{E}_{j}^{c})
\ll 2M\sum_{j = 1}^{N}\mathbb{E}_{T}[|L(s_{0, j}; \chi_{j})M(s_{0, j}; \chi_{j}) - 1|^{2}]
\\&=2M\sum_{j = 1}^{N}\frac{1}{T}\int^{2T}_{T}|1 - L(\sigma_{0} + \mi(U + \alpha_{j}), \chi_{j}) M(\sigma_{0} + \mi (U + \alpha_{j}))|^{2}dt
\ll M N (\log\log T)^{-K/K'}
\end{align*}
by Lemma \ref{lemma 2 in AR24}.
Finally, gathering everything together, we complete the proof.
%
%
%
\end{proof}

\section{Proof of Proposition \ref{Prop 3 in AR24}}\label{pf-prop-3}


In this section, we prove Proposition \ref{Prop 3 in AR24}. To begin, we let $M(s;\chi)$ and $\mathcal{P}(s)=\mathcal{P}(s;\chi)$ be defined as in \eqref{def-M} and \eqref{def of mathcal P}, respectively, and we introduce the following auxiliary series:
\begin{align*}
\mathcal{P}_{0}(s; \chi):=\sum_{2\leq n \leq 13}\frac{\Lambda(n)\chi(n)}{n^{s}\log n},\
\mathcal{P}_{1}(s; \chi):=\sum_{13< n\leq Y}\frac{\Lambda(n)\chi(n)}{n^{s}\log n}, \  \mbox{and}\ \mathcal{P}_{2}(s;\chi):=\sum_{Y<n\leq X}\frac{\Lambda(n)\chi(n)}{n^{s}\log n}.
\end{align*}
Similar to \cite{AR24}, we write
\begin{align}\label{def of usual P}
P(s) = P(s; \chi):=\Re\sum_{p\leq X}\frac{\chi(p)}{p^{s}}
= P_{0}(s; \chi) + P_{1} (s; \chi) + P_{2}(s; \chi),
\end{align}
where
$$
P_{0}(s; \chi):=\Re\sum_{p\leq 13}\frac{\chi(p)}{p^{s}},\ P_{1}(s;\chi):=\Re\sum_{13<p\leq Y}\frac{\chi(p)}{p^{s}}, \  \mbox{and}\ P_{2}(s, \chi):=\Re\sum_{Y<p\leq X}\frac{\chi(p)}{p^{s}}.
$$

 For the sake of convenience, we write $M(s)=M(s;\chi)$, $\mathcal{P}(s)=\mathcal{P}(s;\chi)$, $\mathcal{P}_{j}(s) = \mathcal{P}_{j}(s; \chi)$, and $P_{j}(s) = P_{j}(s; \chi)$ for $j = 0, 1, 2$. 
 We require the following lemma regarding the difference between $\log |M^{-1}(s)|$ and $\Re(\mathcal{P}(s))$  when $|P_{j}(s)|$ are relatively small.

\begin{lemma}\label{Lemma 4 in AR24}
Let $B\geq 1$ be an appropriate constant that will be determined later and $j = 0, 1$. On the events $\{|\mathcal{P}_{j}(s)|\leq\log\log T\}$, $\{|\mathcal{P}_{2}(s)|\leq B\log\log\log T\}$, $\{|P_{j}(s)|\leq \log\log T\}$, and $\{|P_{2}(s)|\leq B\log\log\log T\}$, we have
\begin{align*}
|\log |M^{-1}(s)| -\Re(\mathcal{P}(s))|\ll (\log\log T)^{-100}.
\end{align*}
\end{lemma}

\begin{proof}
Set $a_{0}(n) = 1$ if $n$ does not have any prime divisor $p>13$, and $n$ has at most $100\log\log T$ prime divisors $p\leq 13$; otherwise, we set $a_{0}(n)=0$. Set $\tilde{a}_{0}(n) = 1$ if $n$ does not have any prime divisor $p>13$, and $n$ has at least $100\log\log T$ prime divisors $p\leq 13$; otherwise, we set $\tilde{a}_{0}(n)=0$

Set $a_{1}(n) = 1$ if $n$ is composed of at most $100\log\log T$ primes $p$ satisfying $13< p\leq Y$; otherwise, we set $a_{1}(n) = 0$. Set $\tilde{a}_{1}(n) = 1$  composed of at least $100\log\log T$ primes $p$ satisfying $13< p\leq Y$; otherwise, we set $\tilde{a}_{1}(n) = 0$.

Let $A\geq 100$ be fixed. We set $a_{2}(n) = 1$ if $n$ is composed at most $A\log\log\log T$ primes $p$ satisfying $Y<p\leq X$; otherwise, we set $a_{2}(n) = 0$. Set $\tilde{a}_{2}(n) = 1$ if $n$ is composed at least $A\log\log\log T$ primes $p$ satisfying $Y<p\leq X$; otherwise, we set $\tilde{a}_{2}(n) = 0$.

We then have $M_{j}(s)$ and $\mathcal{E}_{j}(s)$, $j = 0, 1, 2$, as follows
\begin{align*}
M_{j}(s)= M_{j}(s; \chi):=\sum_{n}\frac{\mu(n)a_{j}(n)\chi(n)}{n^{s}}\quad\mbox{and}\quad \mathcal{E}_{j}(s) = \mathcal{E}_{j}(s; \chi):=\sum_{n}\frac{\mu(n)\tilde{a}_{j}(n)\chi(n)}{n^{s}}.
\end{align*}
Processing as in \cite[Eq.(15)]{AR24}, we see that
$
\exp(-\mathcal{P}_{0}(s)) = M_{0}(s) 
$
since $p\leq 13$ and the definition of $\tilde{a}_{0}(n)$. Therefore, we have
$
\Re(\mathcal{P}_{0}(s)) = \log |M^{-1}_{0}(s)|.
$
Similar to \cite[Eq. (15)]{AR24}, we obtain
$
\exp(-\mathcal{P}_{1}(s) ) = M_{1}(s) + \mathcal{E}_{1}(s).
$
For $\mathcal{E}_{1}(s)$, we process as in \cite[Lemma 4]{AR24} to obtain
\begin{align*}
\mathcal{E}_{1}(s)&\ll\exp(-100\log\log T)\exp\bigg(\sum_{ j\ge 1}\frac{e^{j}|P_{1}(js)|}{j}\bigg)
\\& = \exp(-100\log\log T)\exp\bigg(e|P_{1}(s)| + \frac{e^{2}}{2}|P_{1}(2s)| + \sum_{j\geq 3}\frac{e^{j}|P_{1}(js)|}{j}\bigg)
\\&\leq \exp\bigg(-100\log\log T + e\log\log T + \frac{e^{2}}{2}\log\log T + \sum_{j\geq 3}\frac{e^{j}|P_{1}(js)|}{j} \bigg).
\end{align*}
 By the definition of $P_{1}(s)$, for $\ell\geq 3$, we have 
\begin{align*}
|P_{1}(\ell s)|\leq\sum_{13<p\leq Y}\frac{1}{p^{\ell/2}}\leq\sum_{17\leq n\leq Y}\frac{1}{n^{\ell/2}}\leq\sum_{n\geq\sqrt{17}}\frac{1}{n^{\ell}}\leq\int_{\lfloor\sqrt{17}\rfloor - 1}^{\infty}\frac{1}{x^{\ell}}dx
=\frac{x^{-\ell + 1}}{-\ell + 1}\Big|^{\infty}_{3} = \frac{3^{-\ell + 1}}{\ell - 1}\leq\frac{3}{2}3^{-\ell}
\end{align*}
as $\ell\geq 3$. Therefore, we conclude that
$
|P_{1}(js)|\leq\frac{3}{2}3^{-j},
$
for all $j\geq 3$, which implies that
\begin{align*}
\sum_{j\geq 3}\frac{e^{j}|P_{1}(js)|}{j}
\leq\frac{3}{2}\sum_{j\geq 3}\frac{e^{j}3^{-j}}{j}
=\frac{3}{2}\sum_{j\geq 3}\frac{e^{j}e^{-j\log 3}}{j}
=\frac{3}{2}\sum_{j\geq 3}\frac{e^{j(1 - \log 3)}}{j}
\leq\frac{3}{2}\sum_{j\geq 3}e^{j(1 - \log 3)},
\end{align*}
which converges as $1 - \log 3<0$. Thus, we have
\begin{align*}
\mathcal{E}_{1}(s)
 \ll  \exp\bigg(\bigg(-100 + e + \frac{e^{2}}{2}\bigg)\log\log T\bigg)
=(\log T)^{-100 + e + \frac{e^{2}}{2}}
\ll (\log T)^{-90}
\end{align*}
by the numerical result $e + e^{2}/2\approx 6.413$. As the lower bound is obtained similarly, we proceed as in \cite[p. 3359]{AR24} to obtain
\begin{align*}
\Re (\mathcal{P}_{1}(s) ) = \log |M^{-1}_{1}(s)| + O((\log T)^{-80})
\end{align*}

%

We shall give $M_{2}$ a detailed calculation. Similar to \cite[Eq. (15)]{AR24}, we have
\begin{align}\label{15 in AR24}
\exp (-\mathcal{P}_{2}(s)) = M_{2}(s) + \mathcal{E}_{2}(s).
\end{align}
Moreover, 
\begin{align*}
\mathcal{E}_{2}(s) 
&= \sum_{n}\frac{\mu(n)\tilde{a}_{2}(n)}{n^{s}}
=\sum_{\ell>A\log\log\log T}(-1)^{\ell}\bigg(\sum_{p_{1}<p_{2}<\cdots<p_{\ell}}\frac{1}{(p_{1}\cdots p_{\ell})^{s}}\bigg)
\\&= \sum_{\ell>A\log\log\log T}(-1)^{\ell}\bigg(\sum_{\substack{m_{1}, \ldots, m_{k}, \ldots \\ m_{1} + 2m_{2} + \cdots + km_{k} + \cdots = \ell}} \prod_{ j\ge 1}\frac{(-P_{2}(js))^{m_{j}}}{m_{j}! j^{m_{j}}}\bigg)
\end{align*}
which is
$$\ll \sum_{m_{1}, \ldots, m_{k}, \ldots}e^{(-A\log\log\log T + m_{1} + 2m_{2} + \cdots + km_{k} + \cdots )}\prod_{ j\ge 1}\frac{|P_{2}(js)|^{m_{j}}}{m_{j}! j^{m_{j}}}
=e^{(-A\log\log\log T)}e^{\sum_{ j\ge 1}e^{j}|P_{2}(js)|/j}.
$$

Recall that we are currently on the event $\{|P_{2}(s)|\leq B\log\log\log T\}$. Therefore, 
$
e|P_{2}(s)|\leq eB\log\log\log T.
$
An application of Mertens' theorem leads
$
|P_{2}(2s)|\leq \log\log X - \log\log Y.
$
Recall the definition of $X$ and $Y$. A careful calculation yields
\begin{align*}
\log X = \frac{\log T}{K'\log\log\log T};
\end{align*}
hence, $\log\log X  = \log\log T - \log K' - \log\log\log\log T$;
similarly, $\log\log Y = \log\log T - \log K' - \log\log\log T$.
Putting everything together, we obtain
\begin{align*}
|P_{2}(2s)|
\leq \log\log X - \log\log Y
\leq\log\log\log T
\end{align*}
for sufficiently large $T$.

For $\ell\geq 3$, the series $\sum_{n = 1}^{\infty}\frac{1}{n^{\ell/2}}$ converges. Hence, $\sum_{p}\frac{1}{p^{\ell/2}}$ converges, where $p$ is over all of the primes. This implies that, as the tail terms,
$
|P_{2}(\ell s)|\leq 10^{-\ell}
$
for sufficiently large $T$, and so 
\begin{align*}
\sum_{j\geq 3}e^{j}\frac{|P_{2}(js)|}{j}
\leq\sum_{j\geq 3}\frac{e^{j}10^{-j}}{j}
=\sum_{j\geq 3}\frac{e^{j (1 - \log 10)}}{j} <\infty.
\end{align*}
Gathering everything together, we arrive at
\begin{align*}
\mathcal{E}_{2}(s)
&\ll \exp(-A\log\log\log T + eB\log\log\log T + \log\log\log T)
= (\log\log T)^{-A + eB + 1}.
\end{align*}
Now, we deduce a lower bound for $\mathcal{E}_{2}(s)$. Since we are on the event $\{|\mathcal{P}_{2}(s)|\leq B\log\log\log T\}$, we have
\begin{align*}
(\log\log T)^{-B}=\exp(-B\log\log\log T)\leq\exp(- \mathcal{P}_{2}(s)) = M_{2}(s) + \mathcal{E}_{2}(s)
\end{align*}
by \eqref{15 in AR24}. Now we shall choose $A$ and $B$ so that $-A + eB +1\leq -B$. Thus, we obtain $
(\log\log T)^{-B}\ll M_{2}(s).
$
Using \eqref{15 in AR24} again, we have 
\begin{align*}
\Re(\mathcal{P}_{2}(s)) 
= -\log |M_{2}(s) + \mathcal{E}_{2}(s)|= \log |M^{-1}_{2}(s)| + O\bigg(\frac{\mathcal{E}_{2}(s)}{M_{2}(s)}\bigg),
\end{align*}
which is
$$
\log  |M^{-1}_{2}(s)|  + O\bigg(\frac{ (\log\log T)^{-A + eB + 1}}{(\log\log T)^{-B}}\bigg)
=\log  |M^{-1}_{2}(s)| + O((\log\log T)^{-A +(e+1)B + 1}).
$$
Finally, choosing 
\begin{align}\label{the choice of A and B}
A = 400\quad \mbox{and} \quad B = 80,
\end{align} 
we then have
$
 O((\log\log T)^{-A +(e+1)B + 1} )
= O((\log\log T)^{-100}),
$
as desired.
\end{proof}

%
%

\begin{remark}
As may be noticed, differing from the setting in \cite[p. 3352]{AR24}, we split $P(s)$ into three terms. The reason is that we have to isolate the small prime so that the required estimate is reached. In his original proof, Roberts only obtain $(\log\log T)^{-1}$, which is larger than the claimed error $(\log\log T)^{-80}$. The term $P_{0}$ is introduced to fill the gap.
Furthermore, we would like to emphasise that such a modification is necessary. If one work with $ P_1(s)=\Re  \sum_{ p\leq Y}\frac{ 1}{p^{s}}$ as its original definition in \cite[p. 3352]{AR24},  then the bound
$
|P_{1}(\ell s)|\leq 10^{-\ell}
$
seems invalid as $\frac{1}{2^{\ell/2}}$ would appear. Indeed, by the calculation below, the upper bound would be $\le \frac{\sqrt{2}^{-\ell+1}}{\ell-1} $.  Unfortunately, this bound would not lead to the convergence for 
$
\sum_{j\geq 3}\frac{e^{j}|P_{1}(js)|}{j},
$
which was extremely crucial in \cite[p. 3359]{AR24}.

%
%
%
%
%
%

\end{remark}

Now, given ${\bf a} = (a_{1}, a_{2}, \ldots, a_{N})\in\mathbb{R}^{N}$ and $\alpha = (\alpha_{1}, \alpha_{2}, \ldots, \alpha_{N})\in\mathbb{R}^{N}$, we define
\begin{align}\label{6.1 of SLCT_Forum}
\mathcal{V}(T, \alpha_{i}, \alpha_{j}) = \min\Big(\log\log T, \log\Big(\frac{1}{|\alpha_{i} - \alpha_{j}|}\Big)\Big).
\end{align}
Let $\mathcal{D} = (\chi_{j})_{j = 1}^{N}$ be a sequence of primitive Dirichlet characters, and set $\delta_{i, j}$ as in \eqref{def-deltaij}.
Define
\begin{align}\label{6.2 of SCLT_Forum}
\mathcal{U}_{{\bf a}}(\alpha, \mathcal{D}, T) = (a^{2}_{1} + a^{2}_{2} + \cdots + a^{2}_{N} )\log\log T + 2\sum_{1\leq i<j\leq N}a_{i}a_{j}\delta_{i, j}\mathcal{V}(T, \alpha_{i}, \alpha_{j}).
\end{align}
Recall that Merterns' theorem for primitive Dirichlet characters $\chi$ modulo $q$ states that
\begin{align}
\sum_{p\leq x}\frac{\chi(p)}{p} = \delta(\chi)\log\log x + b(\chi) + h(x, q),
\end{align}
for some constant $b(\chi)$ depending on $\chi$, where $h(x, q) = O_{q}((\log x)^{-2})$, and
\begin{align}\label{6.4 of SCLT_Forum}
\delta(\chi) = 
\begin{cases}
1\quad\mbox{if}\ \chi \ \text{is principal};\\
0\quad\mbox{otherwise}.
\end{cases}
\end{align}

Now, we borrow the following useful estimate from the present authors' previous work \cite[Lemma 6.1]{HW-Forum}.
\begin{lemma}\label{Lemma 6.1 in SCLT_Forum}
Let $\lambda$ be a real number. For any Dirichlet character $\chi$ modulo $q$, one has
\begin{align}
\sum_{p\leq z}\frac{\chi(p)e^{\mi \lambda\log p}}{p} 
= \delta(\chi)\min\Big\{\log\log z, \log\Big(\frac{1}{|\lambda|}\Big)\Big\} + O_{q}(|\lambda| + 1),
\end{align}
where $\delta(\chi)$ is defined as in \eqref{6.4 of SCLT_Forum}
\end{lemma}

We shall introduce an auxiliary series
\begin{align}\label{P_{a, 0} series}
\mathcal{P}_{{\bf a}, 0}(s, Z):= \sum_{p\leq Z}\frac{a_{1}\chi_{1}(p)p^{-\mi \alpha_{1}} + \cdots + a_{N}\chi_{N}(p)p^{-\mi\alpha_{n}}}{p^{s}} 
\end{align}
and prove the following lemma regarding its moments. We remind the reader that for any ${\bf a} = (a_{j})_{j = 1}^{N}\in\mathbb{R}^{N}$, we denote by $\|{\bf a}\|_{1}$ the 1-norm of ${\bf a}$, which is introduced in \eqref{lp-norm}.
\begin{lemma}\label{Lemma 6.2 in SCLT_Forum}
Let $0<\varepsilon<1$ be fixed. Let $\alpha = (\alpha_{1}, \ldots, \alpha_{N})\in\mathbb{R}^{N}$ satisfy $|\alpha_{i} - \alpha_{j}| =  O((\log\log T)^{\epsilon})$. Assume that $k, \ell\in\mathbb{Z}$ are non-negative and with $Y^{k+\ell}\leq T$. Then for any ${\bf a} = (a_{j})_{j = 1}^{N}\in\mathbb{R}^{N}$, we have, for $k\neq\ell$,
\begin{align*}
\int^{2T}_{T}\mathcal{P}_{{\bf a}, 0}(\sigma_{0} +\mi t, Y)^{k}\overline{\mathcal{P}_{{\bf a}, 0} (\sigma_{0} + \mi t, Y)^{\ell}}dt
= O( k!\ell !((\|{\bf a}\|_1 Y)^{k + \ell})\ll_{{\bf a}, k, \ell} Y^{k + \ell},
\end{align*}
where the implied constant in the big-O term is absolute, and for $k = \ell$, we have
\begin{align*}
\int^{2T}_{T}|\mathcal{P}_{{\bf a}, 0}(\sigma_{0} + \mi t, Y)|^{2k}dt = k! T |\mathcal{U}_{{\bf a}}(\alpha, \mathcal{D}, T)|^{k} + O(g({\bf a}, k)T(\log\log T)^{k - 1 + \varepsilon})
+O((k!)^{2}(||{\bf a}||_1 Y)^{2k }  ),
\end{align*}
where $\mathcal{U}_{{\bf a}}(\alpha, \mathcal{D}, T)$ is defined as in \eqref{6.2 of SCLT_Forum}, and 
$
g({\bf a}, k) = (\|{\bf a}\|_{1})^{2k}2^{k}k!.
$
\end{lemma}

\begin{proof}
We begin the argument with the case of $k\neq\ell$ and recall some notation for the sake of completeness. Set $\psi(p)=\sum_{i=1}^{N}a_{i}\chi_i(p) p^{-\mi\alpha_{i}}$ and $
\Psi_k(n)=\prod_{j=1}^{r}\psi(p_{j})^{h_{j}}$ for $n = p_{1}^{h_{1}}\cdots p_{r}^{h_{r}}$ with distinct $p_j$ and $h_{1}+\cdots +h_{r}=k$. Write
\begin{align*}
\mathcal{P}_{{\bf a}, 0}(\sigma_{0}+\mi t, Z)^{k} = \sum_{n}\frac{b_{k}(n)\Psi_{k}(n)}{n^{\sigma_{0}+\mi t}},
\end{align*}
where 
$
b_{k}(n)=\frac{k!}{h_{1}!\cdots h_{r}!}$ if $n=\prod_{j=1}^{r} p_{j}^{h_{j}}$  with  $p_{1}<\cdots< p_{r}\leq Z$ and  $\sum_{j=1}^{r}h_{j}=k$; otherwise,  $b_{k}(n)=0$.


Then we proceed as in the proof of \cite[Lemma 6.2]{HW-Forum} to \cite[Eq. (6.9)]{HW-Forum}:
\begin{align}\label{eq 6.9 in SCLT_Forum}
\begin{split}
&\int^{2T}_{T}\mathcal{P}_{{\bf a}, 0}(\sigma_{0} + \mi t, Z)^{k}\overline{\mathcal{P}_{{\bf a}, 0}(\sigma_{0} + \mi t, Z)^{\ell}}dt
\\&= T \sum_{n}\frac{b_{k}(n)b_{\ell}(n)\Psi_{k}(n)\overline{\Psi_{\ell}(n)}}{n^{2\sigma_{0}}}
+O\Big(\sum_{n\neq m}\frac{b_{k}(n)b_{\ell}(n)|\Psi_{k}(n)\overline{\Psi_{\ell}(m)}|}{(nm)^{\sigma_{0} - \frac{1}{2}}}\Big).
\end{split}
\end{align}
By the definition of $\Psi_{k}(n)$, we see that $|\Psi_{k}(n)|\leq (|a_{1}| + \cdots + |a_{N}|)^{k}$. Therefore, the big-O term in \eqref{eq 6.9 in SCLT_Forum} is at most
\begin{align*}
\sum_{n \neq m}b_{k}(n)b_{\ell}(m)(|a_{1}|+ \cdots +|a_{N}|)^{k + \ell}
\ll  k!\ell ! (Z(|a_{1}| + \cdots + |a_{N}|))^{k + \ell}
\ll_{{\bf a}, k, \ell} Z^{k + \ell},
\end{align*}
where the implied constant in the first big-O term is absolute. For the case $n = m$, we see from the definition that $b_{k}(n)b_{\ell}(n) = 0$ if $k\neq \ell$. Plugging $Z = Y$,  we derived the first part of the lemma.

It remains to consider the case $k=\ell$. Similar to \cite[Lemma 6.2]{HW-Forum}, we have
\begin{align}\label{P^{2}:k=l} 
 \begin{split}
\int^{2T}_{T} |\mathcal{P}_{{\bf a}, 0}(\sigma_{0}+\mi t, Z)|^{2k}dt
=T \bigg|\sum_{n}\frac{b_{k}(n)^2\Psi_k(n)\overline{\Psi_{k}(n)}}{n^{2\sigma_{0}}} \bigg|
+O_{{\bf a}, k} (Z^{2k}),
 \end{split}
\end{align}
where the implied constant can be taken as $(k!)^{2}(|a_{1}| + \cdots + |a_{N}|)^{2k}$.

For  $n=\prod_{j=1}^{r}p_{j}^{h_{j}}$, we write
\begin{align}\label{psipsibar_{n}}
 \begin{split}
\Psi_k(n)\overline{\Psi_k(n)}
&=\prod_{j=1}^{r}\Big((a_{1}^{2}|\chi_1(p_j)|^2 +\cdots+a_{N}^{2} |\chi_N(p_j)|^2)\\
&+ \sum_{1\le i<i'\le N}a_{i}a_{i'} \chi_i \bar{\chi}_{i'}(p_j) p_{j}^{-\mi (\alpha_{i}-\alpha_{i'})}+\sum_{1\le i<i'\le N}a_{i}a_{i'} \bar{\chi}_{i} \chi_{i'}(p_j) p_{j}^{-\mi(\alpha_{i'}-\alpha_{i})}
\Big)^{h_{j}}.
  \end{split}
\end{align}
Therefore, by the fact that $\sum_j h_{j}=k$, we deduce
\begin{align*}
\left|\Psi_k(n)\overline{\Psi_k(n)}\right|
\leq\prod_{j = 1}^{r} \left((|a_{1}| + \cdots + |a_{N}|)^{2}\right)^{h_{j}} = (|a_{1}|+ \cdots + |a_{N}|)^{2k}.
\end{align*}
Hence, we conclude that the contribution of non-square-free $n$ to the sum in \eqref{P^{2}:k=l} is 
\begin{align}\label{non-square-free}
\ll(|a_{1}| + \cdots + |a_{N}|)^{2k} \sum_{n\text{ non-square-free}}\frac{b_{k}(n)^{2}}{n^{2\sigma_{0}}}.
\end{align}

Note that for a non-square-free $n$ that has $k$ prime factors, counted with multiplicity, there is a prime $p$ such that $p^2\mid n$. As $m=n/p^2$ has $k-2$ prime factors, counted with multiplicity, without loss of generality, writing  $n=\prod_{j=1}^{r} p_{j}^{h_{j}}$ and taking $p=p_1$,  we have
\begin{align*}
b_{k}(n)=\frac{k!}{h_{1}!\cdots h_{r}!}
\quad\text{and}\quad
b_{k - 2}(m) = b_{k-2}(n/p^2)=\frac{(k-2)!}{( h_{1}-2)!\cdots h_{r}!}.
\end{align*}
This observation leads to
$
b_k(n)\le k(k-1) b_{k-2}(n/p^2).
$
Hence, we conclude that
\begin{align}\label{non-square-free-fixed}
\sum_{n\text{ non-square-free}}\frac{b_{k}(n)^{2}}{n^{2\sigma_{0}}}
\ll k! k^2\Big( \sum_{p\le Z} \frac{1}{p^2}\Big)  \Big( \sum_{\substack{p\le Z }} \frac{1}{p}\Big)^{k-2}
\ll k! k^2(\log\log Z +O(1))^{k-2}.
\end{align}

Next, by the definition of $\sigma_{0}$, we have
$
2\sigma_{0} = 1 + \frac{2W}{\log T}.
$ 
Therefore, by the mean value theorem, we obtain
\begin{align*}
\sum_{p\leq Z}\frac{\chi_{i}\bar{\chi}_{j}(p)p^{-\mi (\alpha_{i} - \alpha_{j})}}{p^{2\sigma_{0}}}
- \sum_{p\leq Z}\frac{\chi_{i}\bar{\chi}_{j}(p)p^{-\mi (\alpha_{i} - \alpha_{j})}}{p}
\ll\frac{W}{\log T}\sum_{p\leq Z}\frac{\log p}{p}.
\end{align*}
By using \cite[Eq. (6.5)]{HW-Forum} and the choice of $W$, we further deduce that
\begin{align*}
\frac{W}{\log T}\sum_{p\leq Z}\frac{\log p}{p}
\ll\frac{(\log\log\log T)^{2}}{\log T}\log Z.
\end{align*}
Therefore, we conclude that
\begin{align*}
\sum_{p\leq Z}\frac{\chi_{i}\bar{\chi}_{j}(p)p^{-\mi (\alpha_{i} - \alpha_{j})}}{p^{2\sigma_{0}}}
= \sum_{p\leq Z}\frac{\chi_{i}\bar{\chi}_{j}(p)p^{-\mi (\alpha_{i} - \alpha_{j})}}{p} + O \Big(\frac{(\log\log\log T)^{2}}{\log T}\log Z\Big).
\end{align*}
Applying Lemma \ref{Lemma 6.1 in SCLT_Forum} with $\lambda = \alpha_{i} - \alpha_{j}$, we have
\begin{align}\label{equation next to (6.15)}
\begin{split}
&\sum_{p\leq Z}\frac{ \chi_i\bar{\chi}_j(p)p^{-\mi(\alpha_{i}-\alpha_{j})}}{p^{2\sigma_{0}}}
\\&=\delta_{i,j} \min\Big(\log\log Z , \log\Big(\frac{1}{|\alpha_{i} - \alpha_{j}|}\Big) \Big) + O_{q_i,q_j}(|\alpha_{i} - \alpha_{j}| +1) + O \Big(\frac{(\log\log\log T)^{2}}{\log T}\log Z\Big).
\end{split}
\end{align}
Choosing $Z = Y = T^{1/K^{\prime}\log\log T}$ in \eqref{equation next to (6.15)}, we deduce
\begin{align}\label{equation next to (6.15)_new}
\begin{split}
\sum_{p\leq Y}\frac{ \chi_i\bar{\chi}_j(p)p^{-\mi(\alpha_{i}-\alpha_{j})}}{p^{2\sigma_{0}}}
&=\delta_{i,j} \min\Big(\log\log Y , \log\Big(\frac{1}{|\alpha_{i} - \alpha_{j}|}\Big) \Big) + O_{q_i,q_j}(|\alpha_{i} - \alpha_{j}| +1) + O \Big(\frac{(\log\log\log T)^{2}}{\log\log T}\Big)
\\&=\delta_{i, j}\mathcal{V}(T, \alpha_{i}, \alpha_{j})+ O_{q_{i}, q_{j}}((\log\log T)^{\varepsilon}),
\end{split}
\end{align}
where the last equality follows from the definition of $\mathcal{V}(T, \alpha_{i}, \alpha_{j})$ in  \eqref{6.1 of SLCT_Forum} and the assumption that $|\alpha_{i} - \alpha_{j}| =  O((\log\log T)^{\epsilon})$.


For square-free integers $n$, we process as in \cite[Lemma 6.2]{HW-Forum} to \cite[Eq. (6.15)]{HW-Forum}.  Plugging \eqref{equation next to (6.15)_new} into \cite[Eq. (6.15)]{HW-Forum} and continuing the rest of the proof as in \cite[p. 17]{HW-Forum}, we see that the sum $\sum_{n}\frac{b_{k}(n)^2\Psi_k(n)\overline{\Psi_{k}(n)}}{n^{2\sigma_{0}}}$ over square-free $n$ equals
\begin{align*}
k!\Big((a^{2}_{1}+\cdots +a_{N}^{2})(\log\log T + O((\log\log T)^{\varepsilon}))
+ 2\sum_{1\leq i<j\leq N}a_{i}a_{j}\delta_{i,j}(\mathcal{V}(T, \alpha_{i}, \alpha_{j}) + O((\log\log T)^{\varepsilon}))\Big)^{k}.
\end{align*}


Recalling the definition of  $\mathcal{U}_{\bf a}(\alpha, \mathcal{D}, T)$ in \eqref{6.2 of SCLT_Forum}, we see that   the sum $\sum_{n}\frac{b_{k}(n)^2\Psi_k(n)\overline{\Psi_{k}(n)}}{n^{2\sigma_{0}}}$ over square-free $n$ is 
\begin{align*}
\begin{split}
&k!\Big(\mathcal{U}_{\bf a}(\alpha, \mathcal{D}, T) +O((|a_{1}| + \cdots + |a_{N}|)^{2}(\log\log T)^{\varepsilon})\Big)^{k}\\
&=   k! \mathcal{U}_{\bf a}(\alpha, \mathcal{D}, T)^k
+ O\Big( k! (|a_{1}| + \cdots + |a_{N}|)^{2k} \sum_{j = 1}^{k}{k\choose j} (\log\log T)^{k - 1 + \varepsilon}\Big)\\
&= k!\mathcal{U}_{\bf a}(\alpha, \mathcal{D}, T)^{k} + O\Big( (|a_{1}| + \cdots + |a_{N}|)^{2k}2^{k}k!(\log\log T)^{k - 1 + \varepsilon}\Big), 
\end{split}
\end{align*}
where the implied constants are \emph{independent} of ${\bf a}$ and $k$.

%
%

To summerise, we see that \eqref{P^{2}:k=l} is
\begin{align*}
&\ll T \bigg|\sum_{n}\frac{b_{k}(n)^2\Psi_k(n)\overline{\Psi_{k}(n)}}{n^{2\sigma_{0}}} \bigg|+((k!)^{2}(|a_{1}| + \cdots + |a_{N}|)^{2k})(Z^{2k})
\\& = T\bigg|\sum_{n\ \text{ non-square-free}}\frac{b_{k}(n)^2\Psi_k(n)\overline{\Psi_{k}(n)}}{n^{2\sigma_{0}}} + \sum_{n \ \text{square-free}}\frac{b_{k}(n)^2\Psi_k(n)\overline{\Psi_{k}(n)}}{n^{2\sigma_{0}}} \bigg|+((k!)^{2}(|a_{1}| + \cdots + |a_{N}|)^{2k})(Z^{2k})
\\& = T \bigg|\sum_{n\ \text{square-free}}\frac{b_{k}(n)^2\Psi_k(n)\overline{\Psi_{k}(n)}}{n^{2\sigma_{0}}} \bigg| + Tk! k^{2}(\log\log Z +O(1))^{k-2}+((k!)^{2}(|a_{1}| + \cdots + |a_{N}|)^{2k})(Z^{2k})
\\&= k!T\mathcal{U}_{\bf a}(\alpha, \mathcal{D}, T)^{k} + O\Big( (|a_{1}| + \cdots + |a_{N}|)^{2k}2^{k}k!T(\log\log T)^{k - 1 + \varepsilon}\Big)
\\&+ Tk! k^{2}(\log\log Z +O(1))^{k-2}+((k!)^{2}(|a_{1}| + \cdots + |a_{N}|)^{2k})(Z^{2k}).
\end{align*}
Finally, by \eqref{P^{2}:k=l} and \eqref{non-square-free} with $Z =  Y = T^{1/K^{\prime}\log\log T}$ as well as the above discussion, we establish the second estimate for the lemma.
\end{proof}

\begin{remark}
We comment on some differences between Lemma \ref{Lemma 6.2 in SCLT_Forum} and \cite[Lemma 6.2]{HW-Forum}. The differences are the upper bounds in the series \eqref{P_{a, 0} series} and the requirement of shifting $|\alpha_{i} - \alpha_{j}|$. Both changes reflect the purpose of characterising the rate of convergence. In particular, in order to trace the contribution of the error, we introduce a dummy variable $Z$ in the series \eqref{P_{a, 0} series}. As shall be seen later, we will make use of such flexibility to deduce the corresponding estimate. Moreover, we precisely track the dependence of the error on ${\bf a}$ and $k$ for later use. 

We shall emphasise that when one employs the method of moments, those ${\bf a}$ and $k$ are \emph{fixed} and hence can be treated as absolute constants (see, e.g., \cite[Lemma 6.2]{HW-Forum}). Later, we will follow Roberts' trick \cite[Lemma 7]{AR24} picking an appropriate $k$, and such an explicitness will play a crucial role there.
\end{remark}

The following lemma gives the tail estimate of the ``random variables'' $|P_{1}(s, \chi)|$ and $|P_2(s, \chi)|$.

\begin{lemma}\label{lemma 7 in AR24}
We have
\begin{align}
\mathfrak{P}_{T}(|P_{1}(s, \chi)|>\log\log T)&\ll \frac{1}{\log T}, \label{P_{1} estimate in lemma 7}
\\\mathfrak{P}_{T}(|P_{2}(s, \chi)|>B\log\log\log T)&\ll\frac{1}{(\log\log T)^B}, \label{P_{2} estimate in lemma 7}
\end{align}
where $B$ is the same parameter introduced in Lemma \ref{Lemma 4 in AR24}.
\end{lemma}

\begin{proof}
In light of the argument from  \eqref{eq 6.9 in SCLT_Forum} to \eqref{non-square-free}, for
$
\mathcal{P}_{\chi}(s,z,Z) = \sum_{ z<p\le Z}\frac{ \chi(p)}{p^{s}},
$
we have
\begin{align}\label{exp}
\begin{split}
\int^{2T}_{T}\mathcal{P}_{\chi}(\sigma_{0} + \mi t,  z,Z)^{k}\overline{\mathcal{P}_{\chi}(\sigma_{0} + \mi t,  z,Z)^{\ell}}dt
= T \sum_{n}\frac{c_{k}(n)c_{\ell}(n)\chi(n)\overline{\chi(n)}}{n^{2\sigma_{0}}}
+O(k!\ell !Z^{k + \ell}),
\end{split}
\end{align}
where 
$
c_{k}(n)=\frac{k!}{h_{1}!\cdots h_{r}!}$ if $n=\prod_{j=1}^{r} p_{j}^{h_{j}}$  with  $z<p_{1}<\cdots< p_{r}\leq Z$ and  $\sum_{j=1}^{r}h_{j}=k$; otherwise,  $c_{k}(n)=0$. Similarly, it follows  from the definition that $c_{k}(n)c_{\ell}(n) = 0$ if $k\neq \ell$.  In addition, we also have
\begin{align*}
 \sum_{n\text{ non-square-free}}\frac{c_{k}(n)^{2}  \chi_0(n)}{n^{2\sigma_{0}}}
\ll k! k^2\Big(\sum_{z<p\le Z}\frac{1}{p}  \Big)^{k - 2},
\end{align*}
where the implied constants are absolute. Moreover,
\begin{align*}
\bigg| \sum_{n\text{ square-free}}\frac{c_{k}(n)^{2}\chi_0(n)}{n^{2\sigma_{0}}}\bigg|
\le \sum_{n\text{ square-free}}\frac{c_{k}(n)^{2}}{n}
= k! \Big(\sum_{z<p\le Z}\frac{1}{p} \Big)^{k}.
\end{align*}
Hence, we arrive at
\begin{align*}
\int^{2T}_{T}(\Re(\mathcal{P}_{\chi}(\sigma_{0} + \mi t,  z,Z)))^{2k}dt
&=\int^{2T}_{T}\frac{1}{2^{2k}}\Big(\mathcal{P}_{\chi}(\sigma_{0} + \mi t,  z,Z)+\overline{\mathcal{P}_{\chi}(\sigma_{0} + \mi t,  z,Z)}\Big)^{2k}dt
\\
&= \frac{1}{2^{2k}}\sum_{\ell = 0}^{2k}{2k \choose\ell}\int^{2T}_{T}\mathcal{P}_{\chi}(\sigma_{0} + \mi t,  z,Z)^{\ell}\overline{\mathcal{P}_{\chi}(\sigma_{0} + \mi t,  z,Z)^{2k-\ell}}dt\\
&=\frac{1}{2^{2k}}{2k\choose k}\int^{2T}_{T}|\mathcal{P}_{\chi}(\sigma_{0} + \mi t,  z,Z)|^{2k}dt
 +O( (2k)! k Z^{2k}),
\end{align*}
where the last big-O term follows from the estimate
\begin{align*}
\sum_{\substack{\ell = 0 \\ \ell\neq k}}^{2k}{2k \choose\ell}\int^{2T}_{T}\mathcal{P}_{\chi}(\sigma_{0} + \mi t,  z,Z)^{\ell}\overline{\mathcal{P}_{\chi}(\sigma_{0} + \mi t,  z,Z)^{2k-\ell}}dt
&\ll \sum_{\substack{\ell = 0 \\ \ell\neq k}}^{2k}{2k \choose\ell}(2k - \ell)! (\ell !) Z^{2k}\le   Z^{2k} \sum_{\ell = 0 }^{2k} (2k)!
.
\end{align*} 

It then follows from the above discussion that
\begin{align*}
&\frac{1}{T}\int^{2T}_{T}(\Re(\mathcal{P}_{\chi}(\sigma_{0} + \mi t,  z,Z)))^{2k}dt 
\\&\ll 2^{-2k}{2k\choose k}k!
\Big(\sum_{z<p\le Z}\frac{1}{p} \Big)^{k} 
 +2^{-2k}{2k\choose k}k! k^2 \Big(\sum_{z<p\le Z}\frac{1}{p}\Big)^{k-2} 
 +  \frac{(2k)! k }{2^{2k}}\frac{Z^{2k}}{T}.
\end{align*}
Finally, instead of using \eqref{equation next to (6.15)} (and the mean value theorem), we shall use Mertens' theorem 
\begin{align*}
\sum_{p\leq z}\frac{1}{p}
=\log\log z +O(1)
\end{align*}
directly to deduce
$$
\sum_{p\leq Y}\frac{ 1}{p} 
=  \log\log Y +O(1)
=\log\log T -   \log\log\log T  +O(1)
\le \log\log T,
$$
for all sufficiently large $T$, as well as
 \begin{align*}
 \sum_{Y<p\leq X}\frac{ 1}{p} 
=  \log\log X-\log\log Y +O(1)
=\log\log\log T -  \log \log\log\log T  +O(1) \le \log\log\log T .
\end{align*}

To summarise, we have  proved the following (cf. \cite[p. 3363, Eq. (20)]{AR24}):
\begin{equation}\label{P1-est'}
\mathbb{E}_T[ |P_1(s;\chi)|^{2k} ] \ll \frac{(2k)!}{k! 2^k} \Big(\sqrt{ \oh \log\log T} \Big)^{2k}
+ \frac{(2k)! k^2}{k! 2^k}  \Big(\sqrt{ \oh \log\log T} \Big)^{2k-4} 
+ O\Big(\frac{(2k)!k}{2^{2k}}\frac{Y^{4k}}{T} \Big),
\end{equation}
and that
\begin{equation}\label{P2-est'}
\mathbb{E}_T[ | P_2(s;\chi)|^{2k} ] \ll \frac{(2k)!}{k! 2^k} \Big(\sqrt{ \oh \log\log\log T} \Big)^{2k}
+ \frac{(2k)! k^2}{k! 2^k}  \Big(\sqrt{ \oh \log\log\log T} \Big)^{2k-4} 
+ O\Big(\frac{(2k)!k}{2^{2k}} \frac{X^{4k}}{T} \Big).
\end{equation}

Next, we will use the Chebyshev inequality to obtain
\begin{align}\label{strategy explanation}
\mathfrak{P}_{T}(|P_{1}(s, \chi)|>\log\log T)
\leq \frac{\mathbb{E}_T[ |P_1(s;\chi)|^{2k} ]}{(\log\log T)^{2k}}
\end{align}
for any $k$. Applying \eqref{P1-est'}  in \eqref{strategy explanation} and then choosing $k$ as a function in $T$ carefully, we shall obtain \eqref{P_{1} estimate in lemma 7}. We shall emphasise that the trick in \eqref{strategy explanation} is the very reason why we have to make all the big-O terms in Lemma \ref{Lemma 6.2 in SCLT_Forum} \emph{independent} of $k$. On the other hand, the $k$-dependence of the big-O terms  in \cite[Eqs. (18)  and  (19)]{AR24} seems not quite clear. 
Moreover, comparing \eqref{P1-est'} and \eqref{P2-est'} to their counterparts in \cite{AR24}, one sees the differences: we keep the lower order terms (the second term) while Roberts dropped it (e.g., \cite[Eq.(20)]{AR24}). Nevertheless, the lower order terms are indispensable: when $k\le \mathcal{C} \log\log T$ for some generic constant $\mathcal{C}$, the 2nd term in \eqref{P1-est'} is
\begin{align}\label{indispensable-2nd-term}
\le  \mathcal{C}^2 \frac{(2k)!}{k! 2^k} \Big(\sqrt{ \oh \log\log T} \Big)^{2k},
\end{align}
which is the same order as the main term. 

Now, we proceed to the proof. As mentioned, when $k\le \mathcal{C} \log\log T$, we have \eqref{indispensable-2nd-term}.
Then by  the Stirling's formula \eqref{Stirling formula}, we have
$$
 \frac{(2k)!k}{2^{2k}} \ll  \frac{k}{2^{2k}}   \sqrt{k}  \Big(\frac{2k}{e}\Big)^{2k}
 \ll k^{2k}.
$$ 
Thus, for $k\le Y$, we have
\begin{align}\label{6k argument}
  \frac{(2k)!k}{2^{2k}}   \frac{Y^{4k}}{T} \le   \frac{Y^{6k}}{T},
\end{align}
which implies that the 3rd term in \eqref{P1-est'} is
$$
\ll    \frac{(2k)!k}{2^{2k}}  \frac{Y^{4k}}{T} 
\ll  \frac{Y^{6k}}{T}
\le  \frac{T^{(6 \mathcal{C} \log\log T  )/ K' \log\log T}}{T} \ll 1.
$$
As $\frac{(2k)!}{k! 2^k}  \ge 1$, we conclude that
\begin{equation}\label{P1-est-upper-bd}
\mathbb{E}_T[ |P_1(s;\chi)|^{2k} ] \ll  (1+ \mathcal{C}^2) \frac{(2k)!}{k! 2^k} \Big(\sqrt{ \oh \log\log T} \Big)^{2k}
\end{equation}
provided that $\frac{6\mathcal{C}}{K'}\le 1$. With \eqref{P1-est-upper-bd} in hands, we proceed as in  \cite[Proof of Lemma 7]{AR24} with $k =\floor{ \log\log T}$ to obtain \eqref{P_{1} estimate in lemma 7}.


%
%
%
%
%

We now proceed to the proof of \eqref{P_{2} estimate in lemma 7}. For $T$ sufficiently large so that $B\log\log\log T\geq 2$, taking $k = \floor{B\log\log\log T}$ with $6B<K'$, we see that $ k\leq X$. Therefore, by the Chebyshev inequality, \eqref{P2-est'}, and \eqref{6k argument}, we have
\begin{align}\label{goal in lemma 7}
\begin{split}
&\mathfrak{P}_{T}(|P_{2}(s, \chi)|>B\log\log\log T)
\leq\Big(\frac{1}{B\log\log\log T}\Big)^{2k}\mathbb{E}_{T}[|P_{2}(s; \chi)|^{2k}]
\\&\ll\Big(\frac{1}{B\log\log\log T}\Big)^{2k}\Big( \frac{(2k)!}{k! 2^k} \Big(\sqrt{ \oh \log\log\log T} \Big)^{2k}
+ \frac{(2k)! k^2}{k! 2^k}  \Big(\sqrt{ \oh \log\log\log T} \Big)^{2k-4} 
+ O\Big(\frac{X^{6k}}{T} \Big)\Big)
\end{split}
\end{align}

The error term is
\begin{align*}
\ll\Big(\frac{1}{B\log\log\log T}\Big)^{2k}\cdot\Big(\frac{X^{6k}}{T}\Big)
\ll\Big(\frac{1}{B\log\log\log T}\Big)^{2B\log\log\log T}\cdot\Big(\frac{T^{\frac{6B}{K'}}}{T}\Big)
\ll (\log\log T)^{-B}
\end{align*}
as desired.

When $k\leq  \floor{B\log\log\log T}$, the terms
\begin{align*}
 \frac{(2k)!}{k! 2^k} \Big(\sqrt{ \oh \log\log\log T} \Big)^{2k}
\quad\mbox{and}\quad \frac{(2k)! k^2}{k! 2^k}  \Big(\sqrt{ \oh \log\log\log T} \Big)^{2k-4} 
\end{align*}
almost have same order. Therefore, when taking $k = \floor{B\log\log\log T}$ with $6B<K'$ in \eqref{goal in lemma 7}, we see that the main term is 
\begin{align*}
\ll_{B}\Big(\frac{1}{B\log\log\log T}\Big)^{2k}\Big( \frac{(2k)!}{k! 2^k} \Big(\sqrt{ \oh \log\log\log T} \Big)^{2k}
\Big)
&\ll\Big(\frac{1}{B\log\log\log T}\Big)^{2k}\Big(\frac{(2k)^{2k}}{k^{k}2^{k}}\frac{e^{-2k}}{e^{-k}}\Big(\sqrt{\oh\log\log\log T}\Big)^{2k}
\Big)
\\&\ll(\log\log T)^{-B},
\end{align*}
where the second $\ll$ follows from Stirling's formula \eqref{Stirling formula}.
%
\end{proof}

With all the preparation above, we are finally in a position to prove Proposition \ref{Prop 3 in AR24}.
\begin{proof}[Proof of Proposition \ref{Prop 3 in AR24}]
A routine calculation gives
\begin{align*}
d_{\mathcal{D}}({\bf M}_{T}, {\bf Q}_{T})
\leq\sup_{f\in\mathcal{L}}\mathbb{E}_{T}[|f({\bf M}_{T}) - f({\bf Q}_{T})|].
\end{align*}
In the same spirit as in \eqref{equation 4 in AR24}, we split the expectation into cases that depend on values of $P_{1}(s)$ and $P_{2}(s)$, and then make use of the properties of $f$ as well as the Chebyshev inequality to obtain
\begin{align}
d_{\mathcal{D}}({\bf M}_{T}, {\bf Q}_{T})
\leq\sup_{f\in\mathcal{L}}\mathbb{E}_{T}[|f({\bf M}_{T}) - f({\bf Q}_{T})|]
\ll L \mathcal{E}_1 + M \mathcal{E}_2,
\end{align}
where
\begin{align}\label{eq 13 in AR24}
\begin{split}
 \mathcal{E}_1
 & = \mathbb{E}\left[\|{\bf M}_{T} - {\bf Q}_{T}\|_{1}\cdot\mathds{1}\big(\cap_{j = 1}^{N}\{ |P_{1}(s_{0, j})|\leq\log\log T, |P_{2}(s_{0. j})|\leq\log\log\log T  \}\big)\right]
\\
& \le \sum_{j = 1}^{N}\mathbb{E}\left[|\mathcal{M}_{j} - \mathcal{Q}_{j}|\cdot\mathds{1}\big(\{ |P_{1}(s_{0, j})|\leq\log\log T, |P_{2}(s_{0, j})|\leq\log\log\log T \} \big)\right]
 \end{split}
\end{align} 
and
 \begin{align*}
  \mathcal{E}_2 = \sum_{j = 1}^{N} \mathfrak{P}_{T}(|P_{1}(s_{0, j})|>\log\log T) +
  \sum_{j = 1}^{N} \mathfrak{P}_{T}(|P_{2}(s_{0, j})|>\log\log\log T)  .
\end{align*}
Here, $s_{0, j} = \sigma_{0} + \mi (U + \alpha_{j})$, and we note that $|P_{0}(s_{0, j})|\leq 7$ for all $j$.

Recall the definitions of  $\mathcal{M}_{j}$ and $\mathcal{Q}_{j}$ from \eqref{def of mathcal M} and \eqref{def of mathcal Q}, respectively.  We then infer from Lemma \ref{Lemma 4 in AR24} and \eqref{eq 13 in AR24} that $\mathcal{E}_{1}\ll (\log\log T)^{-80}$. For $\mathcal{E}_{2}$, we use Lemma \ref{lemma 7 in AR24} to obtain $\mathcal{E}_{2}\ll (\log\log T)^{-B}$. Lastly, the required estimate follows from the choice of $B$ in \eqref{the choice of A and B}.
\end{proof}

\section{Proofs of Propositions \ref{Prop 4 in AR24} and \ref{Prop 5 in AR24}}\label{pf-prop-4&5}
In this section, we prove Propositions \ref{Prop 4 in AR24} and \ref{Prop 5 in AR24}. We first Proposition \ref{Prop 4 in AR24}.
\begin{proof}[Proof of Proposition \ref{Prop 4 in AR24}]
It follows from \eqref{Dudley ineq} that
\begin{align*}
&d_\mathcal{D}  ( {\bf Q}_{T}, {\bf R}_{T})
\leq\frac{L}{\sqrt{\oh\log\log T}}\sum_{j = 1}^{N}\mathbb{E}[| \Re(\mathcal{P}(\sigma_{0} + \mi (U+\alpha_{j}))) -  P(\sigma_{0} + \mi (U+\alpha_{j}));\chi_j)|].
\end{align*}
Recall the definition of $\mathcal{P}(\sigma_{0} + \mi (U+\alpha_{j}))$ from \eqref{def of mathcal P} and the definition of $P(\sigma_{0} + \mi (U+\alpha_{j}));\chi_j)$ from \eqref{def of usual P}, respectively. We have
\begin{align*}
\Big| \Re(\mathcal{P}(\sigma_{0} + \mi (U+\alpha_{j}))) -  P(\sigma_{0} + \mi (U+\alpha_{j}));\chi_j)\Big|
\leq \bigg|\sum_{\substack{2\leq p^{k}\leq X \\ k\geq 3}}\frac{\log p\chi(p^{k})}{p^{ks}k\log p}\bigg|
+\bigg|\sum_{2\leq p^{2}\leq X}\frac{\chi(p^{2})}{2p^{2(\sigma_{0} + \mi(U + \alpha_{j}))}}\bigg|.
\end{align*}
The former term is $\ll 1$ by  the third last displayed equation in \cite[p. 18]{HW-Forum}. Moreover,
\begin{align*}
\mathbb{E}\bigg[\bigg|\sum_{2\leq p^{2}\leq X}\frac{\chi(p^{2})}{2p^{2(\sigma_{0} + \mi(U + \alpha_{j}))}}\bigg|\bigg]\ll 1
\end{align*}
by the last two displayed equations in \cite[p. 18]{HW-Forum} with $a_{j} = 1$. Hence, we complete the proof.
\end{proof}

Finally, we prove Proposition \ref{Prop 5 in AR24}.
\begin{proof}[Proof of Proposition \ref{Prop 5 in AR24}]
Recall from \eqref{MTq} that
\begin{align*}
\mathfrak{M}_{T, \chi_j} = P_{1}(2\sigma_{0};\chi_j \bar{\chi}_j):= \sum_{13<p\leq Y}\frac{\chi_j\bar{\chi}_j(p)}{p^{2\sigma_{0} }}  = \sum_{p\leq Y}\frac{\chi_j\bar{\chi}_j(p)}{p^{2\sigma_{0} }} -  \sum_{2\leq p\leq 13}\frac{\chi_j\bar{\chi}_j(p)}{p^{2\sigma_{0} }}
=  \sum_{p\leq Y}\frac{1}{p^{2\sigma_{0}}} + O(1).
\end{align*}
By the mean value theorem (and the choice of $W$), we further have
\begin{align*}
\mathfrak{M}_{T, \chi_j}
=\sum_{p\leq Y}\frac{1}{p} + O\Big(\frac{(\log\log\log T)^{2}}{\log T}\log Y\Big) + O(1)
=\log\log Y  + O\Big(\frac{(\log\log\log T)^{2}}{\log T}\log Y\Big) + O(1).
\end{align*}
Recall that $Y = T^{1/(K'\log\log T)}$. We conclude that
\begin{align}\label{calculation of frak M}
\begin{split}
\mathfrak{M}_{T, \chi_j}
&= \log\log T - \log(K'\log\log T) + O\Big(\frac{(\log\log\log T)^{2}}{\log T}\frac{\log T}{K'\log\log T}\Big) + O(1) 
\\&= \log\log T - \log\log\log T + O\Big(\frac{(\log\log\log T)^{2}}{K'\log\log T}\Big) - \log K' + O(1)
\\& = \log\log T - \log\log\log T + O\Big(\frac{(\log\log\log T)^{2}}{K'\log\log T}\Big),
\end{split}
\end{align}
where the last equality follows from taking sufficiently large $T$. This implies that $\mathfrak{M}_{T, \chi_j}\leq \log\log T$ and   $\mathfrak{M}_{T, \chi_j}\sim\log\log T$ for sufficiently large $T$.

Recalling the definitions of ${\bf R}_{T}$ and ${\bf R}^{1}_{T}$ from \eqref{bf R_{T}} and \eqref{bf R^{1}_{T}}, respectively, we have
\begin{align*}
&d_\mathcal{D}  \bigg( \frac{1}{\sqrt{\oh \log\log T}} (P(\sigma_{0} + \mi (U+\alpha_{j}));\chi_k)_{j=1}^{N}), \frac{1}{\sqrt{\oh\mathfrak{M}_{T, \chi_j}}} (P_{1}(\sigma_{0} + \mi (U+\alpha_{j}));\chi_k)_{j=1}^{N})\bigg)\\
&\le  \frac{L}{\sqrt{\oh \mathfrak{M}_{T, \chi_j}}} \Bbb{E}\Big[ \Big(\sum_{j=1}^N  (   P_{0}(\sigma_{0} + \mi (U+\alpha_{j})  +P_{2}(\sigma_{0} + \mi (U+\alpha_{j}) )^2 \Big)^{1/2}\Big].
\end{align*}
Note that by the Cauchy-Schwarz inequality, we have
$$
 (   P_{0}(\sigma_{0} + \mi (U+\alpha_{j})  +P_{2}(\sigma_{0} + \mi (U+\alpha_{j}) )^2
 \le 2 (   P_{0}^2(\sigma_{0} + \mi (U+\alpha_{j})  +P_{2}^2(\sigma_{0} + \mi (U+\alpha_{j}) ).
$$
Thus, the above distance is
\begin{align*}
\ll L\sqrt{N} \frac{ \sqrt{ 1+ \log\log\log T}}{\sqrt{\mathfrak{M}_{T, \chi_j}}}.
\end{align*}

Since  $\mathfrak{M}_{T, \chi_j}\sim\log\log T$, as $T\rightarrow\infty$, one may require that $\log\log T/\mathfrak{M}_{T, \chi_j}\leq 2$ for sufficiently large $T$. Therefore, 
\begin{align*}
\sqrt{N} \frac{ \sqrt{ 1+ \log\log\log T}}{\sqrt{\mathfrak{M}_{T, \chi_j}}} = \sqrt{N}\frac{ \sqrt{ 1+ \log\log\log T}}{\sqrt{\log\log T}}\frac{\sqrt{\log\log T}}{\sqrt{\mathfrak{M}_{T, \chi_j}}}
\ll \sqrt{N}\frac{ \sqrt{ 1+ \log\log\log T}}{\sqrt{\log\log T}}
\end{align*}
as desired.
\end{proof}

\section{Proof of Proposition \ref{Prop 6 in AR24}}\label{pf-prop-6}

For the sake of simplicity, we write   $\tilde{\bf Z}_{N} =(\tilde{\mathcal{Z}}_{1}, \ldots, \tilde{\mathcal{Z}}_{N})$ for $\tilde{\bf Z}_{N,T} =(\tilde{\mathcal{Z}}_{1,T}, \ldots, \tilde{\mathcal{Z}}_{N,T})$. We  shall show the existence of the $N$-variate normal distribution  $\tilde{\bf Z}_{N} =(\tilde{\mathcal{Z}}_{1}, \ldots, \tilde{\mathcal{Z}}_{N})$ whose mean is ${\bf 0}_{N}$ and covariance matrix same as ${\bf R}^{1}_{T}$. To prove the existence, we appeal to the following existence theorem of a Gaussian process (see, e.g., \cite[Theorem 2.1]{HH}): 
\begin{theorem}\label{Hida}
Let $\mathcal{I}$ be an index set.  For $T\in \mathcal{I}$, let $ \mathfrak{K}(T)$ be an $N\times N$ nonsingular symmetric positive-definite matrix. Then there exists a (unique) Gaussian process $G = \{G(T): \ T\in \mathcal{I}\}$ with mean vector ${\bf 0}_{N}$ and covariance matrix $\mathfrak{K}(T)$.
\end{theorem}

In light of Theorem \ref{Hida}, it suffices to prove that the ``covariance matrix'' of ${\bf R}^{1}_{T}$ is positive-definite. We define the matrix $\widetilde{ \mathfrak{K} }(T) = (\tilde{k}_{i, j}(T))_{1\leq i, j\leq N}$ as following:
\begin{enumerate}
\item $\tilde{k}_{i, i}(T) = 1$ for all $1\leq i\leq N$.
\item For $i\neq j$, we set 
\begin{align*}
\tilde{k}_{i, j}(T):= \frac{\Cov[P_{1, i}, P_{1, j}]}{\sqrt{\Var[P_{1, i}]\Var[P_{1, j}]}},
\end{align*}
where $P_{1, j} = P_{1}(\sigma_{0} + \mi (U+\alpha_{j}));\chi_j)$.
\end{enumerate}
Then the matrix $\widetilde{ \mathfrak{K} }(T)$ is the ``covariance matrix'' of ${\bf R}^{1}_{T}$ provided that $\widetilde{ \mathfrak{K} }(T)$ is positive-definite. In what follows, we shall claim that $\widetilde{ \mathfrak{K} }(T)$  is positive-definite for sufficiently large $T$.

\begin{lemma}\label{covariance lemma}
There exists a number $T_0>0$ such that the matrix $\widetilde{ \mathfrak{K} }(T)$ defined previously is positive-definite for $T\ge T_0$.
\end{lemma}

\begin{proof}
 For a generic $N\times N$ matrix $M = (m_{i, j})_{1\leq i, j\leq N}$, we denote the upper left $r$-by-$r$ corner of $M$  by $M_{r} = (m_{i, j})_{1\leq i, j\leq r}$, where $1\le r\leq N$. 
Now, we consider  the matrix $\mathfrak{K} = (k_{i, j})_{1\leq i, j\leq N} $ introduced as in \eqref{the covariance matrix}. Since $\mathfrak{K}$ is positive-definite, by the Sylvester's criterion (see, e.g., \cite{Sylvester's criterion}), we have that $\det(\mathfrak{K}_{r})>0$ for all $1\leq n\leq N$. On the other hand,  by  the same calculation as in \cite[Proof of Theorem 1.1]{HW-Forum}, we see that  $\tilde{k}_{i, j}(T)\rightarrow k_{i, j}$ as $T\rightarrow\infty$ for all $1\leq i, j \leq N$. Consequently, as  the determinant function is continuous, we conclude that $\det(\widetilde{ \mathfrak{K} }_{r}(T))\rightarrow\det(\mathfrak{K}_{r})>0$, for all $1\le r \le N$, which yields the lemma.
\end{proof}

By Lemma \ref{covariance lemma}, we see that the matrix $\widetilde{ \mathfrak{K} }(T)$ is positive-definite when $T>T_0$. Therefore, we employ Theorem \ref{Hida} to construct an $N$-variate normal distribution $(\tilde{\mathcal{Z}}_{1}, \ldots, \tilde{\mathcal{Z}}_{N})$  so that its mean is ${\bf 0}_{N}$ and covariance matrix is $\widetilde{ \mathfrak{K} }(T)$ for \emph{every} $T> T_0$. Therefore, in fact, $\tilde{\bf Z}_{N} = (\tilde{\mathcal{Z}}_{1}, \ldots, \tilde{\mathcal{Z}}_{N})$ is a function of $T$ and is a \emph{Gaussian process}. Next, we shall compute its moments:

\begin{lemma}\label{moments of tilde Z_{N}}
Let $T_{0}$ be the same as in Lemma \ref{covariance lemma} and $T>T_{0}$. For any ${\bf u} = (u_{1}, \ldots, u_{N})\in\mathbb{R}^{N}$ and $n\in\mathbb{N}$, we have
\begin{align*}
\mathbb{E}[ ({\bf u}\cdot \tilde{{\bf Z}}_{N})^{2n}]
\leq \frac{(2n)!}{n! 2^{n}}\|{\bf u}\|^{2n}_{2}.
\end{align*}
\end{lemma}

\begin{proof}
Since $T>T_{0}$, the process $\tilde{\bf Z}_{N}$ is a Gaussian process with mean ${\bf 0}_{N}$ and covariance matrix $\widetilde{ \mathfrak{K} }(T)$. Therefore, $({\bf u}\cdot \tilde{{\bf Z}}_{N})$ is a univariate normal distribution with mean $0$ and variance
\begin{align*}
\sum_{1\leq j\leq N }u^{2}_{j}\Var(\tilde{\mathcal{Z}}_{j}) + \sum_{1\leq i< j\leq N}2 u_{i}u_{j}\Cov(\tilde{\mathcal{Z}}_{i}, \tilde{\mathcal{Z}}_{j}).
\end{align*}
Since each $\tilde{\mathcal{Z}}_{j}$ is of univariate normal with mean $0$ and variance $1$, we simplify the above to obtain
\begin{align*}
\sum_{1\leq j\leq  N}u^{2}_{j} + \sum_{1\leq i< j\leq N}2 u_{i}u_{j}\tilde{k}_{i, j }(T) = \sum_{1\leq i, j\leq N}u_{i}u_{j}\tilde{k}_{i, j}(T),
\end{align*}
where $\tilde{k}_{i, j}(T)$ is the $ij$-entry in its covariance matrix $\widetilde{ \mathfrak{K} }(T)$. Therefore, by moment formula for the univariate normal, we have
\begin{align*}
\mathbb{E}[ ({\bf u}\cdot \tilde{{\bf Z}}_{N})^{2n}]
= \frac{(2n)!}{n! 2^{n}}\Big(\sum_{1\leq i, j\leq N}u_{i}u_{j}\tilde{k}_{i, j}(T)\Big)^{n}.
\end{align*}

By the construction, we have that $0\leq \tilde{k}_{i, j}(T)\leq 1$ for all $i, j$, and $T$. Also, by the Cauchy-Schwarz inequality, we have
\begin{align*}
\Big|\sum_{1\leq i, j \leq N}u_{i}u_{j}\Big| \leq\Big(\sum_{1\leq i\leq N} u^{2}_{i}\Big)^{1/2}\Big(\sum_{1\leq j\leq N} u^{2}_{j}\Big)^{1/2} = \|{\bf u}\|^{2}_{2},
\end{align*}
where $\|{\bf u}\|_{2} = (\sum_{j = 1}^{N}u^{2}_{j})^{1/2}$ is the $2$-norm of ${\bf u}$.
Putting everything together then yields
\begin{align}\label{nth moment of Gaussian}
\mathbb{E}[ ({\bf u}\cdot \tilde{{\bf Z}}_{N})^{2n}]
= \frac{(2n)!}{n! 2^{n}}\Big(\sum_{1\leq i, j\leq N}u_{i}u_{j}\tilde{k}_{i, j}(T)\Big)^{n}
\leq \frac{(2n)!}{n! 2^{n}}\|{\bf u}\|^{2n}_{2},
\end{align}
which completes the proof.
\end{proof}

We also compute the moments of ${\bf R}^{1}_{T}$:
\begin{lemma}\label{moments of R^{1}_{T}}
For any ${\bf u} = (u_{1}, \ldots, u_{N})\in\mathbb{R}^{N}$ and $n\in\mathbb{N}$, we have
\begin{align*}
\mathbb{E}\left[ ({\bf u}\cdot {\bf R}^{1}_{T})^{2n}\right]
< ||{\bf u}||_1^{2n} \cdot   \kappa_1  (1+ \mathcal{C}^2) \frac{(2n)!}{n! 2^n},
\end{align*}
where $\kappa_1$ is an absolute constant, and $\mathcal{C}$ is the same constant \footnote{$\mathcal{C}$ may be taken as $1/3$ if necessary.}  in \eqref{P1-est-upper-bd}.
\end{lemma}

\begin{proof}
 Recall that
\begin{align}\label{product R^{1}_{T}}
{\bf u}\cdot {\bf R}^{1}_{T}
= \sum_{j = 1}^{N}u_{j}     P_{1}(s_{0, j})/{\sqrt{\oh\mathfrak{M}_{T, \chi_j}}}  .
\end{align}
Writing   $f_j (t) =   P_{1}(s_{0, j})/{\sqrt{\oh\mathfrak{M}_{T, \chi_j}}} $, we have
\begin{align*}
&\mathbb{E}[ ({\bf u}\cdot {\bf R}^{1}_{T})^{2n}]
\le \frac{1}{T} \int_{T}^{2T} |{\bf u}\cdot {\bf R}^{1}_{T}|^{2n} dt
= \frac{1}{T} \int_{T}^{2T} \Big|\sum_{j=1}^N    u_j f_j (t)\Big|^{2n} dt
=  \frac{1}{T}\|\sum_{j=1}^N    u_j f_j  \|_{2n}^{2n }
\le   \frac{1}{T} \Big( \sum_{j=1}^N   | u_j |  \|f_j  \|_{2n}  \Big)^{2n},
\end{align*}
by the triangle inequality. In addition, by \eqref{P1-est-upper-bd}, if  $\frac{6\mathcal{C}}{K'}\le 1$, we have
\begin{equation*}
\mathbb{E}_T[ |P_1(s;\chi)|^{2k} ] \ll  (1+ \mathcal{C}^2) \frac{(2k)!}{k! 2^k} \Big(\sqrt{ \oh \log\log T} \Big)^{2k}.
\end{equation*}
We therefore obtain
$$
||f_j  ||_{2n}=  \Big( \sqrt{\oh\mathfrak{M}_{T, \chi_j}}\Big)^{-1} \Big( T \mathbb{E}_T[ |P_1(s_{0, j};\chi_j)|^{2n} ] \Big)^{\frac{1}{2n}}
\le \Big(T \kappa_1  (1+ \mathcal{C}^2) \frac{(2n)!}{n! 2^n} \Big(\sqrt{ \oh \log\log T} \Big)^{2n -2n} \Big)^{\frac{1}{2n}}
$$
and thus
\begin{align*}
&\mathbb{E}[ ({\bf u}\cdot {\bf R}^{1}_{T})^{2n}]
\le    \Big( \sum_{j=1}^N   | u_j |  \Big( \kappa_1  (1+ \mathcal{C}^2) \frac{(2n)!}{n! 2^n} \Big)^{\frac{1}{2n}}  \Big)^{2n}
= ||{\bf u}||_1^{2n} \cdot   \kappa_1  (1+ \mathcal{C}^2) \frac{(2n)!}{n! 2^n} .
\end{align*}
\end{proof}

Next, we need a detailed estimate on the difference between $\mathbb{E}[ ({\bf u}\cdot {\bf R}^{1}_{T})^{2n}]$ and $\mathbb{E}[ ({\bf u}\cdot \tilde{{\bf Z}}_{N})^{2n}]$. To begin, we  analyse $\mathfrak{M}_{T, \chi_j}$ further. 
By \eqref{calculation of frak M}, we let
\begin{align}\label{the very definition of C_{1}(T)}
C_{1}(T):=\frac{\sqrt{\oh\log\log T}}{\sqrt{\oh(\log\log T - \log\log\log T + O(\frac{(\log\log\log T)^{2}}{K'\log\log T}))}} = \sqrt{\frac{\log\log T}{\mathfrak{M}_{T, \chi_{j}}}}.
\end{align}
Then we see that $C_{1}(T)\rightarrow 1$ as $T\rightarrow \infty$, and that $|C_{1}(T)|>1$ for all $T$. Let $\kappa_{2}>0$ be a constant such that
\begin{align}\label{the very definition of kappa5}
|C_{1}(T)|^{2}<\kappa_{2}
\end{align}
for all $T$. We then turn our attention to the quantity of $\kappa_{2}$. According to \eqref{the very definition of C_{1}(T)}, we see that
\begin{align*}
|C_{1}(T)|^{2}<\frac{\log\log T}{\log\log T - \log\log\log T} = 1 + \frac{\log\log\log T}{\log\log T - \log\log\log T}:= 1 + f(T)
\end{align*}
for all $T$. Therefore, we may take $\kappa_{2}$ as $1 + \max_{T>0}f(T)$. A standard calculation shows that $|f(T)|< 1/(e - 1)$ for all $T>0$. Therefore, we may take 
\begin{align}\label{simple kappa5}
\kappa_{2} = \frac{e}{e - 1}.
\end{align}

On the other hand,  $\tilde{k}_{i, j}(T)$ converges to $1$ or $0$ as $T\rightarrow\infty$. Hence, we assume  that $|\tilde{k}_{i, j}(T)|< e/(e - 1)$ for sufficiently large $T$.

\begin{lemma}\label{precise difference}
 Let $T_{0}$ be the same as in Lemma \ref{covariance lemma} and $T>T_{0}$ be sufficiently large. For any ${\bf u} = (u_{1}, \ldots, u_{N})\in\mathbb{R}^{N}$ and $k\in\mathbb{N}$, we have that
\begin{align}\label{precise difference odd}
\begin{split}
&\Big|\mathbb{E}[ ({\bf u}\cdot {\bf R}^{1}_{T})^{2k + 1}] - \mathbb{E}[ ({\bf u}\cdot \tilde{{\bf Z}}_{N})^{2k + 1}]\Big|
\ll \frac{1}{T}\frac{(C_{1}(T))^{2k+ 1}}{\big(\sqrt{\oh\log\log T}\big)^{2k + 1}}((2k + 1)!)^{2}(\|{\bf u}\|_1 T^{(K'\log\log T)^{-1}})^{4k + 2},
\end{split}
\end{align}
and that
\begin{align}\label{precise difference even}
\begin{split}
\Big|\mathbb{E}[ ({\bf u}\cdot {\bf R}^{1}_{T})^{2k}] - \mathbb{E}[ ({\bf u}\cdot \tilde{{\bf Z}}_{N})^{2k}] \Big|
&\ll
\frac{(2k)!}{k! 2^{k}}\frac{\|{\bf u}\|_{2}^{2k}\kappa_{2}^{k - 1} (k - 1)k }{T}
+\|{\bf u}\|_{1}^{2k}(C_{1}(T))^{2k}\frac{(2k)!k}{k!2^{k}}\frac{\Delta(T) + \delta(T)}{\log\log T} 
\\&+\|{\bf u}\|_{1}^{2k}(C_{1}(T))^{2k}\frac{(2k)!k}{k!2^{k}}\frac{\delta(T)}{\log\log T} + \|{\bf u}\|_{1}^{2k}\frac{(2k)!}{k!2^{k}} \frac{2k(k - 1)}{(\log\log T)^{2 - 3\varepsilon}} \\
&+ \|{\bf u}\|_{1}^{2k} \frac{(2k)!}{k!2^{k}}\frac{3^{k}(k + 1)(k + 2)}{(\log\log T)^{3 - 4\varepsilon}}
\\&+ \frac{(2k)!k^{2}}{k!2^{k}}\frac{1}{(\log\log T)^{2}}
 + \frac{((2k)!)^{2}k}{(k!)^{2}2^{k}}\frac{T^{2k/(K'\log\log T)}}{T(\log\log T)^{k}},
\end{split}
\end{align}
where the implied constants are  absolute, and $\kappa_{2}$ is introduced in \eqref{simple kappa5}.
\end{lemma}

\begin{proof}
Similar to  \eqref{P_{a, 0} series}, we consider 
\begin{align*}
\mathcal{P}_{{\bf u}, 0}(s, 13, Y):= \sum_{13<p\leq Y}\frac{u_{1}\chi_{1}(p)p^{-\mi \alpha_{1}} + \cdots + u_{N}\chi_{N}(p)p^{-\mi\alpha_{n}}}{p^{s}}. 
\end{align*} 
Take $s = \sigma_{0} + \mi t$ and let $b_{k}(n)$ and $\Psi_{k}(n)$ be defined analogous to those in Lemma \ref{Lemma 6.2 in SCLT_Forum}. We then have, for $k\neq\ell$
\begin{align}\label{detailed moment calculation k neq ell}
\int^{2T}_{T}\mathcal{P}_{{\bf u}, 0}(\sigma_{0} +\mi t, 13, Y)^{k}\overline{\mathcal{P}_{{\bf u}, 0} (\sigma_{0} + \mi t, 13, Y)^{\ell}}dt
= O(k!\ell !((||{\bf u}||_1 Y)^{k + \ell}),
\end{align}
where the implied constant in the big-O term is absolute.

For $k = \ell$, we have
\begin{align}\label{detailed moment calculation}
\int^{2T}_{T}|\mathcal{P}_{{\bf u}, 0}(\sigma_{0} + \mi t, 13, Y)|^{2k}dt = 
T \bigg|\sum_{n}\frac{b_{k}(n)^2\Psi_k(n)\overline{\Psi_{k}(n)}}{n^{2\sigma_{0}}} \bigg|
+O((k!)^{2}\|{\bf u}\|_{1}^{2k}Y^{2k}).
\end{align}
As we know
\begin{align}\label{sum of primes greater than 13}
\sum_{13< p\leq Z}\frac{1}{p} = \sum_{p\leq Z}\frac{1}{p} - \sum_{2\leq p\leq 13}\frac{1}{p} = \sum_{p\leq Z}\frac{1}{p} + O(1),
\end{align}
for the non-square-free $n$ in the first term on right of \eqref{detailed moment calculation},  the calculation from \eqref{psipsibar_{n}} to \eqref{non-square-free-fixed} again yields
\begin{align}\label{non-square-free in detailed moment calculation}
\bigg|\sum_{n\ \text{ non-square-free}}\frac{b_{k}(n)^2\Psi_k(n)\overline{\Psi_{k}(n)}}{n^{2\sigma_{0}}} \bigg| = k! k^{2}(\log\log Y +O(1))^{k-2}.
\end{align}

It follows from \eqref{equation next to (6.15)_new} and \eqref{sum of primes greater than 13} that
$$
\sum_{13<p\leq Y}\frac{ \chi_i\bar{\chi}_j(p)p^{-\mi(\alpha_{i}-\alpha_{j})}}{p^{2\sigma_{0}}}
=\delta_{i,j} \min\Big(\log\log Y , \log\Big(\frac{1}{|\alpha_{i} - \alpha_{j}|}\Big) \Big) + O_{q_i,q_j}(|\alpha_{i} - \alpha_{j}| +1) + O \Big(\frac{(\log\log\log T)^{2}}{\log\log T}\Big).
$$
Thus, we shall require
\begin{equation}\label{distance-cond}
O_{q_i,q_j}(|\alpha_{i} - \alpha_{j}| +1) + O \Big(\frac{(\log\log\log T)^{2}}{\log\log T}\Big)
= O_1{(\Delta(T))},
\end{equation}
for all $i,j$, to proceed further. (Here the implied constant came from Lemma \ref{Lemma 6.1 in SCLT_Forum}.)

Similar to the proof of Lemma \ref{Lemma 6.2 in SCLT_Forum}, we have 
\begin{align*}
&\bigg|\sum_{n\ \text{square-free}}\frac{b_{k}(n)^2\Psi_k(n)\overline{\Psi_{k}(n)}}{n^{2\sigma_{0}}} \bigg| 
\\
&=k!\Big((a^{2}_{1}+\cdots +a_{N}^{2})(\log\log T + O_1(\Delta(T)))
+ 2\sum_{1\leq i<j\leq N}a_{i}a_{j}\delta_{i,j}(\mathcal{V}(T, \alpha_{i}, \alpha_{j}) + O_1(\Delta(T)))\Big)^{k}
\\
&=k!\Big((u^{2}_{1}+\cdots +u_{N}^{2})(\log\log T + O_1(\Delta(T)))
+ 2\sum_{1\leq i<j\leq N}u_{i}u_{j}\delta_{i,j}( c_{i, j}\log\log T + O_1(\Delta(T)) + O_1(\delta(T)))\Big)^{k},
\end{align*} 
where the last equality follows from the definition of $\mathcal{V}(T, \alpha_{i}, \alpha_{j})$ in \eqref{6.1 of SLCT_Forum}.
 Expending the parenthesis, we get
\begin{align*}
&\sum_{1\leq i, j\leq N}u_{i}u_{j}\delta_{i, j}c_{i, j}\log\log T +  \sum_{1\leq i, j\leq N}u_{i}u_{j}\delta_{i, j}O_1(\Delta(T))+ 2\sum_{1\leq i<j\leq N}u_{i}u_{j}\delta_{i,j} O_1(\delta(T))
\\&=  \sum_{1\leq i, j\leq N}u_{i}u_{j}\delta_{i, j}c_{i, j}\log\log T +  \sum_{1\leq i, j\leq N}u_{i}u_{j}\delta_{i, j}(O_1(\Delta(T)) + O_1(\delta(T))) - \sum_{j = 1}^{N}u_{j}^{2}O_1(\delta(T)).
\end{align*}
Therefore, we have
\begin{align}\label{square-free in detailed moment calculation}
\begin{split}
&\bigg|\sum_{n\ \text{square-free}}\frac{b_{k}(n)^2\Psi_k(n)\overline{\Psi_{k}(n)}}{n^{2\sigma_{0}}} \bigg| 
\\
&=k!\Big( \sum_{1\leq i, j\leq N}u_{i}u_{j}\delta_{i, j}c_{i, j}\log\log T +  \sum_{1\leq i, j\leq N}u_{i}u_{j}\delta_{i, j}(O_1(\Delta(T)) + O_1(\delta(T))) - \sum_{j = 1}^{N}u_{j}^{2}O_1(\delta(T))\Big)^{k}.
\end{split}
\end{align}

Recall \eqref{product R^{1}_{T}}. We have
\begin{align*}
&{\bf u}\cdot {\bf R}^{1}_{T}
= \sum_{j = 1}^{N}u_{j}    \frac{ P_{1}(s_{0, j})}{\sqrt{\oh\mathfrak{M}_{T, \chi_j}}} 
=\sum_{j = 1}^{N}u_{j}  \frac{ P_{1}(s_{0, j})}{\sqrt{\oh\log\log T}}\frac{\sqrt{\oh\log\log T}}{\sqrt{\oh\mathfrak{M}_{T, \chi_{j}}}}
=\sum_{j = 1}^{N}u_{j}  \frac{ P_{1}(s_{0, j})}{\sqrt{\oh\log\log T}}C_{1}(T),
\end{align*}
where $C_{1}(T)$ is defined in \eqref{the very definition of C_{1}(T)}. 
By the observation above and  \eqref{detailed moment calculation k neq ell}, we have
\begin{align*}
&\mathbb{E}[ ({\bf u}\cdot {\bf R}^{1}_{T})^{2k + 1}]
\\&=\frac{(C_{1}(T))^{2k+ 1}}{\big(\sqrt{\oh\log\log T}\big)^{2k + 1}}\frac{1}{2^{2k+ 1}}\sum_{\ell = 0}^{2k + 1}{2k+ 1\choose \ell}\frac{1}{T}\int^{2T}_{T}\mathcal{P}_{{\bf u}, 0}(\sigma_{0} +\mi t, 13, Y)^{\ell}\overline{\mathcal{P}_{{\bf u}, 0} (\sigma_{0} + \mi t, 13, Y)^{2k + 1 - \ell}}dt
\\&\ll\frac{(C_{1}(T))^{2k+ 1}}{\big(\sqrt{\oh\log\log T}\big)^{2k + 1}}\frac{1}{2^{2k+ 1}}\sum_{\ell = 0}^{2k + 1}{2k+ 1\choose \ell}\frac{1}{T}((2k+1)!\ell !((\|{\bf u}\|_1 Y)^{2k + 1 + \ell})
\\&\ll\frac{1}{T}\frac{(C_{1}(T))^{2k+ 1}}{\big(\sqrt{\oh\log\log T}\big)^{2k + 1}}((2k + 1)!)^{2}(\|{\bf u}\|_1 Y)^{4k + 2}\frac{1}{2^{2k+ 1}}\sum_{\ell = 0}^{2k + 1}{2k+ 1\choose \ell}
\\&= \frac{1}{T}\frac{(C_{1}(T))^{2k+ 1}}{\big(\sqrt{\oh\log\log T}\big)^{2k + 1}}((2k + 1)!)^{2}(\|{\bf u}\|_1 Y)^{4k + 2},
\end{align*}
where the last equality follows from binomial theorem that $\sum_{\ell = 0}^{2k + 1}{2k+ 1\choose \ell} = (1 + 1)^{2k + 1} = 2^{2k + 1}$. Replacing $Y$ by $T^{(K'\log\log T)^{-1}}$ and recalling that $\mathbb{E}[ ({\bf u}\cdot \tilde{{\bf Z}}_{N})^{2k + 1}] = 0$ for all $k$, we obtain \eqref{precise difference odd}.

On the other hand, for even powers, we have
\begin{align*}
\mathbb{E}[ ({\bf u}\cdot {\bf R}^{1}_{T})^{2k}]
&=\frac{(C_{1}(T))^{2k}}{\big(\sqrt{\oh\log\log T}\big)^{2k}}\frac{1}{2^{2k}}\sum_{\ell = 0}^{2k}{2k\choose \ell}\frac{1}{T}\int^{2T}_{T}\mathcal{P}_{{\bf u}, 0}(\sigma_{0} +\mi t, 13, Y)^{\ell}\overline{\mathcal{P}_{{\bf u}, 0} (\sigma_{0} + \mi t, 13, Y)^{2k - \ell}}dt
\\&=\frac{(C_{1}(T))^{2k}}{\big(\sqrt{\oh\log\log T}\big)^{2k}}\frac{1}{2^{2k}}{2k\choose k}\frac{1}{T}\Big(\int^{2T}_{T}|\mathcal{P}_{{\bf u}, 0}(\sigma_{0} + \mi t, 13, Y)|^{2k}dt + O((2k)! k Y^{2k})\Big).
\end{align*}
Using \eqref{detailed moment calculation k neq ell}, \eqref{detailed moment calculation}, \eqref{non-square-free in detailed moment calculation}, and \eqref{square-free in detailed moment calculation}, we further deduce
\begin{align*}
\begin{split}
&\mathbb{E}\left[ ({\bf u}\cdot {\bf R}^{1}_{T})^{2k}\right]\\
&=\frac{(C_{1}(T))^{2k}}{(\sqrt{\oh\log\log T})^{2k}}\frac{1}{2^{2k}}{2k\choose k}\frac{1}{T}
\\&\quad\times\Big(Tk!\Big(\sum_{1\leq i, j\leq N}u_{i}u_{j}\delta_{i, j}c_{i, j}\log\log T +  \sum_{1\leq i, j\leq N}u_{i}u_{j}\delta_{i, j}(O_1(\Delta(T)) + O_1(\delta(T))) - \sum_{j = 1}^{N}u_{j}^{2}O_1(\delta(T)) \Big)^{k} 
\\&\qquad+  T k! k^{2}(\log\log Y +O(1))^{k-2} + O((2k)! k Y^{2k})\Big)
\\&=(C_{1}(T))^{2k}\frac{(2k)!}{k!2^{k}}
\times\Big(\Big( \sum_{1\leq i, j\leq N}u_{i}u_{j}\delta_{i, j}c_{i, j} +  \sum_{1\leq i, j\leq N}u_{i}u_{j}\delta_{i, j}\frac{(O_1(\Delta(T)) + O_1(\delta(T)))}{\log\log T}  - \sum_{j = 1}^{N}u_{j}^{2}\frac{O_1(\delta(T))}{\log\log T}\Big)^{k} 
\Big) 
\\&+ O\Big(\frac{(2k)!k^{2}}{k!2^{k}}\frac{1}{(\log\log T)^{2}}\Big)
 + O\Big(\frac{((2k)!)^{2}k}{(k!)^{2}2^{k}}\frac{T^{2k/(K'\log\log T)}}{T(\log\log T)^{k}}\Big),
\end{split}
\end{align*}
where the second equality follows from $ Y = T^{1/K'\log\log T}$.

By the Cauchy-Schwarz inequality, we have
$
|\sum_{1\leq i, j\leq N}u_{i}u_{j}|\leq \|{\bf u}\|_{2}^{2},
$
 and thus
\begin{align*}
\sum_{1\leq i, j\leq N}u_{i}u_{j}\delta_{i, j}\frac{(O_1(\Delta(T)) + O_1(\delta(T)))}{\log\log T}
= O\Big(\|{\bf u}\|_{2}^{2}\frac{(\Delta(T) + \delta(T))}{\log\log T}\Big).
\end{align*}
Also, as $\sum_{j = 1}^{N}u_{j}^{2} \le \|{\bf u}\|_{1}^{2}$,  we have 
$
- \sum_{j = 1}^{N}u_{j}^{2}\frac{O_1(\delta(T))}{\log\log T}
= O\Big(\|{\bf u}\|_{1}^{2}\frac{\delta(T)}{\log\log T}\Big).
$
Hence, the last $k$-th power term above becomes 
\begin{align*}
&\Big( \sum_{1\leq i, j\leq N}u_{i}u_{j}\delta_{i, j}c_{i, j} + O\Big(\|{\bf u}\|_{2}^{2}\frac{(\Delta(T) + \delta(T))}{\log\log T}\Big)  + O\Big(\|{\bf u}\|_{1}^{2}\frac{\delta(T)}{\log\log T}\Big)\Big)^{k}
\\&=\Big( \sum_{1\leq i, j\leq N}u_{i}u_{j}\delta_{i, j}c_{i, j}\Big)^{k} +\sum_{\substack{\alpha + \beta + \gamma = k \\ \alpha\neq k }}\frac{k!}{\alpha!\beta!\gamma!}\Big(O(\|{\bf u}\|_{2}^{2})\Big)^{\alpha}\Big(O\Big(\|{\bf u}\|_{2}^{2}\frac{(\Delta(T) + \delta(T))}{\log\log T}\Big)\Big)^{\beta}\Big(O\Big(\|{\bf u}\|_{1}^{2}\frac{\delta(T)}{\log\log T}\Big)\Big)^{\gamma}.
\end{align*}

Notice that $\|{\bf u}\|_{2}\leq \|{\bf u}\|_{1}$. We have $O(\|{\bf u}\|_{2}) = O(\|{\bf u}\|_{1})$. Therefore, we conclude
\begin{align*}
&\sum_{\substack{\alpha + \beta + \gamma = k \\ \alpha\neq k }}\frac{k!}{\alpha!\beta!\gamma!}\Big(O(\|{\bf u}\|_{2}^{2})\Big)^{\alpha}\Big(O\Big(\|{\bf u}\|_{2}^{2}\frac{(\Delta(T) + \delta(T))}{\log\log T}\Big)\Big)^{\beta}\Big(O\Big(\|{\bf u}\|_{1}^{2}\frac{\delta(T)}{\log\log T}\Big)\Big)^{\gamma}
\\&\ll \|{\bf u}\|_{1}^{2k}\sum_{\substack{\alpha + \beta + \gamma = k \\ \alpha\neq k }}\frac{k!}{\alpha!\beta!\gamma!}\Big(1\Big)^{\alpha}\Big(\frac{(\Delta(T) + \delta(T))}{\log\log T}\Big)^{\beta}\Big(\frac{\delta(T)}{\log\log T}\Big)^{\gamma}.
\end{align*}
We next need to isolate the terms whose denominator is $\log\log T$ or $(\log\log T)^{2}$ since $(\log\log T)^{j}, j\geq 3$,  will be absorbed by $(\log\log T)^{2}$. Writing the ordered pair in this fashion $(\alpha, \beta, \gamma)$, we see that 
the terms with denominator $\log\log T$ is contributed by $(k - 1, 1, 0)$ and $(k - 1, 0, 1)$; the terms with denominator $(\log\log T)^{2}$ is contributed by $(k - 2, 2, 0), (k - 2, 0, 2)$, and $(k - 2, 1, 1)$. Thus, we have
\begin{align*}
&\sum_{\substack{\alpha + \beta + \gamma = k \\ \alpha\neq k }}\frac{k!}{\alpha!\beta!\gamma!}\Big(1\Big)^{\alpha}\Big(\frac{(\Delta(T) + \delta(T))}{\log\log T}\Big)^{\beta}\Big(\frac{\delta(T)}{\log\log T}\Big)^{\gamma}
\\&= \frac{k!}{(k - 1)!}\frac{\Delta(T) + \delta(T)}{\log\log T} + \frac{k!}{(k - 1)!}\frac{\delta(T)}{\log\log T} 
\\&+ \frac{k!}{(k - 2)! 2!}\Big(\frac{(\Delta(T) + \delta(T))}{\log\log T}\Big)^{2} + \frac{k!}{(k - 2)! 2!} \Big(\frac{\delta(T)}{\log\log T}\Big)^{2} + \frac{k!}{(k - 2)!} \Big(\frac{(\Delta(T) + \delta(T))}{\log\log T}\Big)\Big(\frac{\delta(T)}{\log\log T}\Big)
\\& + \mathcal{E}.
\end{align*}

We analyze $\mathcal{E}$. First, in the expansion of $(A + B + C)^{k}$, there are overall ${{k + 3 - 1}\choose{3}}$ terms. Therefore, there are ${{k + 3 - 1}\choose{3}} - 6$ terms in $\mathcal{E}$. Moreover, among them, the largest coefficient is the multinomial coefficient with $\alpha = \beta = \gamma = k/3$, which is $k!/((k/3)!)^{3}$. Lastly, using the assumption that $\Delta(T) =  O((\log\log T)^{\varepsilon})$ and that $\delta(T)  = O((\log\log T)^{\varepsilon})$, we conclude that
\begin{align*}
\mathcal{E}
\ll \bigg({{k + 3 - 1}\choose{3}} - 6\bigg)\bigg(\frac{k!}{((k/3)!)^{3}}\bigg)\frac{1}{(\log\log T)^{3 - 3\varepsilon}}.
\end{align*}
It follows from the Stirling formula \eqref{Stirling formula} that
\begin{align*}
\mathcal{E}
&\ll \bigg({{k + 3 - 1}\choose{3}} - 6\bigg)\bigg(\frac{k!}{((k/3)!)^{3}}\bigg)\frac{1}{(\log\log T)^{3 - 3\varepsilon}}
\\&= \bigg(\frac{k (k + 1) (k + 2)}{6} - 6\bigg)\bigg(\frac{\sqrt{2\pi k} (k/e)^{k}(1 + O(1/k))}{(\sqrt{2\pi (k/3)} ((k/3)/e)^{k/3}( 1 + O(1/k)))^{3}}\bigg)\frac{1}{(\log\log T)^{3 - 3\varepsilon}}
\\&\ll \bigg(\frac{k (k + 1) (k + 2)}{6} - 6\bigg)\bigg(\frac{3^{k}}{k}\bigg)\frac{1}{(\log\log T)^{3 - 3\varepsilon}}
\\&\ll \frac{3^{k}(k + 1)(k + 2)}{(\log\log T)^{3 - 3\varepsilon}}.
\end{align*}
Meanwhile, using the same strategy to simplify the terms with denominator $(\log\log T)^{2}$, we obtain
\begin{align*}
&\sum_{\substack{\alpha + \beta + \gamma = k \\ \alpha\neq k }}\frac{k!}{\alpha!\beta!\gamma!}\Big(O(\|{\bf u}\|_{2}^{2})\Big)^{\alpha}\Big(O\Big(\|{\bf u}\|_{2}^{2}\frac{(\Delta(T) + \delta(T))}{\log\log T}\Big)\Big)^{\beta}\Big(O\Big(\|{\bf u}\|_{1}^{2}\frac{\delta(T)}{\log\log T}\Big)\Big)^{\gamma}
\\&\ll \|{\bf u}\|_{1}^{2k}\sum_{\substack{\alpha + \beta + \gamma = k \\ \alpha\neq k }}\frac{k!}{\alpha!\beta!\gamma!}\Big(1\Big)^{\alpha}\Big(\frac{(\Delta(T) + \delta(T))}{\log\log T}\Big)^{\beta}\Big(\frac{\delta(T)}{\log\log T}\Big)^{\gamma}
\\&\ll\|{\bf u}\|_{1}^{2k}\Big( k\frac{\Delta(T) + \delta(T)}{\log\log T} + k\frac{\delta(T)}{\log\log T} +  \frac{2k(k - 1)}{(\log\log T)^{2 - 2\varepsilon}} +\frac{3^{k}(k + 1)(k + 2)}{(\log\log T)^{3 - 3\varepsilon}}\Big).
\end{align*}
Hence, we have
\begin{align}\label{2k moment of R^{1}_{T}}
\begin{split}
&\mathbb{E}[ ({\bf u}\cdot {\bf R}^{1}_{T})^{2k}]
\\&=(C_{1}(T))^{2k}\frac{(2k)!}{k!2^{k}}\Big( \sum_{1\leq i, j\leq N}u_{i}u_{j}\delta_{i, j}c_{i, j}\Big)^{k}
\\&+(C_{1}(T))^{2k}\frac{(2k)!}{k!2^{k}}\Big\{
O\Big(\|{\bf u}\|_{1}^{2k}k\frac{\Delta(T) + \delta(T)}{\log\log T}\Big) + O\Big(\|{\bf u}\|_{1}^{2k}k\frac{\delta(T)}{\log\log T}\Big) + O\Big(\|{\bf u}\|_{1}^{2k} \frac{2k(k - 1)}{(\log\log T)^{2 - 2\varepsilon}}\Big) 
\\&+ O\Big(\|{\bf u}\|_{1}^{2k} \frac{3^{k}(k + 1)(k + 2)}{(\log\log T)^{3 - 3\varepsilon}} \Big)
\Big\} 
+ O\Big(\frac{(2k)!k^{2}}{k!2^{k}}\frac{1}{(\log\log T)^{2}}\Big)
 + O\Big(\frac{((2k)!)^{2}k}{(k!)^{2}2^{k}}\frac{T^{2k/(K'\log\log T)}}{T(\log\log T)^{k}}\Big)
\\&=\frac{(2k)!}{k!2^{k}}\Big( \sum_{1\leq i, j\leq N}u_{i}u_{j}(C_{1}(T))^{2}\delta_{i, j}c_{i, j}\Big)^{k} 
+O\Big(\|{\bf u}\|_{1}^{2k}(C_{1}(T))^{2k}\frac{(2k)!k}{k!2^{k}}\frac{\Delta(T) + \delta(T)}{\log\log T}\Big) 
\\&+ O\Big(\|{\bf u}\|_{1}^{2k}(C_{1}(T))^{2k}\frac{(2k)!k}{k!2^{k}}\frac{\delta(T)}{\log\log T}\Big) + O\Big(\|{\bf u}\|_{1}^{2k}\frac{(2k)!}{k!2^{k}} \frac{2k(k - 1)}{(\log\log T)^{2 - 3\varepsilon}}\Big) + O\Big(\|{\bf u}\|_{1}^{2k} \frac{(2k)!}{k!2^{k}}\frac{3^{k}(k + 1)(k + 2)}{(\log\log T)^{3 - 4\varepsilon}}\Big)
\\&+ O\Big(\frac{(2k)!k^{2}}{k!2^{k}}\frac{1}{(\log\log T)^{2}}\Big)
 + O\Big(\frac{((2k)!)^{2}k}{(k!)^{2}2^{k}}\frac{T^{2k/(K'\log\log T)}}{T(\log\log T)^{k}}\Big),
\end{split}
\end{align}
where the forth and the fifth terms follow from taking $T$ sufficiently large so that $C_{1}(T)/(\log\log T)^{\varepsilon}\leq 1$.
 
Now, we shall estimate $|\mathbb{E}[ ({\bf u}\cdot {\bf R}^{1}_{T})^{2k}] - \mathbb{E}[ ({\bf u}\cdot \tilde{{\bf Z}}_{N})^{2k}]|$. In light of \eqref{2k moment of R^{1}_{T}}, it suffices to  consider the difference between the first term in \eqref{2k moment of R^{1}_{T}} and \eqref{nth moment of Gaussian}.
 Since $C_{1}(T)\rightarrow 1$ as $T\rightarrow\infty$, there exists $T_{C_{1}}>0$ such that 
\begin{align}\label{control of C_{1}}
|(C_{1}(T))^{2} - 1|<\varepsilon_{5}
\end{align}
whenever $T>T_{C_{1}}$. On the other hand, by Lemma \ref{covariance lemma}, we see that $\tilde{k}_{i, j}(T)\rightarrow k_{i, j} = \delta_{i, j}c_{i, j}$ as $T\rightarrow\infty$ for all $1\leq i, j \leq N$. Therefore, there exists $T_{\tilde{k}}>0$ such that
\begin{align}\label{control of tilde k}
|\tilde{k}_{i, j}(T) - \delta_{i, j}c_{i, j}|<\varepsilon_{6}
\end{align}
whenever $T>T_{\tilde{k}}$.

Now, consider
\begin{align}\label{difference in kth moment}
\begin{split}
&\Big|\frac{(2k)!}{k!2^{k}}\Big( \sum_{1\leq i, j\leq N}u_{i}u_{j}(C_{1}(T))^{2}\delta_{i, j}c_{i, j}\Big)^{k}  - \frac{(2k)!}{k!2^{k}}\Big(\sum_{1\leq i, j\leq N}u_{i}u_{j}\tilde{k}_{i, j}(T)\Big)^{k}\Big|
\\&=\frac{(2k)!}{k!2^{k}}\Big|\Big(\sum_{1\leq i, j\leq N}u_{i}u_{j}(C_{1}(T))^{2}\delta_{i, j}c_{i, j}\Big)^{k} - \Big(\sum_{1\leq i, j\leq N}u_{i}u_{j}\tilde{k}_{i, j}(T)\Big)^{k}\Big|
\\&=\frac{(2k)!}{k!2^{k}}\Big|\sum_{1\leq i, j\leq N}u_{i}u_{j}(C_{1}(T))^{2}\delta_{i, j}c_{i, j} - \sum_{1\leq i, j\leq N}u_{i}u_{j}\tilde{k}_{i, j}(T)\Big|
\\
&\times\Big(\sum_{\ell = 0}^{k - 1}\Big(\sum_{1\leq i, j\leq N}u_{i}u_{j}(C_{1}(T))^{2}\delta_{i, j}c_{i, j}\Big)^{k - 1 - \ell}\Big(\sum_{1\leq i, j\leq N}u_{i}u_{j}\tilde{k}_{i, j}(T)\Big)^{\ell} \Big).
\end{split}
\end{align}
Recall the definition of $\kappa_{2}$ from \eqref{simple kappa5}. We have already set that $|(C_{1}(T))^{2}|\leq\kappa_{2}$ for all $T$. Moreover, since  $\tilde{k}_{i, j}(T)$ converges to $1$ or $0$ as $T\rightarrow\infty$, we may assume  that $|\tilde{k}_{i, j}(T)|< e/(e - 1) = \kappa_{2}$ for sufficiently large $T$. Also, using the fact that $|\delta_{i , j}c_{i, j}|\leq 1$ for all $1\leq i, j\leq N$, we have, for sufficiently large $T$,
\begin{align}\label{error term in in difference in kth moment}
\begin{split}
&\Big|\sum_{\ell = 0}^{k - 1}\Big(\sum_{1\leq i, j\leq N}u_{i}u_{j}(C_{1}(T))^{2}\delta_{i, j}c_{i, j}\Big)^{k - 1 - \ell}\Big(\sum_{1\leq i, j\leq N}u_{i}u_{j}\tilde{k}_{i, j}(T)\Big)^{\ell}\Big|
\\&\leq\Big|\sum_{\ell = 0}^{k - 1}\kappa_{2}^{k - 1 - \ell}\kappa_{2}^{\ell}\Big(\sum_{1\leq i, j\leq N}u_{i}u_{j}\Big)^{k - 1 - \ell}\Big(\sum_{1\leq i, j\leq N}u_{i}u_{j}\Big)^{\ell}\Big|
\leq\Big|\sum_{\ell = 0}^{k - 1}\kappa_{2}^{k - 1 - \ell}\kappa_{2}^{\ell}\Big(\|{\bf u}\|_{2}^{2}\Big)^{k - 1 - \ell}\Big(\|{\bf u}\|_{2}^{2}\Big)^{\ell}\Big|
\\&=\frac{(k - 1)k}{2}\kappa_{2}^{k - 1}\|{\bf u}\|_{2}^{2k - 2}.
\end{split}
\end{align}
 On the other hand, by \eqref{control of C_{1}} and \eqref{control of tilde k}, when $T>\max\{T_{C_{1}}, T_{\tilde{k}}\}$, we have
\begin{align*}
|(C_{1}(T))^{2}\delta_{i, j}c_{i, j} - \tilde{k}_{i, j}(T)|
&\leq |(C_{1}(T))^{2}\delta_{i, j}c_{i, j} - \delta_{i, j}c_{i, j}| + |\delta_{i, j}c_{i, j} - \tilde{k}_{i, j}(T)|
\\&\leq |\delta_{i, j}c_{i, j}| |(C_{1}(T))^{2} - 1| +  |\delta_{i, j}c_{i, j} - \tilde{k}_{i, j}(T)|
\leq \varepsilon_{5} + \varepsilon_{6},
\end{align*}
where the last inequality follows from $|\delta_{i, j}c_{i, j}|\leq 1$.
Therefore,
\begin{align}\label{main term in difference in kth moment}
\begin{split}
&\Big|\sum_{1\leq i, j\leq N}u_{i}u_{j}(C_{1}(T))^{2}\delta_{i, j}c_{i, j} - \sum_{1\leq i, j\leq N}u_{i}u_{j}\tilde{k}_{i, j}(T)\Big|
\\&\leq\sum_{1\leq i, j\leq N}u_{i}u_{j}|(C_{1}(T))^{2}\delta_{i, j}c_{i, j} - \tilde{k}_{i, j}(T)|
\leq\sum_{1\leq i, j\leq N}u_{i}u_{j}|(\varepsilon_{5} + \varepsilon_{6})|
\leq \|{\bf u}\|_{2}^{2} (\varepsilon_{5} + \varepsilon_{6}).
\end{split}
\end{align}

By taking $\varepsilon_{5}+\varepsilon_{6} = 2T^{-1}$ and plugging \eqref{main term in difference in kth moment} and \eqref{error term in in difference in kth moment} in to \eqref{difference in kth moment},we then infer from \eqref{2k moment of R^{1}_{T}} that
\begin{align*}
&\Big|\mathbb{E}[ ({\bf u}\cdot {\bf R}^{1}_{T})^{2k}] - \mathbb{E}[ ({\bf u}\cdot \tilde{{\bf Z}}_{N})^{2k}] \Big|
\\&\leq \frac{(2k)!}{k! 2^{k}}\frac{\|{\bf u}\|_{2}^{2k}\kappa_{2}^{k - 1} (k - 1)k }{T}
+O\Big(\|{\bf u}\|_{1}^{2k}(C_{1}(T))^{2k}\frac{(2k)!k}{k!2^{k}}\frac{\Delta(T) + \delta(T)}{\log\log T}\Big) 
\\&+ O\Big(\|{\bf u}\|_{1}^{2k}(C_{1}(T))^{2k}\frac{(2k)!k}{k!2^{k}}\frac{\delta(T)}{\log\log T}\Big) + O\Big(\|{\bf u}\|_{1}^{2k}\frac{(2k)!}{k!2^{k}} \frac{2k(k - 1)}{(\log\log T)^{2 - 3\varepsilon}}\Big) + O\Big(\|{\bf u}\|_{1}^{2k} \frac{(2k)!}{k!2^{k}}\frac{3^{k}(k + 1)(k + 2)}{(\log\log T)^{3 - 4\varepsilon}}\Big)
\\&+ O\Big(\frac{(2k)!k^{2}}{k!2^{k}}\frac{1}{(\log\log T)^{2}}\Big)
 + O\Big(\frac{((2k)!)^{2}k}{(k!)^{2}2^{k}}\frac{T^{2k/(K'\log\log T)}}{T(\log\log T)^{k}}\Big).
\end{align*}
Hence, we complete the proof.
\end{proof}

Lemma \ref{precise difference} implies that $P_{1}(s_{0, j}, \chi_{j})/\sqrt{\oh\mathfrak{M}_{T, \chi_{j}}}$ converges to the standard normal in distribution. Using this result, we can refine the estimate \eqref{P_{1} estimate in lemma 7}.

\begin{lemma}\label{exponential estimate of P_{1}}
For any $\mathfrak{r}>0$, we have
\begin{align}\label{exponential estimate of P_{1} general}
\mathfrak{P}_{T}(|P_{1}(s_{0, j}, \chi_{j})| > \mathfrak{r})
\leq\sqrt{\frac{1}{\pi}}\frac{\sqrt{\log\log T}}{\mathfrak{r}}\exp\Big(-\frac{\mathfrak{r}^{2}}{\log\log T}\Big) + \frac{1}{T}.
\end{align}
In particular, when $\mathfrak{r} = \log\log T$, we have
\begin{align*}
\mathfrak{P}_{T}(|P_{1}(s_{0, j}, \chi_{j})|>\log\log T)\ll \frac{1}{(\sqrt{\log\log T}) \log T} + \frac{1}{T}.
\end{align*}
\end{lemma}

\begin{proof}
Let  $\mathcal{N}$ be the standard normal distribution. By Lemma \ref{precise difference}, $P_{1}(s_{0, j}, \chi_{j})/\sqrt{\oh\mathfrak{M}_{T, \chi_{j}}}$ converges to $\mathcal{N}$ in distribution, we have
\begin{align*}
\lim_{T\rightarrow\infty}\mathfrak{P}_{T}(P_{1}(s_{0, j}, \chi_{j})/\sqrt{\oh\mathfrak{M}_{T, \chi_{j}}} > r ) = \mathfrak{P}\left(\mathcal{N} > r\right)
\end{align*}
for every $r\in\mathbb{R}$. Since $\mathcal{N}$ is symmetric, we obtain
\begin{align*}
\lim_{T\rightarrow\infty}\mathfrak{P}_{T}(|P_{1}(s_{0, j}, \chi_{j})/\sqrt{\oh\mathfrak{M}_{T, \chi_{j}}}| > r) = \mathfrak{P}\left(|\mathcal{N}| > r\right)
\end{align*}
for every $r>0$.
Therefore,  there exists $T_{1}>0$ such that
\begin{align}\label{mu and nu}
\mathfrak{P}_{T}(|P_{1}(s_{0, j}, \chi_{j})/\sqrt{\oh\mathfrak{M}_{T, \chi_{j}}}| > r)<  \mathfrak{P}\left(|\mathcal{N}| > r\right) + \frac{1}{T}.
\end{align}
whenever $T>T_{1}$ for every $r>0$. Using \cite[Eg. (3.6.3)]{CB}, we further have
\begin{align*}
\mathfrak{P}_{T}(|P_{1}(s_{0, j}, \chi_{j})/\sqrt{\oh\mathfrak{M}_{T, \chi_{j}}}| > r)\leq \sqrt{\frac{2}{\pi}}\frac{e^{-r^{2}/2}}{r} + \frac{1}{T}.
\end{align*}
Now, taking $r = \mathfrak{r}/\sqrt{\oh\mathfrak{M}_{T, \chi_{j}}}$, we obtain
\begin{align}
\begin{split}\label{almost-done-exponential-estimate}
\mathfrak{P}_{T}(|P_{1}(s_{0, j}, \chi_{j})| > \mathfrak{r})
&\leq\sqrt{\frac{2}{\pi}}\frac{\sqrt{\oh\mathfrak{M}_{T, \chi_{j}}}}{\mathfrak{r}}\exp\Big(-\frac{1}{2}\frac{\mathfrak{r}^{2}}{\frac{1}{2}\mathfrak{M}_{T, \chi_{j}}}\Big) + \frac{1}{T}
=\sqrt{\frac{1}{\pi}}\frac{\sqrt{\mathfrak{M}_{T, \chi_{j}}}}{\mathfrak{r}}\exp\Big(-\frac{\mathfrak{r}^{2}}{\mathfrak{M}_{T, \chi_{j}}}\Big) + \frac{1}{T}.
\end{split}
\end{align}
Using \eqref{calculation of frak M} in \eqref{almost-done-exponential-estimate}, we complete the proof.
\end{proof}

We also require the following technical estimates.
\begin{lemma}\label{magic exponential approximation}
If $n\geq 7.5(|z| + 1)$, then
$
|e^{z} - \sum_{j = 0}^{n}\frac{z^{j}}{j!} |
< e^{-n}.
$
\end{lemma}

\begin{proof}
 By the Stirling formula, we have that $j!\geq (j/e)^{j}$. Thus,
\begin{align*}
\Big|e^{z} - \sum_{j = 0}^{n}\frac{z^{j}}{j!}\Big|
\leq \sum_{j = n + 1}^{\infty}\frac{|z|^{j}}{j!}
\leq\sum_{j =  n + 1}^{\infty}\frac{|z|^{j}}{(j/e)^{j}}.
\end{align*}
For $n\geq 7.5(|z| + 1)$, we further have
\begin{align*}
\Big|e^{z} - \sum_{j = 0}^{n}\frac{z^{j}}{j!}\Big|
&\leq\sum_{j =  n + 1}^{\infty}\frac{|z|^{j}}{(j/e)^{j}}
\leq\sum_{j = n + 1}^{\infty}\frac{(|j|/7.5)^{j}}{(j/e)^{j}}
=\sum_{j = n + 1}^{\infty}\Big(\frac{e}{7.5}\Big)^{j}
=\Big(1 - \frac{e}{7.5}\Big)^{-1}\Big(\frac{e}{7.5}\Big)^{n + 1}
\\&=\Big(1 - \frac{e}{7.5}\Big)^{-1}\Big(\frac{7.5}{e}\Big)^{-(n + 1)}
=\Big(1 - \frac{e}{7.5}\Big)^{-1} \Big(\frac{7.5}{e}\Big)^{-1}\Big(\frac{7.5}{e}\Big)^{- n}
=\Big(\frac{e}{7.5 - e}\Big) \Big(\frac{7.5}{e}\Big)^{- n}
<\Big(\frac{7.5}{e}\Big)^{- n}
< e^{-n}
\end{align*}
since $7.5/e >e$.
\end{proof}


\begin{remark}
We now comment here some crucial steps for proving Proposition \ref{Prop 6 in AR24}.\\
\noindent (i) Lemma \ref{covariance lemma} is essential since there is no a priori reason that the ``covariance matrix'' of ${\bf R}^{1}_{T}$ is positive-definite. Indeed, the matrix $\widetilde{ \mathfrak{K} }(T)$ may \emph{not} be positive-definite if $T$ is not large enough (cf. \cite[p. 3362,  the displayed equation below Eq. (19)]{AR24}).\\
\noindent (ii) In Lemma \ref{precise difference}, we trace the contribution of $k$ in the big-O terms (cf. \cite[p. 3367, Eq. (23) and below]{AR24}). This will be crucial in the proof of Lemma \ref{lemma 9' in AR24}. \\
 \noindent (iii) Lemma \ref{exponential estimate of P_{1}} is not just a sharper estimate. It provides a flexibility to choose $\mathfrak{r}$, and such a freedom in choosing parameters is crucial in Lemma \ref{lemma 9' in AR24}. \\
\noindent (iv) In the proof of Lemma \ref{lemma 9' in AR24}, we parametrised the argument so that one can easily adopt it to a different set of parameters (Lemma \ref{better-lemma 9}).  \\
\noindent  (v) As shall be seen later, we choose a new set of parameters that is different from Roberts' proof \cite[Sec. 2.8]{AR24}. The term $\mathcal{E}_{2}$ in \eqref{main estimate in lemma 7.9} explains the necessity. If we choose $\|{\bf u}\|_{1} = \sqrt{\log\log T}$ (\cite[Eq. (22)]{AR24}) and $\mathfrak{r} = \log\log T$ (\cite[Proof of Lemma 9]{AR24}) as in \cite[Sec. 2.8]{AR24}, then \eqref{error2} becomes
\begin{align}\label{explain-the-shit-of-AR24}
\mathcal{E}_{2}
\ll \|{\bf u}\|_{1}^{\mathfrak{N}}\cdot(\mathfrak{P}_{T}(|P_{1}(s_{0, j})|>\log\log T))^{1/2}.
\end{align}
Plugging $\|{\bf u}\|_{1} = \sqrt{\log\log T}$ and $\mathfrak{r} = \log\log T$ into \eqref{final choice of mathfrak N}, we obtain $\mathfrak{N} = \sqrt{2}C_{1}(T)\mathcal{C}_{2}\log\log T$. Using this choice of $\mathfrak{N}$ and \eqref{P_{1} estimate in lemma 7}, we see that \eqref{explain-the-shit-of-AR24} becomes
\begin{align}\label{explain-the-shit-of-AR24-end}
\begin{split}
\mathcal{E}_{2}
&\ll (\sqrt{\log\log T})^{ \sqrt{2}C_{1}(T)\mathcal{C}_{2}\log\log T}\cdot\frac{1}{\sqrt{\log T}}
\\&= \exp\Big( (\sqrt{2}C_{1}(T)\mathcal{C}_{2}\log\log T)(\log\sqrt{\log\log T}) - \log\sqrt{\log T}\Big)
\\&=\exp\Big(\frac{\sqrt{2}C_{1}(T)\mathcal{C}_{2}}{2}(\log\log T)(\log\log\log T) - \frac{1}{2}\log\log T\Big),
\end{split}
\end{align}
which \emph{diverges} as $T\rightarrow\infty$ (cf. \cite[p. 3367, the displayed equation below Eq. (23)]{AR24}).\\
\noindent (iv) Equation \eqref{explain-the-shit-of-AR24-end} is one of the reasons that the present authors choose a new set of parameters that differs from \cite[Sec. 2.8]{AR24}. Also, it shows the necessity of the flexibility of choosing $\mathfrak{r}$ in Lemma \ref{exponential estimate of P_{1}}.
\end{remark}


Now, we return to the proof. With all the preparation above, we shall show that, as stochastic processes, ${\bf R}^{1}_{T}$ is similar to the Gaussian process $\tilde{\bf Z}_{N} = (\tilde{\mathcal{Z}}_{1}, \ldots, \tilde{\mathcal{Z}}_{N})$ in the sense that the difference between the corresponding characteristic functions is small.

\begin{lemma}\label{lemma 9' in AR24}
Suppose that $\Delta(T)$ and $\delta(T)$ are of the order $O((\log\log T)^{\varepsilon})$ for a fixed $\varepsilon\in(0, \frac{2}{3})$ and not identically zero. Let ${\bf u} = (u_{1}, u_{2}, \ldots u_{N})\in\mathbb{R}^{N}$. Then with the notation introduce above, we have
\begin{align}\label{Goal in lemma 9' in AR24}
\Big|\mathbb{E}[\exp(\mi ({\bf u}\cdot {\bf R}^{1}_{T}))] - \mathbb{E}[\exp(\mi ({\bf u}\cdot\tilde{\bf Z}_{N}))]\Big|\ll\exp\Big(-\frac{1}{2}(\log\log\log T)^{\varepsilon_{1} + \varepsilon_{2}}\Big),
\end{align}
where $\varepsilon_{j}>0, j  = 1, 2$, such that $\varepsilon_{1} + \varepsilon_{2}<1$.
\end{lemma}

\begin{proof}

Recall that $s_{0, j} = \sigma_{0} + \mi (U + \alpha_{j})$. For any $\mathfrak{r}>0$, we have  
\begin{align}\label{main estimate in lemma 7.9}
\begin{split}
&\Big|\mathbb{E}[\exp(\mi ({\bf u}\cdot {\bf R}^{1}_{T}))] - \mathbb{E}[\exp(\mi ({\bf u}\cdot\tilde{\bf Z}_{N}))]\Big|
\\&\leq \Big|\mathbb{E}[\exp(\mi ({\bf u}\cdot {\bf R}^{1}_{T}))\cdot\mathds{1}(\cap_{j = 1}^{N}\{|P_{1}(s_{0, j})|\leq \mathfrak{r}\})]  -  \mathbb{E}[\exp(\mi ({\bf u}\cdot \tilde{{\bf Z}}_{N}))]\Big|
\\& + \mathbb{E}[\exp(\mi ({\bf u}\cdot {\bf R}^{1}_{T}))\cdot  \mathds{1}(\cup_{j = 1}^{N}\{|P_{1}(s_{0, j})|> \mathfrak{r}\})]
\\&\leq \sum_{n\leq \mathfrak{N}}\frac{1}{n!}\Big|\mathbb{E}[ ({\bf u}\cdot {\bf R}^{1}_{T})^{n}] - \mathbb{E}[ ({\bf u}\cdot \tilde{{\bf Z}}_{N})^{n}]\Big|
+\Big| \sum_{n>\mathfrak{N}}\frac{1}{n!}\mi^{n}\mathbb{E}[ ({\bf u}\cdot {\bf R}^{1}_{T})^{n}\cdot\mathds{1}(\cap_{j = 1}^{N}\{|P_{1}(s_{0, j})|\leq \mathfrak{r}\})]\Big|
\\&+ \Big|\sum_{n\leq \mathfrak{N}}\frac{1}{n!}\mi^{n}\mathbb{E}[ ({\bf u}\cdot {\bf R}^{1}_{T})^{n}\cdot  \mathds{1}(\cup_{j = 1}^{N}\{|P_{1}(s_{0, j})|> \mathfrak{r}\})]\Big|
+\Big|\sum_{n>\mathfrak{N}}\frac{1}{n!}\mi^{n}\mathbb{E}[ ({\bf u}\cdot \tilde{{\bf Z}}_{N})^{n}]\Big|
\\& + \mathbb{E}[\exp(\mi ({\bf u}\cdot {\bf R}^{1}_{T}))\cdot  \mathds{1}(\cup_{j = 1}^{N}\{|P_{1}(s_{0, j})|> \mathfrak{r}\})]
\\
&=\sum_{n\leq \mathfrak{N}}\frac{1}{n!}\Big|\mathbb{E}[ ({\bf u}\cdot {\bf R}^{1}_{T})^{n}] - \mathbb{E}[ ({\bf u}\cdot \tilde{{\bf Z}}_{N})^{n}]\Big|
+\mathcal{E}_{1} + \mathcal{E}_{2} + \mathcal{E}_{3} + \mathcal{E}_{4}.
\end{split}
\end{align}
We begin with $\mathcal{E}_{1}$. Notice that we are on the event where $|P_{1}(s_{0,j})|\leq\mathfrak{r}$. Therefore, by \eqref{the very definition of C_{1}(T)}, we have
\begin{align*}
{\bf u}\cdot {\bf R}^{1}_{T}
=  \sum_{j = 1}^{N}u_{j}\frac{P_{1}(s_{0, j})}{\sqrt{\oh\mathfrak{M}_{T, \chi_j}}}
\leq \sum_{j = 1}^{N}u_{j}\frac{\mathfrak{r}}{\sqrt{\oh\mathfrak{M}_{T, \chi_j}}}
=\sum_{j = 1}^{N}u_{j}\frac{\mathfrak{r} C_{1}(T)}{\sqrt{\oh\log\log T}}
\leq \frac{\mathfrak{r} C_{1}(T)}{\sqrt{\oh\log\log T}} \|{\bf u}\|_{1},
\end{align*}
which implies that
\begin{align*}
\Big|\sum_{n>\mathfrak{N}}\frac{1}{n!}\mi^{n}\mathbb{E}[ ({\bf u}\cdot {\bf R}^{1}_{T})^{n}\cdot\mathds{1}(\cap_{j = 1}^{N}\{|P_{1}(s_{0, j})|\leq \mathfrak{r}\})]\Big|
\leq \sum_{n>\mathfrak{N}}\frac{1}{n!}\bigg(\frac{\mathfrak{r} C_{1}(T)}{\sqrt{\oh\log\log T}} \|{\bf u}\|_{1}\bigg)^{n}.
\end{align*}

Take the parameters in the following fashion:
Let $\mathcal{C}_{2}\in\mathbb{R}$ such that
\begin{align}\label{final choice of C_{2}}
\mathcal{C}_{2}>7.5.
\end{align}
For any $\varepsilon_{2}>0$, take
\begin{align}\label{final choice of mathfrak r}
\mathfrak{r} = \frac{\sqrt{\oh\log\log T}}{C_{1}(T)\mathcal{C}_{2}}(\log\log\log T)^{\varepsilon_{2}},
\end{align}
where $C_{1}(T)$ is introduced in \eqref{the very definition of C_{1}(T)}. For any $\varepsilon_{1}>0$ such that $\varepsilon_{1}<\varepsilon_{2}$ and  that $\varepsilon_{1} + \varepsilon_{2}<1$, take
\begin{align}\label{final choice of F}
\|{\bf u}\|_{1}  = (\log\log\log T)^{\varepsilon_{1}}.
\end{align}
Let
\begin{align}\label{final choice of mathfrak N}
\mathfrak{N} =\frac{C_{1}(T)\mathcal{C}_{2}}{\sqrt{\oh\log\log T}}\mathfrak{r}\|{\bf u}\|_{1}.
\end{align}
Then by \eqref{final choice of mathfrak r} and \eqref{final choice of F}, we have
\begin{align}\label{simplified mathfrak N}
\mathfrak{N} = (\log\log\log T)^{\varepsilon_{1} + \varepsilon_{2}};
\end{align}
hence
\begin{align}\label{NlogF}
\mathfrak{N}\log \|{\bf u}\|_{1} = (\log\log\log T)^{\varepsilon_{1} + \varepsilon_{2}}(\varepsilon_{1}\log\log\log\log T).
\end{align}

By \eqref{final choice of mathfrak N} and \eqref{the very definition of C_{1}(T)}, we see that
\begin{align*}
\mathfrak{N}> 7.5\bigg(\frac{\mathfrak{r} C_{1}(T)}{\sqrt{\oh\log\log T}} \|{\bf u}\|_{1} + 1\bigg)
\end{align*}
for sufficiently large $T$. Therefore, we conclude from Lemma \ref{magic exponential approximation} that
\begin{align}\label{error1}
\begin{split}
\mathcal{E}_{1}&= 
\Big|\sum_{n>\mathfrak{N}}\frac{1}{n!}\mi^{n}\mathbb{E}[ ({\bf u}\cdot {\bf R}^{1}_{T})^{n}\cdot\mathds{1}(\cap_{j = 1}^{N}\{|P_{1}(s_{0, j})|\leq \mathfrak{r}\})]\Big|
\leq \sum_{n>\mathfrak{N}}\frac{1}{n!}\bigg(\frac{\mathfrak{r} C_{1}(T)}{\sqrt{\oh\log\log T}} \|{\bf u}\|_{1}\bigg)^{n}\\
&<e^{- \mathfrak{N}}
=  \exp(- (\log\log\log T)^{\varepsilon_{1} + \varepsilon_{2}})
\end{split}
\end{align}
by \eqref{simplified mathfrak N}.

Next, we consider $\mathcal{E}_{2}$. By the Cauchy-Schwarz inequality, we have 
\begin{align*}
\bigg|\mathbb{E}\left[ ({\bf u}\cdot {\bf R}^{1}_{T})^{n}\cdot  \mathds{1}\left\{\cup_{j = 1}^{N}\left\{|P_{1}(s_{0, j})|> \mathfrak{r}\right\}\right\}\right]\bigg|
\leq\bigg(\mathbb{E}\left[ ({\bf u}\cdot {\bf R}^{1}_{T})^{2n}\right]\bigg)^{1/2}\bigg(\mathbb{E}\left[  \mathds{1}\left\{\cup_{j = 1}^{N}\left\{|P_{1}(s_{0, j})|> \mathfrak{r}\right\}\right\}\right]^{2}\bigg)^{1/2}.
\end{align*}
We infer from \eqref{exponential estimate of P_{1} general} that
\begin{align*}
\Big(\mathbb{E}[  \mathds{1}(\cup_{j = 1}^{N}\{|P_{1}(s_{0, j})|> \mathfrak{r}\})]^{2}\Big)^{1/2}
&\leq\Big(\sum_{j = 1}^{N}\mathfrak{P}_{T}(\{|P_{1}(s_{0, j})|> \mathfrak{r}\})\Big)^{1/2}
\\&\leq \sqrt{N}\Big(\sqrt{\frac{1}{\pi}}\frac{\sqrt{\log\log T}}{\mathfrak{r}}\exp\Big(-\frac{\mathfrak{r}^{2}}{\log\log T}\Big) + \frac{1}{T}\Big)^{1/2}.
\end{align*}
By Lemma \ref{moments of R^{1}_{T}}, we have
\begin{align*}
\sum_{n\leq \mathfrak{N}}\frac{1}{n!}\Big(\mathbb{E}[ ({\bf u}\cdot {\bf R}^{1}_{T})^{2n}]\Big)^{1/2}
&<\sum_{n\leq \mathfrak{N}}\frac{1}{n!}\Big(\|{\bf u}\|_1^{2n} \cdot   \kappa_1  (1+ \mathcal{C}^2) \frac{(2n)!}{n! 2^n}\Big)^{1/2}
\leq\|{\bf u}\|_{1}^{\mathfrak{N}}\sqrt{\kappa_1  (1+ \mathcal{C}^2)}\sum_{n\leq\mathfrak{N}}\sqrt{\frac{(2n)!}{(n!)^{3}2^{n}}}\\
&\ll\|{\bf u}\|_{1}^{\mathfrak{N}},
\end{align*}
where the last $\ll$ follows from the fact that the series converges. Therefore, 
\begin{align}\label{error2}
\begin{split}
\mathcal{E}_{2}
\ll \|{\bf u}\|_{1}^{\mathfrak{N}}\cdot \sqrt{N}\Big(\sqrt{\frac{1}{\pi}}\frac{\sqrt{\log\log T}}{\mathfrak{r}}\exp\Big(-\frac{\mathfrak{r}^{2}}{\log\log T}\Big) + \frac{1}{T}\Big)^{1/2}.
\end{split}
\end{align}
Using the inequality $\sqrt{a + b}\leq \sqrt{a} + \sqrt{b}$ in \eqref{error2}, we see that
\begin{align}\label{error2_1}
\begin{split}
\mathcal{E}_{2}
\ll\sqrt{N} \|{\bf u}\|_{1}^{\mathfrak{N}}\cdot\Big(\sqrt{\frac{1}{\pi}}\frac{\sqrt{\log\log T}}{\mathfrak{r}}\exp\Big(-\frac{\mathfrak{r}^{2}}{\log\log T}\Big)\Big)^{1/2} +\sqrt{N}\frac{\|{\bf u}\|_{1}^{\mathfrak{N}}}{\sqrt{T}}.
\end{split}
\end{align}

Plugging \eqref{final choice of mathfrak r}, \eqref{final choice of F}, and \eqref{simplified mathfrak N} into \eqref{error2_1}, we infer from \eqref{the very definition of kappa5} that the first term is
\begin{align*}
&\ll\sqrt{N} (\log\log\log T)^{\varepsilon_{1}(\log\log\log T)^{\varepsilon_{1} + \varepsilon_{2}}}\Big(\frac{\sqrt{2\kappa_{2}}\mathcal{C}_{2}}{(\log\log\log T)^{\varepsilon_{2}}}\exp\Big(- \frac{(\log\log\log T)^{2\varepsilon_{2}}}{2\kappa_{2}\mathcal{C}_{2}^{2}}\Big)\Big)^{1/2}
\\
&=\sqrt{N}\exp\Big(\varepsilon_{1}(\log\log\log T)^{\varepsilon_{1} + \varepsilon_{2}}(\log\log\log\log T)
\\& + \frac{1}{2}\Big(\log(\sqrt{2\kappa_{2}}\mathcal{C}_{2}) - \varepsilon_{2}(\log\log\log\log T) - \frac{(\log\log\log T)^{2\varepsilon_{2}}}{2\kappa_{2}\mathcal{C}_{2}^{2}}\Big)\Big)
\\&\leq\sqrt{N}\exp\Big(\varepsilon_{1}(\log\log\log T)^{\varepsilon_{1} + \varepsilon_{2}}(\log\log\log\log T) - \frac{\varepsilon_{2}}{2}(\log\log\log\log T) - \frac{(\log\log\log T)^{2\varepsilon_{2}}}{4\kappa_{2}\mathcal{C}_{2}^{2}}\Big)
\end{align*}
for sufficiently large $T$.

Recall that $\varepsilon_{1}<\varepsilon_{2}$ (one line above \eqref{final choice of F}). Therefore, both $\frac{\log\log\log\log T}{(\log\log\log T)^{\varepsilon_{2} - \varepsilon_{1}}}\rightarrow 0$ and $\frac{\log\log\log\log T}{(\log\log\log T)^{2\varepsilon_{2}}}\rightarrow 0$ as $T\rightarrow\infty$. We keep simplifying to obtain
\begin{align*}
&\sqrt{N}\exp\Big(\varepsilon_{1}(\log\log\log T)^{\varepsilon_{1} + \varepsilon_{2}}(\log\log\log\log T) - \frac{\varepsilon_{2}}{2}(\log\log\log\log T) - \frac{(\log\log\log T)^{2\varepsilon_{2}}}{4\kappa_{2}\mathcal{C}_{2}^{2}}\Big)
\\
&=\sqrt{N}\exp\Big((\log\log\log T)^{2\varepsilon_{2}}\Big(\varepsilon_{1}\frac{\log\log\log\log T}{(\log\log\log T)^{\varepsilon_{2} - \varepsilon_{1}}} - \frac{\varepsilon_{2}}{2}\frac{\log\log\log\log T}{(\log\log\log T)^{2\varepsilon_{2}}} - \frac{1}{4\kappa_{2}\mathcal{C}_{2}^{2}}\Big)\Big)
\\&\ll\sqrt{N}\exp\Big(- \Big(\frac{1}{4\kappa_{2}\mathcal{C}_{2}^{2}} -\varepsilon'\Big)(\log\log\log T)^{2\varepsilon_{2}}\Big),
\end{align*}
where $\varepsilon'>0$ can be arbitrarily small.

Plugging \eqref{final choice of mathfrak r}, \eqref{final choice of F}, and \eqref{simplified mathfrak N} into \eqref{error2_1}, we see that the second term becomes
\begin{align*}
\sqrt{N}\exp\Big(\varepsilon_{1}(\log\log\log T)^{\varepsilon_{1} + \varepsilon_{2}}(\log\log\log\log T) - \frac{1}{2}\log T\Big)
\ll \sqrt{N}\exp\Big(-\Big(\frac{1}{2} - \varepsilon'\Big)(\log T)\Big).
\end{align*}

Putting everything together, we have that
\begin{align}\label{error2 final}
\begin{split}
\mathcal{E}_{2}
&\ll \sqrt{N}\exp\Big(- \Big(\frac{1}{4\kappa_{2}\mathcal{C}_{2}^{2}} -\varepsilon'\Big)(\log\log\log T)^{2\varepsilon_{2}}\Big) + \sqrt{N}\exp\Big(-\Big(\frac{1}{2} - \varepsilon'\Big)(\log T)\Big)
\\&\ll \sqrt{N}\exp\Big(- \Big(\frac{1}{4\kappa_{2}\mathcal{C}_{2}^{2}} -\varepsilon'\Big)(\log\log\log T)^{2\varepsilon_{2}}\Big).
\end{split}
\end{align}

For $\mathcal{E}_{3}$, we see that $\mathbb{E}[({\bf u}\cdot \tilde{{\bf Z}}_{N})^{n}] = 0$ if $n = 2k + 1$ by the property of moments of normal distribution. Moreover, by Lemma \ref{moments of tilde Z_{N}}, we further have
\begin{align*}
\Big|\sum_{n>\mathfrak{N}}\frac{1}{n!}\mi^{n}\mathbb{E}[ ({\bf u}\cdot \tilde{{\bf Z}}_{N})^{n}]\Big|
\leq\sum_{k>\mathfrak{N}/2}\frac{1}{(2k)!}\frac{(2k)!}{(k)! 2^{k}}\|{\bf u}\|^{2k}_{2}
= \sum_{k>\mathfrak{N}/2}\frac{\|{\bf u}\|^{2k}_{2}}{(k)!2^{k}}
\leq\sum_{k>\mathfrak{N}/2}\frac{\|{\bf u}\|^{2k}_{1}}{(k)!2^{k}}
=\sum_{k>\mathfrak{N}/2}\frac{(\|{\bf u}\|^{2}_{1}/2)^{k}}{(k)!}
\end{align*}
where the second last inequality follows from  $\|{\bf u}\|_{2}\leq \|{\bf u}\|_{1}$. By \eqref{final choice of F} and \eqref{simplified mathfrak N}, for $T$ sufficiently large, we have 
$\mathfrak{N}/2> 7.5 ((\|{\bf u}\|^{2}_{1}/2) + 1)$. 
Therefore,by Lemma \ref{magic exponential approximation}, we have
\begin{align}\label{error3}
\begin{split}
\mathcal{E}_{3}
= \Big|\sum_{n>\mathfrak{N}}\frac{1}{n!}\mi^{n}\mathbb{E}[ ({\bf u}\cdot \tilde{{\bf Z}}_{N})^{n}]\Big|
\leq \sum_{k>\mathfrak{N}/2}\frac{(\|{\bf u}\|^{2}_{1}/2)^{k}}{(k)!}
\leq e^{-\mathfrak{N}/2}
= \exp\Big(-\frac{1}{2}(\log\log\log T)^{\varepsilon_{1} + \varepsilon_{2}}\Big).
\end{split}
\end{align}

For $\mathcal{E}_{4}$, by the property that $|e^{\mi\theta}|\leq 1$ for all $\theta$, we have
\begin{align}\label{error4}
\begin{split}
\mathcal{E}_{4} 
&= \mathbb{E}[\exp (\mi ({\bf u}\cdot {\bf R}^{1}_{T}) )\cdot  \mathds{1}(\cup_{j = 1}^{N}\{|P_{1}(s_{0, j})|> \mathfrak{r}\})]
\\
&\leq\sum_{j = 1}^{N}\mathfrak{P}_{T}(|P_{1}(s_{0, j})|>\mathfrak{r})
\leq N\sqrt{\frac{1}{\pi}}\frac{\sqrt{\log\log T}}{\mathfrak{r}}\exp\Big(-\frac{\mathfrak{r}^{2}}{\log\log T}\Big) + \frac{N}{T},
\end{split}
\end{align}
where the last inequality follows from \eqref{exponential estimate of P_{1} general}. Plugging \eqref{final choice of mathfrak r} into \eqref{error4} and using \eqref{the very definition of kappa5}, we have
\begin{align*}
\frac{\sqrt{\log\log T}}{\mathfrak{r}}\exp\Big(-\frac{\mathfrak{r}^{2}}{\log\log T}\Big)
\leq \frac{\sqrt{2\kappa_{2}}\mathcal{C}_{2}}{(\log\log\log T)^{\varepsilon_{2}}}\exp\Big(- \frac{(\log\log\log T)^{2\varepsilon_{2}}}{2\kappa_{2}\mathcal{C}_{2}^{2}}\Big)
\end{align*}
and thus
\begin{align}\label{error4-final}
\begin{split}
\mathcal{E}_{4}
&\ll N \frac{\sqrt{2\kappa_{2}}\mathcal{C}_{2}}{(\log\log\log T)^{\varepsilon_{2}}}\exp\Big(- \frac{(\log\log\log T)^{2\varepsilon_{2}}}{2\kappa_{2}\mathcal{C}_{2}^{2}}\Big)
\\& \ll N \exp\Big(\log\sqrt{2\kappa_{2}}\mathcal{C}_{2} - \varepsilon_{2}\log\log\log\log T - \frac{(\log\log\log T)^{2\varepsilon_{2}}}{2\kappa_{2}\mathcal{C}_{2}^{2}}\Big)
\\& \ll N \exp\Big( - \Big(\frac{1}{2\kappa_{2}\mathcal{C}_{2}^{2}} - \varepsilon'\Big)(\log\log\log T)^{2\varepsilon_{2}}\Big).
\end{split}
\end{align}

Now, we consider the main term in \eqref{main estimate in lemma 7.9}:
\begin{align}\label{main estimate in lemma 7.9_1}
\begin{split}
\sum_{n\leq \mathfrak{N}}\frac{1}{n!}\Big|\mathbb{E}[ ({\bf u}\cdot {\bf R}^{1}_{T})^{n}] - \mathbb{E}[ ({\bf u}\cdot \tilde{{\bf Z}}_{N})^{n}]\Big|
&=\sum_{2k + 1\leq\mathfrak{N}} \frac{1}{(2k + 1)!}\Big|\mathbb{E}[ ({\bf u}\cdot {\bf R}^{1}_{T})^{2k + 1}] - \mathbb{E}[ ({\bf u}\cdot \tilde{{\bf Z}}_{N})^{2k + 1}]\Big| 
\\&+ \sum_{2k\leq\mathfrak{N}}\frac{1}{(2k)!}\Big|\mathbb{E}[ ({\bf u}\cdot {\bf R}^{1}_{T})^{2k}] - \mathbb{E}[ ({\bf u}\cdot \tilde{{\bf Z}}_{N})^{2k}]\Big|.
\end{split}
\end{align}

By \eqref{precise difference odd}, the first term on the right of \eqref{main estimate in lemma 7.9_1} is
\begin{align*}
&\ll \frac{1}{T}\sum_{k\leq(\mathfrak{N} - 1)/2}\frac{(C_{1}(T))^{2k+ 1}}{\big(\sqrt{\oh\log\log T}\big)^{2k + 1}}((2k + 1)!)(\|{\bf u}\|_1 T^{(K'\log\log T)^{-1}})^{4k + 2}
\\&\ll \frac{1}{T}\sum_{k\leq(\mathfrak{N} - 1)/2}\frac{(C_{1}(T))^{2k+ 1}}{\big(\sqrt{\oh\log\log T}\big)^{2k + 1}}((2k + 1)!)((\log\log\log T)^{\varepsilon_{1}} T^{(K'\log\log T)^{-1}})^{4k + 2},
\end{align*}
where the second $\ll$ follows from \eqref{final choice of F}. Since the series diverges, the terms with largest exponent dominates. Hence, by using the Stirling formula \eqref{Stirling formula}, we have
\begin{align*}
 &\frac{1}{T}\sum_{k\leq(\mathfrak{N} - 1)/2}\frac{(C_{1}(T))^{2k+ 1}}{\big(\sqrt{\oh\log\log T}\big)^{2k + 1}}((2k + 1)!)((\log\log\log T)^{\varepsilon_{1}} T^{(K'\log\log T)^{-1}})^{4k + 2}
\\&\ll \frac{1}{T}\frac{C_{1}(T)^{\mathfrak{N}}}{(\sqrt{\oh\log\log T})^{\mathfrak{N}}}\mathfrak{N}^{(1/2) + \mathfrak{N}}e^{-\mathfrak{N}}(\log\log\log T)^{(2\mathfrak{N}\varepsilon_{1})}T^{\frac{2\mathfrak{N}}{K'\log\log T}}
\end{align*}
for sufficiently large $T$. Moreover, using \eqref{simplified mathfrak N}, we further have
\begin{align}\label{main estimate in lemma 7.9_1 first final}
\begin{split}
 &\frac{1}{T}\frac{C_{1}(T)^{\mathfrak{N}}}{(\sqrt{\oh\log\log T})^{\mathfrak{N}}}\mathfrak{N}^{(1/2) + \mathfrak{N}}e^{-\mathfrak{N}}(\log\log\log T)^{(2\mathfrak{N}\varepsilon_{1})}T^{\frac{2\mathfrak{N}}{K'\log\log T}}
\\&=\exp\Big(\Big(\frac{2\mathfrak{N}}{K'\log\log T} - 1\Big)\log T + \mathfrak{N}\log C_{1}(T)  - \mathfrak{N}\log(\sqrt{\oh\log\log T}) + (\oh + \mathfrak{N})\log\mathfrak{N}
\\&\qquad\quad - \mathfrak{N} + (2\mathfrak{N}\varepsilon_{1})\log\log\log\log T\Big)  
\\&\ll\exp(- (1 - \varepsilon')\log T),
\end{split}
\end{align}
where  the last $\ll$ follows from \eqref{simplified mathfrak N}.

Using \eqref{precise difference even}, we see that the second term on the right of \eqref{main estimate in lemma 7.9_1} becomes
\begin{align}\label{main estimate in lemma 7.9_1 second}
\begin{split}
\ll&\sum_{k\leq\mathfrak{N}/2}\frac{1}{(2k)!}\frac{(2k)!}{k! 2^{k}}\frac{\|{\bf u}\|_{2}^{2k}\kappa_{2}^{k - 1} (k - 1)k }{T}
+\sum_{k\leq\mathfrak{N}/2}\frac{1}{(2k)!}\|{\bf u}\|_{1}^{2k}(C_{1}(T))^{2k}\frac{(2k)!k}{k!2^{k}}\frac{\Delta(T) + \delta(T)}{\log\log T}
\\&+\sum_{k\leq\mathfrak{N}/2}\frac{1}{(2k)!}\|{\bf u}\|_{1}^{2k}(C_{1}(T))^{2k}\frac{(2k)!k}{k!2^{k}}\frac{\delta(T)}{\log\log T}
+\sum_{k\leq\mathfrak{N}/2}\frac{1}{(2k)!}\|{\bf u}\|_{1}^{2k}\frac{(2k)!}{k!2^{k}} \frac{2k(k - 1)}{(\log\log T)^{2 - 3\varepsilon}}
\\&+ \sum_{k\leq\mathfrak{N}/2}\frac{1}{(2k)!}\|{\bf u}\|_{1}^{2k} \frac{(2k)!}{k!2^{k}}\frac{3^{k}(k + 1)(k + 2)}{(\log\log T)^{3 - 4\varepsilon}}
\\&+\sum_{k\leq\mathfrak{N}/2}\frac{1}{(2k)!}\frac{(2k)!k^{2}}{k!2^{k}}\frac{1}{(\log\log T)^{2}}
+\sum_{k\leq\mathfrak{N}/2}\frac{1}{(2k)!}\frac{((2k)!)^{2}k}{(k!)^{2}2^{k}}\frac{T^{2k/(K'\log\log T)}}{T(\log\log T)^{k}}
\end{split}
\end{align}

We begin with the first term on the right of \eqref{main estimate in lemma 7.9_1 second}. A simplification yields
\begin{align}\label{1st-in-7.7}
\sum_{k\leq\mathfrak{N}/2}\frac{1}{(2k)!}\frac{(2k)!}{k! 2^{k}}\frac{\|{\bf u}\|_{2}^{2k}\kappa_{2}^{k - 1} (k - 1)k }{T}
=\frac{1}{T}\sum_{k\leq\mathfrak{N}/2}\frac{\kappa_{2}^{k - 1}(k - 1)k}{k! 2^{k}}\|{\bf u}\|_{2}^{2k}\ll \frac{\|{\bf u}\|_{2}^{\mathfrak{N}}}{T}
\ll \frac{\|{\bf u}\|_{1}^{\mathfrak{N}}}{T},
\end{align}
where the convergence of the series is due to $|\kappa_{2}|<2$ by  \eqref{simple kappa5}, and the last $\ll$ follows from the fact that $\|{\bf u}\|_{2}\leq \|{\bf u}\|_{1}$. Then using \eqref{NlogF}, we deduce
\begin{align}\label{main estimate in lemma 7.9_1 second first}
\begin{split}
\frac{\|{\bf u}\|_{1}^{\mathfrak{N}}}{T}
&=\exp(\mathfrak{N}\log\|{\bf u}\|_{1} - \log T)
\\&=\exp((\log\log\log T)^{\varepsilon_{1} + \varepsilon_{2}}(\varepsilon_{1}\log\log\log\log T) - \log T)
\ll \exp(- (1 - \varepsilon')\log T)
\end{split}
\end{align}
for sufficiently large $T$.

The second term on the right of \eqref{main estimate in lemma 7.9_1 second} equals 
\begin{align*}
\frac{\Delta(T) + \delta(T)}{\log\log T}\sum_{k<\mathfrak{N}/2}\frac{k (C_{1}(T))^{2k}}{k!2^{k}}\|{\bf u}\|_{1}^{2k}.
\end{align*}
By \eqref{the very definition of C_{1}(T)} and \eqref{simple kappa5}, we see that $|C_{1}(T)|^{2}<2$, so the series converges. Therefore, we have
\begin{align*}
\frac{\Delta(T) + \delta(T)}{\log\log T}\sum_{k<\mathfrak{N}/2}\frac{k (C_{1}(T))^{2k}}{k!2^{k}}\|{\bf u}\|_{1}^{2k}
<\frac{\Delta(T) + \delta(T)}{\log\log T}\sum_{k<\mathfrak{N}/2}\frac{k 2^{k}}{k!2^{k}}\|{\bf u}\|_{1}^{2k}
\ll\frac{\Delta(T) + \delta(T)}{\log\log T}\|{\bf u}\|_{1}^{\mathfrak{N}}.
\end{align*}
Recall that both $\Delta(T) = O((\log\log T)^{\varepsilon})$ and $\delta(T) = O((\log\log T)^{\varepsilon})$. We use \eqref{NlogF} to deduce that
\begin{align*}
\frac{\Delta(T) + \delta(T)}{\log\log T}\|{\bf u}\|_{1}^{\mathfrak{N}}
\ll \exp( - (1 - \varepsilon)\log\log\log T + (\log\log\log T)^{\varepsilon_{1} + \varepsilon_{2}}(\varepsilon_{1}\log\log\log\log T) ).
\end{align*}
Recall that $\varepsilon_{1} + \varepsilon_{2}<1$ (one line before \eqref{final choice of F}). Therefore, there exists $\varepsilon'>0$ that can be arbitrarily small such that
\begin{align}\label{main estimate in lemma 7.9_1 second second}
\begin{split}
&\exp( - (1 - \varepsilon)\log\log\log T + (\log\log\log T)^{\varepsilon_{1} + \varepsilon_{2}}(\varepsilon_{1}\log\log\log\log T))
\\&\ll\exp( - (1 - \varepsilon - \varepsilon')\log\log\log T)  
\end{split}
\end{align}
for sufficiently large $T$.

The third term on the right of \eqref{main estimate in lemma 7.9_1 second} is 
\begin{align*}
\sum_{k\leq\mathfrak{N}/2}\frac{1}{(2k)!}\|{\bf u}\|_{1}^{2k}(C_{1}(T))^{2k}\frac{(2k)!k}{k!2^{k}}\frac{\delta(T)}{\log\log T}.
\end{align*}
By \eqref{the very definition of C_{1}(T)} and \eqref{simple kappa5}, we see that $|C_{1}(T)|^{2}<2$. This implies that the series converges. Therefore, similar to the second term, we have
\begin{align*}
\sum_{k\leq\mathfrak{N}/2}\frac{1}{(2k)!}\|{\bf u}\|_{1}^{2k}(C_{1}(T))^{2k}\frac{(2k)!k}{k!2^{k}}\frac{\delta(T)}{\log\log T}
<\frac{\delta(T)}{\log\log T}\sum_{k\leq\mathfrak{N}/2}\frac{k}{k!}\|{\bf u}\|_{1}^{2k}
\ll\frac{\delta(T)}{\log\log T} \|{\bf u}\|_{1}^{\mathfrak{N}}
\end{align*}
Since both $\Delta(T) = O((\log\log T)^{\varepsilon})$ and $\delta(T) = O((\log\log T)^{\varepsilon})$, we deduce from \eqref{NlogF} that
\begin{align}\label{main estimate in lemma 7.9_1 second third}
\frac{\delta(T)}{\log\log T} \|{\bf u}\|_{1}^{\mathfrak{N}}
\ll\exp( - (1 - \varepsilon - \varepsilon')\log\log\log T)  
\end{align}
for sufficiently large $T$.

The fourth term on the right of \eqref{main estimate in lemma 7.9_1 second} is 
\begin{align*}
\sum_{k\leq\mathfrak{N}/2}\frac{1}{(2k)!}\|{\bf u}\|_{1}^{2k}\frac{(2k)!}{k!2^{k}} \frac{2k(k - 1)}{(\log\log T)^{2 - 3\varepsilon}}
=\frac{1}{(\log\log T)^{2 - 3\varepsilon}}\sum_{k\leq\mathfrak{N}/2}\frac{2k(k - 1)}{k!2^{k}}\|{\bf u}\|_{1}^{2k}
\ll\frac{\|{\bf u}\|_{1}^{\mathfrak{N}}}{(\log\log T)^{2 - 3\varepsilon}}
\end{align*}
since the series converges. By \eqref{NlogF}, we see that
\begin{align}\label{main estimate in lemma 7.9_1 second fourth}
\begin{split}
\frac{\|{\bf u}\|_{1}^{\mathfrak{N}}}{(\log\log T)^{2 - 3\varepsilon}}
\ll\exp( - (2 - 3\varepsilon - \varepsilon')\log\log\log T).
\end{split}
\end{align}

The fifth term on the right of \eqref{main estimate in lemma 7.9_1 second} is  
\begin{align*}
 &\sum_{k\leq\mathfrak{N}/2}\frac{1}{(2k)!}\|{\bf u}\|_{1}^{2k} \frac{(2k)!}{k!2^{k}}\frac{3^{k}(k + 1)(k + 2)}{(\log\log T)^{3 - 4\varepsilon}}
\\&=\frac{1}{(\log\log T)^{3 - 4\varepsilon}}\sum_{k\leq\mathfrak{N}/2}\frac{(3/2)^{k} (k + 1) (k + 2)}{k!}\|{\bf u}\|_{1}^{2k}
\ll \frac{\|{\bf u}\|_{1}^{\mathfrak{N}}}{(\log\log T)^{3 - 4\varepsilon}}
\end{align*}
since the series converges.  By \eqref{NlogF}, we see that
\begin{align}\label{main estimate in lemma 7.9_1 second fifth'}
\begin{split}
 \frac{\|{\bf u}\|_{1}^{\mathfrak{N}}}{(\log\log T)^{3 - 4\varepsilon}}
\ll \exp(- (3 - 4\varepsilon - \varepsilon')\log\log\log T)
\end{split}
\end{align}

The sixth term on the right of \eqref{main estimate in lemma 7.9_1 second} is  
\begin{align}\label{main estimate in lemma 7.9_1 second fifth}
\sum_{k\leq\mathfrak{N}/2}\frac{1}{(2k)!}\frac{(2k)!k^{2}}{k!2^{k}}\frac{1}{(\log\log T)^{2}}\ll\frac{1}{(\log\log T)^{2}}
\end{align}
since the series converges.

The seventh term on the right of \eqref{main estimate in lemma 7.9_1 second} is  
\begin{align}\label{main estimate in lemma 7.9_1 second sixth_1}
\begin{split}
\sum_{k\leq\mathfrak{N}/2}\frac{1}{(2k)!}\frac{((2k)!)^{2}k}{(k!)^{2}2^{k}}\frac{T^{2k/(K'\log\log T)}}{T(\log\log T)^{k}}
&=\frac{1}{T}\sum_{k\leq\mathfrak{N}/2}\frac{(2k)!k}{(k!)^{2}2^{k}}\frac{T^{2k/K'(\log\log T)}}{(\log\log T)^{k}}
\\&<\frac{T^{\mathfrak{N}/K'(\log\log T)}}{T}\sum_{k\leq\mathfrak{N}/2}\frac{(2k)!k}{(k!)^{2}2^{k}}\frac{1}{(\log\log T)^{k}}.
\end{split}
\end{align}
For $T$ sufficiently large so that $1/\log\log T<1/3$, we have
\begin{align}\label{main estimate in lemma 7.9_1 second sixth_2}
\sum_{k\leq\mathfrak{N}/2}\frac{(2k)!k}{(k!)^{2}2^{k}}\frac{1}{(\log\log T)^{k}}
<\sum_{k\leq\mathfrak{N}/2}\frac{(2k)!k}{(k!)^{2}2^{k}}\frac{1}{3^{k}}
\ll 1.
\end{align} 
Therefore, by \eqref{main estimate in lemma 7.9_1 second sixth_2}, we infer from \eqref{main estimate in lemma 7.9_1 second sixth_1} that
\begin{align*}
\begin{split}
\sum_{k\leq\mathfrak{N}/2}\frac{1}{(2k)!}\frac{((2k)!)^{2}k}{(k!)^{2}2^{k}}\frac{T^{2k/(K'\log\log T)}}{T(\log\log T)^{k}}
=\frac{1}{T}\sum_{k\leq\mathfrak{N}/2}\frac{(2k)!k}{(k!)^{2}2^{k}}\frac{T^{2k/K'(\log\log T)}}{(\log\log T)^{k}}.
\ll\frac{T^{\mathfrak{N}/K'(\log\log T)}}{T}
\end{split}
\end{align*}
Using \eqref{simplified mathfrak N}, we have
\begin{align}\label{main estimate in lemma 7.9_1 second sixth final}
\frac{T^{\mathfrak{N}/K'(\log\log T)}}{T}
=\exp\Big(- \Big(1 - \frac{(\log\log\log T)^{\varepsilon_{1} + \varepsilon_{2}}}{K'(\log\log T)}\Big)\log T\Big)
\ll\exp(- (1 - \varepsilon') \log T).
\end{align}

In light of the argument above, we have completed the individual estimates. Now, we are in a position to assemble them. By \eqref{main estimate in lemma 7.9}, the left of \eqref{Goal in lemma 9' in AR24} is
\begin{align}\label{assemble step1}
\leq \sum_{n\leq \mathfrak{N}}\frac{1}{n!}\Big|\mathbb{E}[ ({\bf u}\cdot {\bf R}^{1}_{T})^{n}] - \mathbb{E}[ ({\bf u}\cdot \tilde{{\bf Z}}_{N})^{n}]\Big|
+\mathcal{E}_{1} + \mathcal{E}_{2} + \mathcal{E}_{3} + \mathcal{E}_{4}.
\end{align}
By further using \eqref{error1}, \eqref{error2 final}, \eqref{error3}, \eqref{error4-final}, we see that \eqref{assemble step1} is
\begin{align}\label{assemble step2}
\begin{split}
 &\ll \sum_{n\leq \mathfrak{N}}\frac{1}{n!}\Big|\mathbb{E}[ ({\bf u}\cdot {\bf R}^{1}_{T})^{n}] - \mathbb{E}[ ({\bf u}\cdot \tilde{{\bf Z}}_{N})^{n}]\Big|
+\exp(- (\log\log\log T)^{\varepsilon_{1} + \varepsilon_{2}})
\\&+\sqrt{N}\exp\Big(- \Big(\frac{1}{4\kappa_{2}\mathcal{C}_{2}^{2}} -\varepsilon'\Big)(\log\log\log T)^{2\varepsilon_{2}}\Big)
+\exp\Big(-\frac{1}{2}(\log\log\log T)^{\varepsilon_{1} + \varepsilon_{2}}\Big)
\\&+N \exp\Big( - \Big(\frac{1}{2\kappa_{2}\mathcal{C}_{2}^{2}} - \varepsilon'\Big)(\log\log\log T)^{2\varepsilon_{2}}\Big).
\end{split}
\end{align}
Recall that $\varepsilon_{1}<\varepsilon_{2}$ (one line before \eqref{final choice of F}). Therefore, the third and the fifth term in \eqref{assemble step2} decays faster than the second and the fourth. This implies that \eqref{assemble step2} is
\begin{align}\label{assemble step3}
\ll\sum_{n\leq \mathfrak{N}}\frac{1}{n!}\Big|\mathbb{E}[ ({\bf u}\cdot {\bf R}^{1}_{T})^{n}] - \mathbb{E}[ ({\bf u}\cdot \tilde{{\bf Z}}_{N})^{n}]\Big|
+ \exp\Big(-\frac{1}{2}(\log\log\log T)^{\varepsilon_{1} + \varepsilon_{2}}\Big),
\end{align}
where $\varepsilon_{2}>0$ and $\varepsilon_{1}>0$ such that $\varepsilon_{1} + \varepsilon_{2}<1$. 

By \eqref{main estimate in lemma 7.9_1}, we see that  \eqref{assemble step3} equals
\begin{align}\label{assemble step4}
\begin{split}
&\sum_{2k + 1\leq\mathfrak{N}} \frac{1}{(2k + 1)!}\Big|\mathbb{E}[ ({\bf u}\cdot {\bf R}^{1}_{T})^{2k + 1}] - \mathbb{E}[ ({\bf u}\cdot \tilde{{\bf Z}}_{N})^{2k + 1}]\Big| 
+ \sum_{2k\leq\mathfrak{N}}\frac{1}{(2k)!}\Big|\mathbb{E}[ ({\bf u}\cdot {\bf R}^{1}_{T})^{2k}] - \mathbb{E}[ ({\bf u}\cdot \tilde{{\bf Z}}_{N})^{2k}]\Big|
\\&+ \exp(-(1/2)(\log\log\log T)^{\varepsilon_{1} + \varepsilon_{2}}).
\end{split}
\end{align}
Using \eqref{main estimate in lemma 7.9_1 first final} in \eqref{assemble step4}, we see that \eqref{assemble step4} is
\begin{align}\label{assemble step5}
\begin{split}
&\ll \exp(- (1 - \varepsilon')\log T)
+ \sum_{2k\leq\mathfrak{N}}\frac{1}{(2k)!}\Big|\mathbb{E}[ ({\bf u}\cdot {\bf R}^{1}_{T})^{2k}] - \mathbb{E}[ ({\bf u}\cdot \tilde{{\bf Z}}_{N})^{2k}]\Big|
+ \exp\Big(-\frac{1}{2}(\log\log\log T)^{\varepsilon_{1} + \varepsilon_{2}}\Big)
\\&\ll \sum_{2k\leq\mathfrak{N}}\frac{1}{(2k)!}\Big|\mathbb{E}[ ({\bf u}\cdot {\bf R}^{1}_{T})^{2k}] - \mathbb{E}[ ({\bf u}\cdot \tilde{{\bf Z}}_{N})^{2k}]\Big|
+ \exp\Big(-\frac{1}{2}(\log\log\log T)^{\varepsilon_{1} + \varepsilon_{2}}\Big).
\end{split}
\end{align}
Finally, plugging \eqref{main estimate in lemma 7.9_1 second first}, \eqref{main estimate in lemma 7.9_1 second second}, \eqref{main estimate in lemma 7.9_1 second third}, \eqref{main estimate in lemma 7.9_1 second fourth}, \eqref{main estimate in lemma 7.9_1 second fifth'}, \eqref{main estimate in lemma 7.9_1 second fifth}, and \eqref{main estimate in lemma 7.9_1 second sixth final} into \eqref{assemble step5}, we have that \eqref{assemble step5} is
\begin{align*}
&\ll \exp(- (1 - \varepsilon')\log T)
+\exp( - (1 - \varepsilon - \varepsilon')\log\log\log T)  
+\exp( - (1 - \varepsilon - \varepsilon')\log\log\log T)  
\\
&+\exp(- (2 - 3\varepsilon - \varepsilon')\log\log\log T)  
+ \exp(- (3 - 4\varepsilon - \varepsilon')\log\log\log T)
+(\log\log T)^{-2}
\\&+\exp(- (1 - \varepsilon') \log T)
+ \exp(-(1/2)(\log\log\log T)^{\varepsilon_{1} + \varepsilon_{2}})
\\&\ll\exp(-(1/2)(\log\log\log T)^{\varepsilon_{1} + \varepsilon_{2}}).
\end{align*}
Hence, we have complete the proof.
\end{proof}


Similar to \cite{AR24}, we recall the following lemma from \cite[Lemma 2.6]{ABB17}.
\begin{lemma}\label{lemma 8 in AR24}
Let $N\geq 1$. If $\mu$ and $\nu$ are probability measures on $\mathbb{R}^{N}$ with Fourier transforms $\hat{\mu}$ and $\hat{\nu}$, then for any $R, F>0$ and any function $f:\mathbb{R}^{N}\rightarrow\mathbb{R}$ with Lipschitz constant $\|f\|_{Lip}$, one has
\begin{align}\label{Goal in Prop 6 in AR24}
\begin{split}
&\bigg|\int_{\mathbb{R}^{N}}fd\mu - \int_{\mathbb{R}^{N}}f d\nu\bigg|\\
&\ll_{N} \frac{\|f\|_{Lip}}{F}  + \|f\|_{\infty}\left((RF)^{N}\|(\hat{\mu} - \hat{\nu})\mathds{1}_{(-F, F)^{N}}\|_{\infty}
+\mu\big(\big([-R, R]^{N}\big)^{c}\big) + \nu\big(\big([-R, R]^{N}\big)^{c}\big)\right),
\end{split}
\end{align}
where $([-R, R]^{N})^{c}$ denote the complement of $[-R, R]^{N}$.
\end{lemma}

Having Lemmata \ref{lemma 9' in AR24} and \ref{lemma 8 in AR24} in hands, we are in a position to proof Proposition \ref{Prop 6 in AR24}.

\begin{proof}[Proof of Proposition \ref{Prop 6 in AR24}]

Let $\mu$ be the measure induced by ${\bf R}^{1}_{T}$ and $\nu$ the measure induced by $\tilde{\bf Z}_{N}$. For any $R>0$,  by \eqref{mu and nu}, we have
\begin{align}\label{1st-middle-work-in-Prop6}
\begin{split}
\mu(([-R, R]^{N})^{c})
&\leq\sum_{j = 1}^{N}\mathfrak{P}_{T}\Big(\Big|P_{1}(s_{0, j}, \chi_{j})/\sqrt{\oh\mathfrak{M}_{T, \chi_{j}}}\Big| > R\Big)
\leq N \mathfrak{P}(|\mathcal{N}| > R) + \frac{1}{T},
\end{split}
\end{align}
Thus, by \cite[Eg. 3.6.3]{CB}, we further have
\begin{align}\label{2nd-middle-work-in-Prop6}
\begin{split}
&\mu(([-R, R]^{N})^{c}) + \nu(([-R, R]^{N})^{c})
\leq 2N \mathfrak{P}\left(|\mathcal{N}| > R\right) + \frac{1}{T}
\leq 2N\sqrt{\frac{2}{\pi}}\frac{e^{-R^{2}/2}}{R} + \frac{1}{T}
\end{split}
\end{align}
Therefore, by taking $R = (\log\log\log T)^{\varepsilon_{2}}$, we have that
\begin{align}\label{R part}
\begin{split}
\mu(([-R, R]^{N})^{c}) + \nu(([-R, R]^{N})^{c})
\ll N \frac{\exp((-1/2)(\log\log\log T)^{2\varepsilon_{2}})}{(\log\log\log T)^{\varepsilon_{2}}}.
\end{split}
\end{align}
Moreover, by taking $F =(\log\log\log T)^{\varepsilon_{1}}$ and plugging \eqref{Goal in lemma 9' in AR24} and \eqref{R part} into the right of \eqref{Goal in Prop 6 in AR24}, we obtain
\begin{align*}
\begin{split}
&\ll_{N} \frac{L}{(\log\log\log T)^{\varepsilon_{1}}} 
+ M(\log\log\log T)^{N(\varepsilon_{1} + \varepsilon_{2})}\exp((-1/2)(\log\log\log T)^{\varepsilon_{1} + \varepsilon_{2}})
\\&+ NM  \frac{\exp((-1/2)(\log\log\log T)^{2\varepsilon_{2}})}{(\log\log\log T)^{\varepsilon_{2}}}
\end{split}
\end{align*}
for sufficiently large $T$, which leads the desired estimate.
\end{proof}

\section{Proof of Proposition \ref{Prop 7}}\label{pf-prop-7}

In this section, we will prove the following general result, and we shall note that Proposition \ref{Prop 7} follows from the lemma with $\mathfrak{C} = \mathfrak{K}$ and $\widetilde{ \mathfrak{C} }=\widetilde{ \mathfrak{K} }$.

\begin{lemma}
Let $\mathfrak{C} = \mathfrak{C}(T)$ and $\widetilde{ \mathfrak{C} }=\widetilde{ \mathfrak{C} }(T)$ be $N\times N$ symmetric positive-definite matrices such 
that
$$
\widetilde{ \mathfrak{C} }=  \mathfrak{C} +  \widetilde{ E },
$$
where each entry of $ \widetilde{ E }$ is $O( (\log\log T)^{-1+\varepsilon} )$ for some (fixed) $\varepsilon<1)$. Then for any $\varepsilon_{3}>0$, we have
\begin{align*}
 \begin{split}
\int_{\mathcal{B}} \frac{f(x)}{2\pi }
 \bigg(  \frac{e^{- (x^T  \widetilde{ \mathfrak{C}}^{-1} x)/2 } }{\sqrt{\det \widetilde{ \mathfrak{C} }}}
-  \frac{e^{- (x^T   \mathfrak{C}^{-1} x)/2 } }{\sqrt{\det\mathfrak{C} }} \bigg) dx
\ll_\mathfrak{C} ||f||_\infty (\log\log T)^{-1+\varepsilon +\varepsilon_{3}},
 \end{split}
\end{align*}
where $\mathcal{B}$ is a Borel set in $\mathbb{R}^{N}$.
\end{lemma}

\begin{proof}
We shall start with showing that 
\begin{equation}\label{tildeC-expansion}
\widetilde{ \mathfrak{C}}^{-1} 
= \mathfrak{C}^{-1}  + E'
\end{equation}
for some $E'$ whose entries are all $\ll  (\log\log T)^{ -1+\varepsilon +\varepsilon_{3} }$. Observe that 
$$
\widetilde{ \mathfrak{C}}^{-1} 
= ( \mathfrak{C} +  \widetilde{ E } )^{-1}
=  ( I +   \mathfrak{C}^{-1}\widetilde{ E } )^{-1} \mathfrak{C}^{-1},
$$
and hence
\begin{equation}\label{inverse-ICE-expansion}
( I +   \mathfrak{C}^{-1}\widetilde{ E } )^{-1}  = I  +\sum_{k=1}^\infty (-1)^k (\mathfrak{C}^{-1}\widetilde{ E } )^k.
\end{equation}
Denoting $\gamma$  the maximum of absolute values of entries of $\mathfrak{C}^{-1}$, we know that every entry of $\mathfrak{C}^{-1}\widetilde{ E }$  is $O( N\gamma (\log\log T)^{-1+\varepsilon} ) $, and thus all the entries of $(\mathfrak{C}^{-1}\widetilde{ E })^k$  are $O( N^{2k-1}\gamma^k (\log\log T)^{(-1+\varepsilon)k} ) $.
Hence, for any  $\varepsilon_{3}>0$  such that $-1+\varepsilon +\varepsilon_{3}< 0 $, there is a  sufficiently large $T_1$  so that for $T\ge T_1$, we have  $N^2 \gamma 
\ll  (\log\log T)^{\varepsilon_{3} }$ and so
$$
N^2 \gamma (\log\log T)^{-1+\varepsilon }
\ll  (\log\log T)^{\varepsilon_{3} } (\log\log T)^{-1+\varepsilon }<1;
$$
for such an instance, every entry of the sum in \eqref{inverse-ICE-expansion} is at most
$$
\frac{O((\log\log T)^{ -1+\varepsilon +\varepsilon_{3} })}{1 + O((\log\log T)^{-1+\varepsilon +\varepsilon_{3} }) }
\ll   (\log\log T)^{ -1+\varepsilon +\varepsilon_{3} },
$$
which yields the claim \eqref{tildeC-expansion}.

Next, we note that
$
\widetilde{ \mathfrak{C} }=  \mathfrak{C}+  \widetilde{ E }   =   \mathfrak{C}( I + \mathfrak{C}^{-1}\widetilde{ E } )
$
and thus
\begin{equation}\label{det-tildeC}
\det \widetilde{ \mathfrak{C} } =  \det  \mathfrak{C} \det  ( I + \mathfrak{C}^{-1}\widetilde{ E } ).
\end{equation}
From the Leibniz formula for  determinants, it is clear that
$$
\det  ( I + \mathfrak{C}^{-1}\widetilde{ E } ) = 1 +  O_{\mathfrak{C},N}( (\log\log T)^{-1+\varepsilon}).
$$
Clearly, for  $\varepsilon_{3}>0$, there is  $T_2>0$, depending on $\mathfrak{C},N$, so that the implied constant is $\ll (\log\log T)^{\varepsilon_{3}} $, for $T\ge T_2$, which yields
\begin{equation}\label{det-ICE}
\det  ( I + \mathfrak{C}^{-1}\widetilde{ E } ) = 1 +  O( (\log\log T)^{-1+\varepsilon +\varepsilon_{3}}).
\end{equation}
Now, by \eqref{det-tildeC}, \eqref{det-ICE}, and the fact that $\sqrt{1+x} =1 +O(x)$ uniformly in $|x|\le\frac{1}{2}$, we deduce
$$
\frac{1}{\sqrt{\det \widetilde{ \mathfrak{C} }}} -  \frac{1}{\sqrt{\det\mathfrak{C}}} 
=  \frac{1}{\sqrt{\det\mathfrak{C}}}  \bigg( \frac{1}{  1 +  O( (\log\log T)^{-1+\varepsilon +\varepsilon_{3}}) }  -1 \bigg)
\ll_\mathfrak{C} (\log\log T)^{-1+\varepsilon +\varepsilon_{3}},
$$
which allows us to write
\begin{align}\label{final-int}
 \begin{split}
&\int_{\mathcal{B}} \frac{f(x)}{2\pi }
 \bigg(  \frac{e^{- (x^T  \widetilde{ \mathfrak{C}}^{-1} x)/2 } }{\sqrt{\det \widetilde{ \mathfrak{C} }}}
-  \frac{e^{- (x^T   \mathfrak{C}^{-1} x)/2 } }{\sqrt{\det\mathfrak{C} }} \bigg) dx\\
&=  \int_{\mathcal{B}} \frac{f(x)}{2\pi \sqrt{\det\mathfrak{C} }}
 \Big(  e^{- (x^T  \widetilde{ \mathfrak{C}}^{-1} x)/2 }  - e^{- (x^T   \mathfrak{C}^{-1} x)/2 } \Big) dx
+ O_{\mathfrak{C}} ( ||f||_\infty (\log\log T)^{-1+\varepsilon +\varepsilon_{3}}).
 \end{split}
\end{align}

%

To handle the last integral, we apply the mean value theorem (with $e^{-t}$) to write
$$
e^{- (x^T  \widetilde{ \mathfrak{C}}^{-1} x)/2 }  - e^{- (x^T   \mathfrak{C}^{-1} x)/2 }
= - e^{- \xi  } (( x^T  \widetilde{ \mathfrak{C}}^{-1} x)/2    -  (x^T   \mathfrak{C}^{-1} x)/2)
$$
for some $\xi$ between $( x^T  \widetilde{ \mathfrak{C}}^{-1} x)/2$ and $ (x^T   \mathfrak{C}^{-1} x)/2$. By \eqref{tildeC-expansion}, 
we have
\begin{align*}
 x^T  \widetilde{ \mathfrak{C}}^{-1} x
 =
 x^T ( \mathfrak{C}^{-1} +E'
) x
=   x^T    \mathfrak{C}^{-1}  x  +  x^T E'  x
\ge \lambda_{\min}( \mathfrak{C}^{-1}) ||x||_2^2
+ \lambda_{\min}(E' ) ||x||_2^2,
\end{align*}
where $\lambda_{\min}(M)$ denotes the smallest eigenvalues of a matrix $M$, and the last inequality follows from the lower bound for the Rayleigh quotients.
(Note that the symmetry of  $\mathfrak{C}^{-1}$ and $\widetilde{ \mathfrak{C} }^{-1}$ implies that  $E'$ must be symmetric.) As each entry of $E'$ is $O_{\mathfrak{C}}((\log\log T)^{-1+\varepsilon +\varepsilon_{3}})$, we have 
$$
|\lambda_{\min}(E' )|
\le \rho(E' )
\le ||E' ||
 \ll (\log\log T)^{-1+\varepsilon +\varepsilon_{3}},
$$
where $\rho(M)$ is the spectral radius of $M$, and $||M||$ denotes the (usual) matrix norm of $M$. (Here we used the fact that the matrix norm and the max norm are equivalent for $N\times N$ matrices.) Consequently,  $ x^T  \widetilde{ \mathfrak{C}}^{-1} x \ge  \frac{1}{2} \lambda_{\min}( \mathfrak{C}^{-1}) ||x||_2^2$. Hence, as $\xi$ lies between $( x^T  \widetilde{ \mathfrak{C}}^{-1} x)/2$ and  $ (x^T   \mathfrak{C}^{-1} x)/2$, we conclude that $\xi\ge   \frac{1}{4} \lambda_{\min}( \mathfrak{C}^{-1}) ||x||_2^2 =  \frac{1}{4} \lambda_{\min}( \mathfrak{C}^{-1}) x^Tx $.
(Note that as $\mathfrak{C}^{-1}$ is positive-definite, $\lambda_{\min}( \mathfrak{C}^{-1})>0$.)

Finally, by \eqref{tildeC-expansion} ($
\widetilde{ \mathfrak{C}}^{-1} 
= \mathfrak{C}^{-1}E'
$) and  the same reasoning as above,
we know $$
|( x^T  \widetilde{ \mathfrak{C}}^{-1} x)/2    -  (x^T   \mathfrak{C}^{-1} x)/2|
=|(x^T   E' x)/2|
\le \rho(E' )||x||_2^2
\ll (\log\log T)^{-1+\varepsilon +\varepsilon_{3}}||x||_2^2.
$$ 
Therefore, we conclude that the last integral in \eqref{final-int} is
$$
 \ll_\mathfrak{C}   ||f||_\infty
\int_{\mathcal{B}}  e^{- \frac{1}{4} \lambda_{\min}( \mathfrak{C}^{-1}) x^Tx}  (\log\log T)^{-1+\varepsilon +\varepsilon_{3}}||x||_2^2 dx
\ll_\mathfrak{C} ||f||_\infty (\log\log T)^{-1+\varepsilon +\varepsilon_{3}}
$$ 
as desired.
\end{proof}


%
%
%
%
%
%
%
%

\begin{remark}
Unfortunately, there were some missing pieces in the proof of \cite[Proposition 7]{AR24}. Firstly, the calculation for $\widetilde{C}$ was not given, and by the discussion on \cite[p. 3362]{AR24} (and the one beneath \cite[Eq. (24)]{AR24}), it seems that the diagonal terms of $\widetilde{C}-C$  would be $O(\log\log T)$ instead of 0. 

Secondly, in the last displayed equation in \cite[p. 3368]{AR24}, the Taylor approximation $e^t = 1 +O(t)$ was used. However, as $x^T (  C^{-1}\widetilde{ E }    C^{-1}) x$ could not be bounded for all $x\in \Bbb{R}^2$, the required Taylor approximation would lack of the uniformity (i.e., the implied is not absolute).

Thirdly, in the last paragraph in \cite[p. 3368]{AR24}, there was a claimed cancellation that requires the a bound for $\frac{1}{\sqrt{\det \widetilde {C }}} -  \frac{1}{\sqrt{\det C}} $, which was not precisely calculated. Such a bound is crucial as it would affect the final estimate (cf. the last big-O term in \eqref{final-int}).
\end{remark}

\section{Proof of Theorem \ref{indep-rate}}\label{Tsang}

In this section, we will take $\alpha_{i} = \alpha_{j} = \alpha$ to be a fixed constant for all $i, j$ so that $\Delta(T) = 0$ for all $T$. We also set $\delta(T) = 0$ for all $T$. For Dirichlet characters, we take $\chi_{i}\neq\chi_{j}$ if $i\neq j$. As a result, the matrix $\mathfrak{K}$ introduced in \eqref{the covariance matrix} becomes the identity matrix, and  the random vector ${\bf X}_{T}$ becomes 
\begin{align}\label{X_{T}-Tsang}
{\bf X}_{T} = {\bf X}_{T}(U) = (\mathcal{X}_{\alpha, \chi_{1}, T}(U), \ldots, \mathcal{X}_{\alpha, \chi_{N}, T}(U)) ,
\end{align}
where $U = U_{T}$ be a uniform distribution on $[T, 2T]$,  and
$$
\mathcal{X}_{\alpha, \chi_j, T}(U) = (\log|L(\oh +\mi(U+\alpha),\chi_j)|) / \sqrt{\oh\log\log T}.
$$

In the present setting, it is clear that the estimates in Propositions \ref{Prop 1 in AR24} - \ref{Prop 5 in AR24} and Proposition \ref{Prop 7} remain unaffected; the only difference lies in Proposition \ref{Prop 6 in AR24}, where the (possible) non-vanishing of $\Delta(T)$ and $\delta(T)$ leads to 
\begin{align}\label{badbad}
\|{\bf u}\|_{1}^{2k}(C_{1}(T))^{2k}\frac{(2k)!k}{k!2^{k}}\frac{\Delta(T) + \delta(T)}{\log\log T}
\quad\mbox{and}\quad
\|{\bf u}\|_{1}^{2k}(C_{1}(T))^{2k}\frac{(2k)!k}{k!2^{k}}\frac{\delta(T)}{\log\log T}
\end{align}
on the right of \eqref{precise difference even}. Our first goal is to show that in the absence of \eqref{badbad}, we can improve the estimate in Lemma \ref{lemma 9' in AR24} as follows.

\begin{lemma}\label{better-lemma 9}
Suppose that $\Delta(T) = 0$ and $\delta(T) = 0$ for all $T$. Let ${\bf u} = (u_{1}, \ldots, u_{N})\in\mathbb{R}^{N}$, and let ${\bf R}^{1}_{T}$ and $\tilde{\bf Z}_{N}$ be defined as in Lemma \ref{lemma 9' in AR24}. For sufficiently large $T$, We have
$$
\Big|\mathbb{E}[\exp(\mi ({\bf u}\cdot {\bf R}^{1}_{T}))] - \mathbb{E}[\exp(\mi ({\bf u}\cdot\tilde{\bf Z}_{N}))]\Big|
\ll (\log\log T)^{-2}.
$$
\end{lemma}

\begin{proof}
Processing the same decomposition as in \eqref{main estimate in lemma 7.9}, we have
\begin{align}\label{main estimate in lemma 7.9'}
\begin{split}
&\Big|\mathbb{E}[\exp(\mi ({\bf u}\cdot {\bf R}^{1}_{T}))] - \mathbb{E}[\exp(\mi ({\bf u}\cdot\tilde{\bf Z}_{N}))]\Big|
\leq\sum_{n\leq \mathfrak{N}}\frac{1}{n!}\Big|\mathbb{E}[ ({\bf u}\cdot {\bf R}^{1}_{T})^{n}] - \mathbb{E}[ ({\bf u}\cdot \tilde{{\bf Z}}_{N})^{n}]\Big|
+\sum_{j = 1}^{4}\mathcal{E}_{j}.
\end{split}
\end{align}
Now, we take
\begin{align}\label{Tsang's r}
\mathfrak{r} = \log\log T
\end{align}
and
\begin{align}\label{the growth of u}
\|{\bf u}\|_{1} = \frac{\sqrt{\oh\log\log T}}{(C_{1}(T)\mathcal{C}_{2})(\log\log\log T)},
\end{align}
where $C_{1}(T)$ and $\mathcal{C}_{2}$ are the same as in \eqref{the very definition of C_{1}(T)} and \eqref{final choice of C_{2}}, respectively.
Let $\mathfrak{N}$ be defined as in \eqref{final choice of mathfrak N}. Then with the choice of $\mathfrak{r}$ as in \eqref{Tsang's r} and $\|{\bf u}\|_{1}$ as in \eqref{the growth of u}, we have that
\begin{align}\label{the choice of M}
\mathfrak{N} = \frac{\log\log T}{\log\log\log T}
\end{align}
and that
\begin{align}\label{NlogF'}
\begin{split}
\mathfrak{N}\log\|{\bf u}\|_{1} 
=  \frac{1}{2}\log\log T - (\log\sqrt{2} + \log C_{1}(T)\mathcal{C}_{2})\frac{\log\log T}{\log\log\log T} - \frac{(\log\log T)(\log\log\log\log T)}{\log\log\log T}.
\end{split}
\end{align}

For $\mathcal{E}_{1}$ in the present setting, analogously, we have
\begin{align}\label{error1'}
\begin{split}
\mathcal{E}_{1}
&= 
\Big|\sum_{n>\mathfrak{N}}\frac{1}{n!}\mi^{n}\mathbb{E}[ ({\bf u}\cdot {\bf R}^{1}_{T})^{n}\cdot\mathds{1}(\cap_{j = 1}^{N}\{|P_{1}(s_{0, j})|\leq \mathfrak{r}\})]\Big|
\\&\leq \sum_{n>\mathfrak{N}}\frac{1}{n!}\bigg(\frac{\mathfrak{r} C_{1}(T)}{\sqrt{\oh\log\log T}} \|{\bf u}\|_{1}\bigg)^{n}
<e^{- \mathfrak{N}}
=  \exp\Big(-  \frac{(\log\log T)}{(\log\log\log T)}\Big)
\end{split}
\end{align}
by \eqref{the choice of M}.

Next, we consider $\mathcal{E}_{2}$ in the present setting. Processing analogously, we have the counterpart of \eqref{error2_1}:
\begin{align}\label{error2_1'}
\begin{split}
\mathcal{E}_{2}
\ll\sqrt{N} \|{\bf u}\|_{1}^{\mathfrak{N}}\cdot\Big(\sqrt{\frac{1}{\pi}}\frac{\sqrt{\log\log T}}{\mathfrak{r}}\exp\Big(-\frac{\mathfrak{r}^{2}}{\log\log T}\Big)\Big)^{1/2} +\sqrt{N}\frac{\|{\bf u}\|_{1}^{\mathfrak{N}}}{\sqrt{T}}.
\end{split}
\end{align}
Plugging \eqref{Tsang's r} and \eqref{NlogF'} into \eqref{error2_1'}, we see that the first term in \eqref{error2_1'} is
\begin{align*}
\ll\sqrt{N}\exp\Big(- (1 - \varepsilon')\frac{(\log\log T)(\log\log\log\log T)}{\log\log\log T}\Big)
\end{align*}
for sufficiently large $T$. Analogously, the second term in \eqref{error2_1'} is
\begin{align*}
\ll\sqrt{N}\exp( - (1/2 - \varepsilon')\log T )
\end{align*}
for sufficiently large $T$. Therefore, we conclude that
\begin{align}\label{error2_1' final}
\begin{split}
\mathcal{E}_{2}
\ll\sqrt{N}\exp\Big(- (1 - \varepsilon')\frac{(\log\log T)(\log\log\log\log T)}{\log\log\log T}\Big)
\end{split}
\end{align}
for sufficiently large $T$.

For $\mathcal{E}_{3}$ in the present setting, by the same reasoning as in \eqref{error3}, we have
\begin{align}\label{error3'}
\begin{split}
&\mathcal{E}_{3}
\leq e^{-\mathfrak{N}/2}
= \exp\Big(-\frac{1}{2} \frac{\log\log T}{\log\log\log T}\Big)
\end{split}
\end{align}
for sufficiently large $T$.

For $\mathcal{E}_{4}$ in the present setting, by the property that $|e^{\mi\theta}|\leq 1$ for all $\theta$, we have
\begin{align}\label{error4'}
\begin{split}
\mathcal{E}_{4}& = \mathbb{E}[\exp(\mi ({\bf u}\cdot {\bf R}^{1}_{T}))\cdot  \mathds{1}(\cup_{j = 1}^{N}\{|P_{1}(s_{0, j})|> \mathfrak{r}\})]
\\&\leq\sum_{j = 1}^{N}\mathfrak{P}_{T}\big(|P_{1}(s_{0, j})|>\mathfrak{r}\big)
\leq N\sqrt{\frac{1}{\pi}}\frac{\sqrt{\log\log T}}{\mathfrak{r}}\exp\Big(-\frac{\mathfrak{r}^{2}}{\log\log T}\Big) + \frac{N}{T},
\end{split}
\end{align}
where the last inequality follows from \eqref{exponential estimate of P_{1} general}. Plugging \eqref{Tsang's r} into \eqref{error4'}, we have
\begin{align}\label{error4' final}
\mathcal{E}_{4}
\ll N \frac{1}{\sqrt{\log\log T}}e^{-\log\log T} + \frac{N}{T}
\ll\frac{N}{\log T}
\end{align}
for sufficiently large $T$.

So far, we have dealt with the errors in \eqref{main estimate in lemma 7.9'}. Now, we turn our attention to the main term in \eqref{main estimate in lemma 7.9'}. Similar to Lemma \ref{lemma 9' in AR24}, we also use Lemma \ref{precise difference}. The main difference is that in \eqref{precise difference even}, the second, third, fourth, and fifth terms are \emph{zero} in the present setting ($\Delta (T) = 0 = \delta(T)$). In other words, in the present setting, \eqref{precise difference even} becomes
\begin{align}\label{precise difference even easy}
\begin{split}
&\Big|\mathbb{E}[ ({\bf u}\cdot {\bf R}^{1}_{T})^{2k}] - \mathbb{E}[ ({\bf u}\cdot \tilde{{\bf Z}}_{N})^{2k}] \Big|
\\&=
\frac{(2k)!}{k! 2^{k}}\frac{\|{\bf u}\|_{2}^{2k}\kappa_{2}^{k - 1} (k - 1)k }{T}
+ O\Big(\frac{(2k)!k^{2}}{k!2^{k}}\frac{1}{(\log\log T)^{2}}\Big)
 + O\Big(\frac{((2k)!)^{2}k}{(k!)^{2}2^{k}}\frac{T^{2k/(K'\log\log T)}}{T(\log\log T)^{k}}\Big),
\end{split}
\end{align}
where the implied constants are  absolute, and $\kappa_{2}$ is introduced in \eqref{simple kappa5}.

Similar to \eqref{main estimate in lemma 7.9_1},  we decompose the main term in \eqref{main estimate in lemma 7.9'} into odd and even terms:
\begin{align}\label{main estimate in lemma 7.9_1'}
\begin{split}
\sum_{n\leq \mathfrak{N}}\frac{1}{n!}\Big|\mathbb{E}[ ({\bf u}\cdot {\bf R}^{1}_{T})^{n}] - \mathbb{E}[ ({\bf u}\cdot \tilde{{\bf Z}}_{N})^{n}]\Big|
&=\sum_{2k + 1\leq\mathfrak{N}} \frac{1}{(2k + 1)!}\Big|\mathbb{E}[ ({\bf u}\cdot {\bf R}^{1}_{T})^{2k + 1}] - \mathbb{E}[ ({\bf u}\cdot \tilde{{\bf Z}}_{N})^{2k + 1}]\Big| 
\\&+ \sum_{2k\leq\mathfrak{N}}\frac{1}{(2k)!}\Big|\mathbb{E}[ ({\bf u}\cdot {\bf R}^{1}_{T})^{2k}] - \mathbb{E}[ ({\bf u}\cdot \tilde{{\bf Z}}_{N})^{2k}]\Big|.
\end{split}
\end{align}
By \eqref{precise difference odd}, the first term on the right of \eqref{main estimate in lemma 7.9_1'} is
\begin{align*}
&\ll \frac{1}{T}\sum_{k\leq(\mathfrak{N} - 1)/2}\frac{(C_{1}(T))^{2k+ 1}}{\big(\sqrt{\oh\log\log T}\big)^{2k + 1}}((2k + 1)!)(\|{\bf u}\|_1 T^{(K'\log\log T)^{-1}})^{4k + 2}
\\&\ll \frac{1}{T}\sum_{k\leq(\mathfrak{N} - 1)/2}\frac{(C_{1}(T))^{2k+ 1}}{\big(\sqrt{\oh\log\log T}\big)^{2k + 1}}((2k + 1)!)\bigg(\frac{\sqrt{\oh\log\log T}}{(C_{1}(T)\mathcal{C}_{2})(\log\log\log T)} T^{(K'\log\log T)^{-1}}\bigg)^{4k + 2},
\end{align*}
where the second $\ll$ follows from \eqref{the growth of u}. Since the series diverges, the terms with the largest exponent dominate. Hence, by using the Stirling formula \eqref{Stirling formula}, we have
\begin{align*}
 &\frac{1}{T}\sum_{k\leq(\mathfrak{N} - 1)/2}\frac{(C_{1}(T))^{2k+ 1}}{\big(\sqrt{\oh\log\log T}\big)^{2k + 1}}((2k + 1)!)\bigg(\frac{\sqrt{\oh\log\log T}}{(C_{1}(T)\mathcal{C}_{2})(\log\log\log T)} T^{(K'\log\log T)^{-1}}\bigg)^{4k + 2}
\\&\ll T^{\frac{2\mathfrak{N}}{K'\log\log T} - 1}\frac{(\sqrt{\oh\log\log T})^{\mathfrak{N}}}{(C_{1}(T))^{\mathfrak{N}}(\mathcal{C}_{2})^{2\mathfrak{N}}(\log\log\log T)^{2\mathfrak{N}}}\mathfrak{N}^{(1/2) + \mathfrak{N}}e^{-\mathfrak{N}}
\end{align*}
for sufficiently large $T$. Moreover, using \eqref{the choice of M}, we further have
\begin{align*}
\begin{split}
 &T^{\frac{2\mathfrak{N}}{K'\log\log T} - 1}\frac{(\sqrt{\oh\log\log T})^{\mathfrak{N}}}{(C_{1}(T))^{\mathfrak{N}}(\mathcal{C}_{2})^{2\mathfrak{N}}(\log\log\log T)^{2\mathfrak{N}}}\mathfrak{N}^{(1/2) + \mathfrak{N}}e^{-\mathfrak{N}}
\\&=\exp\bigg(\bigg(\frac{2}{K'\log\log\log T} - 1\bigg)\log T -(\log\sqrt{2} + 1 + \log C_{1}(T) + 2\log\mathcal{C}_{2})\frac{\log\log T}{\log\log\log T} + \frac{3}{2}\log\log T
\\&\qquad\quad + \frac{1}{2}\log\log\log T- 3\frac{(\log\log T)(\log\log\log\log T)}{\log\log\log T}
 - \frac{1}{2}\log\log\log\log T 
\bigg)
\\&\ll\exp(-(1 - \varepsilon')\log T)
\end{split}
\end{align*}
for sufficiently large $T$. Thus, we conclude that the first term on the right of \eqref{main estimate in lemma 7.9_1'} is
\begin{align}\label{main estimate in lemma 7.9_1' first final}
\ll \exp(-(1 - \varepsilon')\log T)
\end{align}
for sufficiently large $T$.

Now, we estimate the second term on the right of  \eqref{main estimate in lemma 7.9_1'}. Using \eqref{precise difference even easy}, we see that second term on the right of  \eqref{main estimate in lemma 7.9_1'} is
\begin{align}\label{main estimate in lemma 7.9_1' second}
\begin{split}
\ll&\sum_{k\leq\mathfrak{N}/2}\frac{1}{(2k)!}\frac{(2k)!}{k! 2^{k}}\frac{\|{\bf u}\|_{2}^{2k}\kappa_{2}^{k - 1} (k - 1)k }{T}
\\&+\sum_{k\leq\mathfrak{N}/2}\frac{1}{(2k)!}\frac{(2k)!k^{2}}{k!2^{k}}\frac{1}{(\log\log T)^{2}}
+\sum_{k\leq\mathfrak{N}/2}\frac{1}{(2k)!}\frac{((2k)!)^{2}k}{(k!)^{2}2^{k}}\frac{T^{2k/(K'\log\log T)}}{T(\log\log T)^{k}}.
\end{split}
\end{align}
The first term in \eqref{main estimate in lemma 7.9_1' second} follows from \eqref{1st-in-7.7} and \eqref{NlogF'}:
\begin{align}\label{main estimate in lemma 7.9_1' second first}
\sum_{2k\leq\mathfrak{N}}\frac{1}{(2k)!}\frac{(2k)!}{k! 2^{k}}\frac{\|{\bf u}\|_{2}^{2k}\kappa_{2}^{k - 1} (k - 1)k }{T}
\ll \exp(-(1 - \varepsilon')\log T)
\end{align}
for sufficiently large $T$. The second term in \eqref{main estimate in lemma 7.9_1' second} satisfies the same estimate as in \eqref{main estimate in lemma 7.9_1 second fifth} since the series converges. For the third term, we plug  \eqref{the choice of M} into \eqref{main estimate in lemma 7.9_1 second sixth final} to obtain
\begin{align}\label{main estimate in lemma 7.9_1' second second}
\begin{split}
&\sum_{k\leq\mathfrak{N}/2}\frac{1}{(2k)!}\frac{((2k)!)^{2}k}{(k!)^{2}2^{k}}\frac{T^{2k/(K'\log\log T)}}{T(\log\log T)^{k}}
\ll \frac{T^{\mathfrak{N}/K'(\log\log T)}}{T}
\ll\exp(- (1/2 - \varepsilon') \log T)
\end{split}
\end{align}
for sufficiently large $T$.

Finally, we are in a position to assemble the estimates. By \eqref{error1'}, \eqref{error2_1' final}, \eqref{error3'}, and \eqref{error4' final}, we have
\begin{align*}
\begin{split}
&\Big|\mathbb{E}[\exp(\mi ({\bf u}\cdot {\bf R}^{1}_{T}))] - \mathbb{E}[\exp(\mi ({\bf u}\cdot\tilde{\bf Z}_{N}))]\Big|
\\&\leq\sum_{n\leq \mathfrak{N}}\frac{1}{n!}\Big|\mathbb{E}[ ({\bf u}\cdot {\bf R}^{1}_{T})^{n}] - \mathbb{E}[ ({\bf u}\cdot \tilde{{\bf Z}}_{N})^{n}]\Big|
+ \exp\Big(-  \frac{(\log\log T)}{(\log\log\log T)}\Big)
\\&+\sqrt{N}\exp\Big(- (1 - \varepsilon')\frac{(\log\log T)(\log\log\log\log T)}{\log\log\log T}\Big)
+\exp\Big(-\frac{1}{2} \frac{\log\log T}{\log\log\log T}\Big)
+\frac{N}{\log T}
\\&\ll\sum_{n\leq \mathfrak{N}}\frac{1}{n!}\Big|\mathbb{E}[ ({\bf u}\cdot {\bf R}^{1}_{T})^{n}] - \mathbb{E}[ ({\bf u}\cdot \tilde{{\bf Z}}_{N})^{n}]\Big|
+\exp\Big(-\frac{1}{2} \frac{\log\log T}{\log\log\log T}\Big).
\end{split}
\end{align*}
Lastly, by \eqref{main estimate in lemma 7.9_1' first final}, \eqref{main estimate in lemma 7.9_1' second first}, \eqref{main estimate in lemma 7.9_1 second fifth}, and \eqref{main estimate in lemma 7.9_1' second second}, we complete the proof.
\end{proof}

%

Now, we shall state and prove the counterpart of Proposition \ref{Prop 6 in AR24} in the present setting. In addition, the reason for having a restriction on $N$ will be clear through the proof.
\begin{proposition}\label{better-Prop6}
Let $N = 1, 2, 3$. Suppose that $\Delta(T) = 0 = \delta(T)$ for all $T$. Let ${\bf R}^{1}_{T}$ and $\tilde{{\bf Z}}_{N}$ be the same as in Lemma \ref{better-lemma 9}. The for sufficiently large $T$, we have
\begin{align*}
d_{\mathcal{D}}({\bf R}^{1}_{T}, \tilde{{\bf Z}}_{N})
\ll_{{\mathfrak{K}}, N } \frac{L(\log\log\log T)}{\sqrt{\log\log T}} + \frac{M}{(\log\log T)^{2-N/2-\varepsilon_{4}}}
\end{align*}
where $\varepsilon_{4}>0$ can be arbitrarily small. 
\end{proposition}

\begin{proof}
Similar to Proposition \ref{Prop 6 in AR24}, the present proof is based on the results established in Lemmata \ref{better-lemma 9} and \ref{lemma 8 in AR24}. 
Let $\mu$ be the measure induced by ${\bf R}^{1}_{T}$ and $\nu$ the measure induced by $\tilde{\bf Z}_{N}$. Processing the proof as in \eqref{1st-middle-work-in-Prop6} and \eqref{2nd-middle-work-in-Prop6}, we can take $R = \log\log\log T$ in \eqref{2nd-middle-work-in-Prop6} to obtain
\begin{align}\label{R' part}
\mu(([-R, R]^{N})^{c}) + \nu(([-R, R]^{N})^{c})
\ll N \exp((-1/2)(\log\log\log T)^{2})
\end{align}
for sufficiently large $T$.

On the other hand, by Lemma \ref{better-lemma 9}, taking $F = \|{\bf u}\|_{1}$ as in \eqref{the growth of u}, we have that
\begin{align}\label{R'F' part}
\begin{split}
(RF)^{N}\|(\hat{\mu} - \hat{\nu})\mathds{1}_{(-F, F)^{N}}\|_{\infty}
&=\exp\Big( N\Big( \log\frac{\sqrt{1/2}}{C_{1}(T)\mathcal{C}_{2}}+ \frac{1}{2}\log\log\log T\Big)- 2\log\log\log T\Big)
\end{split}
\end{align}
for sufficiently large $T$, where $\varepsilon_{4}>0$ can be arbitrarily small. Plugging\eqref{R' part} and \eqref{R'F' part} into the right of \eqref{Goal in Prop 6 in AR24}, we have
\begin{align}\label{final-same-bd}
\ll_{N} \frac{L(\log\log\log T)}{\sqrt{\log\log T}} + M\Big(\exp((N/2 - 2 + \varepsilon_{4})\log\log\log T)   + N \exp((-1/2)(\log\log\log T)^{2})\Big),
\end{align}
which yields the claimed estimates.
\end{proof}

With Proposition \ref{better-Prop6} in hand, we shall prove Theorem \ref{indep-rate}.
\begin{proof}[Proof of Theorem \ref{indep-rate}]
Let ${\bf X}_{T}$ be the same as in \eqref{X_{T}-Tsang}, and let $\tilde{\bf X}$ be defined as in \eqref{normal X} with $\mathfrak{K} = \I_{N}$. Also, suppose that 
${\bf X}^{0}_{T}, {\bf M}_{T},{\bf Q}_{T}, {\bf R}_{T}, {\bf R}^{1}_{T}$, and $\tilde{{\bf Z}}_{N, T}$ are defined correspondingly as in Propositions \ref{Prop 1 in AR24}-\ref{Prop 5 in AR24}, Proposition \ref{better-Prop6}, and Proposition \ref{Prop 7} in the present setting. Then, by the triangle inequality, we have
\begin{align*}
d_{\mathcal{D}}({\bf X}_{T}, \tilde{{\bf X}})
&\ll d_{\mathcal{D}}({\bf X}_{T}, {\bf X}^{0}_{T}) 
+ d_{\mathcal{D}}({\bf X}^{0}_{T}, {\bf M}_{T})
+ d_{\mathcal{D}}({\bf M}_{T}, {\bf Q}_{T})
+ d_{\mathcal{D}}({\bf Q}_{T}, {\bf R}_{T})
+ d_{\mathcal{D}}({\bf R}_{T}, {\bf R}^{1}_{T})
+ d_{\mathcal{D}}({\bf R}^{1}_{T}, \tilde{{\bf Z}}_{T})
+ d_{\mathcal{D}}(\tilde{\bf Z}_{T}, \tilde{\bf X})
\\& \ll
\frac{LN(\log\log\log T)^{2}}{\sqrt{\log\log T}}
+\Big(\frac{LN}{\sqrt{\log\log T}}+MN(\log\log T)^{-K/K'}\Big)
 +N(L + M)(\log\log T)^{-80}
\\&+\frac{LN}{\sqrt{\log\log T}}
+L\sqrt{N} \frac{ \sqrt{ 1+ \log\log\log T}}{\sqrt{\log\log T}}
+\frac{L(\log\log\log T)}{\sqrt{\log\log T}} + \frac{M}{(\log\log T)^{2-N/2-\varepsilon_{4}}}
\\&+M (\log\log T)^{-1+\varepsilon +\varepsilon_{3}}.
\end{align*}
For $\frac{K}{K'}>\frac{3}{2}$, $MN(\log\log T)^{-K/K'}$  vanishes faster than $ \frac{M}{(\log\log T)^{2-N/2-\varepsilon_{4}}}$ for any possible $N$. Moreover, when $N = 2$, we take $\varepsilon_{4} = \varepsilon + \varepsilon_{3}$ so that $\frac{M}{(\log\log T)^{2-N/2-\varepsilon_{4}}} = M (\log\log T)^{-1+\varepsilon +\varepsilon_{3}}$, which concludes the proof.
\end{proof}

\section*{Acknowledgments}
The authors thank Guan-Yu Chen, Nathan Ng, Jia-Han Shih, and Guan-Ru Yu for helpful discussions and suggestions.

\end{document}